%% file: main.tex
\input{preamble}

\title{Quantum Graph Theory by Example}

\author[1,2]{Gian Luca Spitzer}
\email{gian-luca.spitzer@u-bordeaux.fr}
\orcid{0009-0008-0270-506X}
\affil[1]{LaBRI, Universit\'e de Bordeaux, CNRS, Bordeaux INP, UMR-5800, France}

\author[2]{Ion Nechita}
\email{nechita@irsamc.ups-tlse.fr}
\affil[2]{Laboratoire de Physique Th\'eorique, Universit\'e de Toulouse, CNRS, UPS, France}
\orcid{0000-0003-3016-7795}

\begin{document}

\maketitle

\begin{abstract}
  Quantum graphs have been introduced by Duan, Severini, and Winter to describe the zero-error behaviour of quantum channels. Since then, quantum graph theory has become a field of study in its own right. A substantial source of difficulty in working with quantum graphs compared to classical graphs stems from the fact that they are no longer discrete objects. This makes it generally difficult to construct insightful, non-trivial examples. We present a collection of non-trivial quantum graphs that can be thought of in discrete terms, and that can be expressed in the diagrammatic formalism introduced by Musto, Reutter, and Verdon. The examples arise as the quantum graphs acted on by increasingly smaller classical matrix groups, and are parametrised by triples of matrices $(A, B, C)$. The parametrisation reveals a clean decomposition of quantum graph structure into classical and genuinely quantum components: $A$ and $C$ are described by a classical weighted graph called the strange graph, while $B$ provides a purely quantum contribution with no classical analogue. Based on this model, we give exact formulas or establish bounds for quantum graph parameters, such as the number of connected components, the chromatic number, the independence number, and the clique number. Our results provide the first large, parametric families of quantum graphs for which standard graph parameters can be computed analytically.
\end{abstract}

\tableofcontents

\newpage

\section{Introduction}\label{sec:introduction}
\input{content/introduction}

\section{Preliminaries}\label{sec:preliminaries}
\input{content/prelims}

\section{Classical Matrix Groups Acting on Quantum Graphs}\label{sec:groupactions}
\input{content/groupactions}

\section{Graph Theoretic Properties}\label{sec:graphtheory}
\input{content/graphtheory}

\section{Conclusion}\label{sec:conclusion}
\input{content/conclusion}

\bigskip

\noindent\textbf{Acknowledgements.} The authors were supported by the ANR projects \href{https://esquisses.math.cnrs.fr/}{ESQuisses} grant number ANR-20-CE47-0014-01 and \href{https://www.ceremade.dauphine.fr/dokuwiki/anr-tagada:start}{TAGADA} grant number ANR-25-CE40-5672. The first author moreover acknowledges support from the CNRS 80Prime grant ``QuantGraphe''.

\medskip

\noindent\textbf{Author contributions.} The authors contributed equally to the research. GLS did the majority of the writing. Both authors reviewed and approved the final version.

\printbibliography

\end{document}

%% file: preamble.tex
\documentclass[a4paper,onecolumn,accepted=2026-09-11]{quantumarticle}

\pdfoutput=1

\usepackage[table,dvipsnames]{xcolor}

\usepackage[margin=2.5cm]{geometry}
\usepackage[utf8]{inputenc}
\usepackage[T1]{fontenc}
\usepackage{scrbase}
\usepackage{microtype}
\usepackage[shortlabels]{enumitem}
\usepackage{etoolbox}
\usepackage{graphicx}
\usepackage{tabularx}
\usepackage{multirow}
\usepackage{makecell}
\usepackage{pdflscape}
\usepackage{authblk}

\usepackage[
backend=biber,
style=alphabetic,
maxnames=99,
backref=true
]{biblatex}
\usepackage{amsmath}
\usepackage{amssymb}
\usepackage{amsthm}
\usepackage{thmtools}
\usepackage{mathtools}
\usepackage{braket}

\usepackage{array}
\newcolumntype{Y}{>{\centering\arraybackslash}X}
\definecolor{betterblue}{RGB}{0, 110, 230}
\definecolor{bettergreen}{RGB}{164, 180, 43}

\usepackage[most]{tcolorbox}
\newtcolorbox{mybox}{
    colback=orange!10, colframe=orange!40, arc=4mm, boxrule=1pt,
    left=6pt, right=6pt, top=6pt, bottom=6pt,
    width=0.95\textwidth, center
}

\usepackage{hyperref}
\hypersetup{%
	pdftoolbar=false,
	pdfmenubar=false,
	colorlinks=true,
	urlcolor={OliveGreen},
	linkcolor={betterblue},
	citecolor={orange}
}
\usepackage[capitalise,noabbrev]{cleveref}

\theoremstyle{plain}
\newtheorem{theorem}{Theorem}[section]
\newtheorem{lemma}[theorem]{Lemma}
\newtheorem{proposition}[theorem]{Proposition}
\newtheorem{corollary}[theorem]{Corollary}

\newtheorem{observation}[theorem]{Observation}

\newtheorem*{question*}{Question}

\newtheorem*{conjecture*}{Conjecture}
\newtheorem{claim}{Claim}
\AtEndEnvironment{proposition}{\setcounter{claim}{0}}
\AtEndEnvironment{lemma}{\setcounter{claim}{0}}
\AtEndEnvironment{theorem}{\setcounter{claim}{0}}

\theoremstyle{definition}
\newtheorem{definition}[theorem]{Definition}

\newtheorem{remark}[theorem]{Remark}
\newtheorem{examplex}[theorem]{Example}
\newenvironment{example}
  {\pushQED{\qed}\examplex}
  {\popQED\endexamplex}

\newenvironment{subproof}[1][\proofname]{%
  \begin{proof}[#1]%
}{%
  \end{proof}%
}

\renewcommand{\epsilon}{\varepsilon}

\newcommand{\C}{\ensuremath{\mathbb{C}}}
\newcommand{\R}{\ensuremath{\mathbb{R}}}
\newcommand{\N}{\ensuremath{\mathbb{N}}}

\DeclareMathOperator{\GL}{GL}
\DeclareMathOperator{\fdHilb}{fdHilb}

\DeclareMathOperator{\End}{End}

\DeclareMathOperator{\id}{id}
\DeclareMathOperator{\Tr}{Tr}
\DeclareMathOperator{\diag}{diag}
\DeclareMathOperator{\lspan}{span}

\DeclareMathOperator{\img}{img}

\DeclareMathOperator{\rk}{rk}

\newcommand{\transpose}{^\mathsf{T}}
\newcommand{\adjoint}{^\dagger}

\newcommand{\identity}{I}
\newcommand{\compl}[1]{#1^c}

\newcommand{\1}{\mathbf 1}
\newcommand{\IFF}{\Longleftrightarrow}

\newcommand{\Implies}{\Longrightarrow}
\renewcommand{\implies}{\Rightarrow}
\newcommand{\abs}[1]{\lvert #1 \rvert}
\newcommand{\norm}[1]{\lVert #1 \rVert}
\newcommand{\Cstar}{C\ensuremath{^*}}

\newcommand{\Gsym}{G^\mathrm{sym}}
\newcommand{\Gasym}{G^\mathrm{asym}}

\newcommand{\ccdot}{\makebox[1ex]{\textbf{$\cdot$}}}

\newcommand{\lrangle}[1]{\langle #1 \rangle}

\makeatletter
\let\amsmath@bigm\bigm

\renewcommand{\bigm}[1]{%
  \ifcsname fenced@\string#1\endcsname
    \expandafter\@firstoftwo
  \else
    \expandafter\@secondoftwo
  \fi
  {\expandafter\amsmath@bigm\csname fenced@\string#1\endcsname}%
  {\amsmath@bigm#1}%
}

\newcommand{\DeclareFence}[2]{\@namedef{fenced@\string#1}{#2}}
\makeatother

\DeclareFence{\mid}{|} % dynamic-size \mid

\makeatletter
\DeclareRobustCommand\ket[1]{%
  \@ifnextchar\bra{\k@t{#1}\!}{\k@t{#1}}%
}
\newcommand\k@t[1]{{|{#1}\rangle}}
\makeatother

\makeatletter
\newcommand{\vast}{\bBigg@{3}}
\newcommand{\Vast}{\bBigg@{4}}
\newcommand{\VVast}{\bBigg@{5}}
\newcommand{\VAST}{\bBigg@{6}}
\makeatother

\usepackage{arydshln}

\newcommand{\pad}[2]{\hspace{#1}#2\hspace{#1}}

\input{./style/my_tikzit.sty}
\input{./style/diagrams.tikzstyles}
\input{./style/tikz_adjustments.sty}

\input{./style/shapes.tex}

%% file: style/diagrams.tikzstyles
\tikzstyle{myBlue}=[draw={rgb,255: red,25; green,132; blue,232}]
\tikzstyle{myBlueFill}=[fill={rgb,255: red,25; green,132; blue,232}]
\tikzstyle{myGreen}=[draw={rgb,255: red,164; green,180; blue,43}]
\tikzstyle{myPink}=[draw={rgb,255: red,208; green,12; blue,133}]
\tikzstyle{myPinkFill}=[fill={rgb,255: red,208; green,12; blue,133}]
\tikzstyle{myYellow}=[draw={rgb,255: red,255; green,160; blue,0}]
\tikzstyle{myYellowFill}=[fill={rgb,255: red,255; green,160; blue,0}]
\tikzstyle{myBrown}=[draw={rgb,255: red,173; green,100; blue,82}]
\tikzstyle{myBrownFill}=[fill={rgb,255: red,173; green,100; blue,82}]
\tikzstyle{myOrange}=[draw={rgb,255: red,245; green,124; blue,0}]
\tikzstyle{myOrangeFill}=[fill={rgb,255: red,245; green,124; blue,0}]
\tikzstyle{myPurple}=[draw={rgb,255: red,178; green,121; blue,219}]
\tikzstyle{myPurpleFill}=[fill={rgb,255: red,178; green,121; blue,219}]
\tikzstyle{myPlum}=[draw={rgb,255: red,123; green,31; blue,162}]
\tikzstyle{myPlumFill}=[fill={rgb,255: red,123; green,31; blue,162}]
\tikzstyle{myTeal}=[draw={rgb,255: red,0; green,151; blue,167}]
\tikzstyle{map}=[myBlue, fill=white, shape=circle, thick, tikzit category=maps, minimum size=20pt, inner sep=1pt, tikzit draw={rgb,255: red,25; green,132; blue,232}]
\tikzstyle{wmap}=[myBlue, fill=white, shape=rectangle, rounded corners=10pt, thick, tikzit category=maps, minimum size=20pt, inner sep=1pt, tikzit draw={rgb,255: red,25; green,132; blue,232}]
\tikzstyle{correlation}=[map, myPink, tikzit draw={rgb,255: red,208; green,12; blue,133}, tikzit category=games]
\tikzstyle{game}=[map, myGreen, tikzit draw={rgb,255: red,104; green,159; blue,56}, tikzit category=games]
\tikzstyle{q-func}=[map, myPink, tikzit draw={rgb,255: red,208; green,12; blue,133}, tikzit category=games]
\tikzstyle{frob}=[fill=white, tikzit draw=black, shape=circle, very thick, tikzit category=frobenius, minimum size=4pt, inner sep=0pt]
\tikzstyle{x-frob}=[frob, myYellow, tikzit draw={rgb,255: red,255; green,160; blue,0}, tikzit category=frobenius]
\tikzstyle{dark-x-frob}=[frob, myYellow, myYellowFill, tikzit fill={rgb,255: red,255; green,160; blue,0}, tikzit draw={rgb,255: red,255; green,160; blue,0}, tikzit category=frobenius]
\tikzstyle{y-frob}=[frob, myBrown, tikzit draw={rgb,255: red,173; green,100; blue,82}, tikzit category=frobenius]
\tikzstyle{xy-frob}=[frob, myOrange, ultra thick, tikzit draw={rgb,255: red,245; green,124; blue,0}, tikzit category=frobenius]
\tikzstyle{half-xy-proj}=[xy-frob, shape=rectangle, minimum size=6pt, tikzit draw={rgb,255: red,245; green,124; blue,0}, tikzit category=frobenius, tikzit fill=white, tikzit shape=rectangle]
\tikzstyle{half-xy-proj-flipped}=[half-xy-proj, myOrangeFill, tikzit draw={rgb,255: red,245; green,124; blue,0}, tikzit category=frobenius, tikzit fill={rgb,255: red,245; green,124; blue,0}, tikzit shape=rectangle]
\tikzstyle{twisted-xy-proj}=[myOrange, left color=white, right color={rgb,255: red,245; green,124; blue,0}, tikzit draw={rgb,255: red,245; green,124; blue,0}, shape=rectangle, thin, tikzit category=frobenius, minimum size=8pt, inner sep=0pt, tikzit fill=white, tikzit shape=rectangle]
\tikzstyle{twisted-xy-proj-flipped}=[myOrange, right color=white, left color={rgb,255: red,245; green,124; blue,0}, tikzit draw={rgb,255: red,245; green,124; blue,0}, shape=rectangle, thin, tikzit category=frobenius, minimum size=8pt, inner sep=0pt, tikzit fill={rgb,255: red,245; green,124; blue,0}, tikzit shape=rectangle]
\tikzstyle{a-frob}=[frob, myPurple, tikzit draw={rgb,255: red,178; green,121; blue,219}, tikzit category=frobenius]
\tikzstyle{dark-a-frob}=[frob, myPurple, myPurpleFill, tikzit fill={rgb,255: red,178; green,121; blue,219}, tikzit draw={rgb,255: red,178; green,121; blue,219}, tikzit category=frobenius]
\tikzstyle{b-frob}=[frob, myBlue, tikzit draw={rgb,255: red,25; green,118; blue,210}, tikzit category=frobenius]
\tikzstyle{ab-frob}=[frob, myPlum, ultra thick, tikzit draw={rgb,255: red,123; green,31; blue,162}, tikzit category=frobenius]
\tikzstyle{half-ab-proj}=[ab-frob, shape=rectangle, minimum size=6pt, tikzit draw={rgb,255: red,123; green,31; blue,162}, tikzit category=frobenius, tikzit shape=rectangle]
\tikzstyle{half-ab-proj-flipped}=[half-ab-proj, myPlumFill, tikzit category=frobenius, tikzit draw={rgb,255: red,123; green,31; blue,162}, tikzit fill={rgb,255: red,123; green,31; blue,162}, tikzit shape=rectangle]
\tikzstyle{twisted-ab-proj}=[myPlum, left color=white, right color={rgb,255: red,123; green,31; blue,162}, tikzit draw={rgb,255: red,123; green,31; blue,162}, shape=rectangle, thin, tikzit category=frobenius, minimum size=8pt, inner sep=0pt, tikzit fill=white, tikzit shape=rectangle]
\tikzstyle{twisted-ab-proj-flipped}=[myPlum, right color=white, left color={rgb,255: red,123; green,31; blue,162}, shape=rectangle, thin, tikzit category=frobenius, minimum size=8pt, inner sep=0pt, tikzit draw={rgb,255: red,123; green,31; blue,162}, tikzit fill={rgb,255: red,123; green,31; blue,162}, tikzit shape=rectangle]
\tikzstyle{trap}=[fill=white, myPink, shape=trapezium, thick, tikzit shape=rectangle, trapezium left angle=90, trapezium right angle=60, minimum height=16pt, minimum width=16pt, inner sep=2pt, draw={rgb,255: red,25; green,132; blue,232}, tikzit category=maps]
\tikzstyle{conj}=[trap, trapezium left angle=60, trapezium right angle=90, tikzit draw={rgb,255: red,25; green,132; blue,232}, tikzit shape=rectangle, tikzit category=maps]
\tikzstyle{dagger}=[conj, shape border rotate=180, tikzit draw={rgb,255: red,25; green,132; blue,232}, tikzit shape=rectangle, tikzit category=maps]
\tikzstyle{transpose}=[trap, shape border rotate=180, tikzit draw={rgb,255: red,25; green,132; blue,232}, tikzit shape=rectangle, tikzit category=maps]
\tikzstyle{bra}=[myTeal, fill=white, tikzit draw={rgb,255: red,0; green,151; blue,167}, regular polygon, regular polygon sides=3, tikzit category=maps, inner sep=0pt, minimum size=20pt]
\tikzstyle{bra_s}=[myTeal, fill=white, tikzit draw={rgb,255: red,0; green,151; blue,167}, regular polygon, regular polygon sides=3, tikzit category=maps, inner sep=-1.25pt, minimum size=20pt]
\tikzstyle{bra_xs}=[myTeal, fill=white, tikzit draw={rgb,255: red,0; green,151; blue,167}, regular polygon, regular polygon sides=3, tikzit category=maps, inner sep=-1.75pt, minimum size=20pt]
\tikzstyle{bra_xxs}=[myTeal, fill=white, tikzit draw={rgb,255: red,0; green,151; blue,167}, regular polygon, regular polygon sides=3, tikzit category=maps, inner sep=-3pt, minimum size=1pt]
\tikzstyle{ket}=[bra, shape border rotate=180, tikzit draw={rgb,255: red,0; green,151; blue,167}, tikzit category=maps]
\tikzstyle{ket_s}=[{bra_s}, shape border rotate=180, tikzit draw={rgb,255: red,0; green,151; blue,167}, tikzit category=maps]
\tikzstyle{ket_xs}=[{bra_xs}, shape border rotate=180, tikzit draw={rgb,255: red,0; green,151; blue,167}, tikzit category=maps]
\tikzstyle{ket_xxs}=[{bra_xxs}, shape border rotate=180, tikzit draw={rgb,255: red,0; green,151; blue,167}, tikzit category=maps]
\tikzstyle{graph_effect}=[bra, thick, myGreen, tikzit draw={rgb,255: red,104; green,159; blue,56}, tikzit category=games, tikzit shape=rectangle]
\tikzstyle{graph_state}=[ket, thick, myGreen, tikzit draw={rgb,255: red,104; green,159; blue,56}, tikzit category=games, tikzit shape=rectangle]
\tikzstyle{map_multi}=[myGreen, thick, shape=rectangle, tikzit draw={rgb,255: red,25; green,132; blue,232}, tikzit shape=rectangle, minimum width=1.5cm, tikzit category=maps]
\tikzstyle{small_map_multi}=[{map_multi}, minimum width=1cm, tikzit category=maps, shape=rectangle, tikzit draw={rgb,255: red,25; green,132; blue,232}]
\tikzstyle{graph_state_small}=[{graph_state}, tikzit category=maps, tikzit shape=rectangle, tikzit draw={rgb,255: red,104; green,159; blue,56}, inner sep=-20pt]
\tikzstyle{hilbertnoarrow}=[tikzit draw={rgb,255: red,0; green,151; blue,167}, myTeal, thick]
\tikzstyle{linmap}=[minimum width=.7cm, fill=white, draw=black, shape=rectangle, tikzit shape=rectangle, tikzit draw=black, tikzit category=linalg, tikzit fill=white, minimum height=.7cm]
\tikzstyle{graph_state_medium}=[{graph_state}, tikzit shape=rectangle, tikzit draw={rgb,255: red,104; green,159; blue,56}, inner sep=-2pt]
\tikzstyle{graph_effect_medium}=[{graph_effect}, inner sep=-2pt, tikzit draw={rgb,255: red,104; green,159; blue,56}, tikzit shape=rectangle]
\tikzstyle{graph_effect_small}=[{graph_effect}, thick, myGreen, inner sep=-20pt, tikzit draw={rgb,255: red,104; green,159; blue,56}, tikzit shape=rectangle]
\tikzstyle{proj}=[fill=white, shape=rectangle, tikzit category=maps, myGreen]
\tikzstyle{cket}=[bra, shape=cut triangle down, tikzit draw={rgb,255: red,0; green,128; blue,128}, apex angle=110, minimum width=.95cm]
\tikzstyle{cbra}=[cket, tikzit draw={rgb,255: red,0; green,128; blue,128}, shape=cut triangle up]
\tikzstyle{widecket}=[cket, minimum width=1.35cm, apex angle=100]
\tikzstyle{ccket}=[cket, shape=cut triangle down alt, tikzit draw={rgb,255: red,0; green,128; blue,128}, tikzit shape=circle]
\tikzstyle{ccbra}=[cbra, shape=cut triangle up alt, tikzit draw={rgb,255: red,0; green,128; blue,128}, tikzit shape=circle]

\tikzstyle{hilbert}=[-, myTeal, thick, tikzit draw={rgb,255: red,0; green,151; blue,167}]
\tikzstyle{hilbert2}=[-, myGreen, thick, tikzit draw={rgb,255: red,148; green,210; blue,75}]
\tikzstyle{x-wire}=[-, myYellow, very thick, tikzit draw={rgb,255: red,255; green,160; blue,0}]
\tikzstyle{y-wire}=[-, myBrown, very thick, tikzit draw={rgb,255: red,173; green,100; blue,82}]
\tikzstyle{xy-wire}=[-, myOrange, line width=2pt, tikzit draw={rgb,255: red,245; green,124; blue,0}]
\tikzstyle{a-wire}=[-, myPurple, very thick, tikzit draw={rgb,255: red,178; green,121; blue,219}]
\tikzstyle{b-wire}=[-, myBlue, very thick, tikzit draw={rgb,255: red,25; green,118; blue,210}]
\tikzstyle{ab-wire}=[-, myPlum, tikzit draw={rgb,255: red,123; green,31; blue,162}, line width=2pt]
\tikzstyle{hilbertnoarrow}=[-, tikzit draw={rgb,255: red,0; green,151; blue,167}, thick, myTeal]
\tikzstyle{dirwire}=[-]
\tikzstyle{wire}=[-]
\tikzstyle{dirwire02}=[-, tikzit draw=magenta]

%% file: content/introduction.tex
Graph theory is one of the most versatile branches of discrete mathematics, with applications ranging from computer science and combinatorial optimisation to statistical physics and network science. Many of its central notions such as colouring, connectivity, independence, cliques, etc.~ have been studied for over two centuries \cite{diestel2025graph}. A natural question, motivated by the advent of quantum information theory, is whether these notions admit meaningful generalisations to a non-commutative setting, where the vertex set of a graph is replaced by a quantum set and the adjacency relation by a suitable operator-algebraic structure.

The origins of \emph{quantum graph theory} can be traced to the work of Weaver \cite{weaver2010quantumrelations} who introduced quantum graphs as special cases of \emph{quantum relations}, and the work of Duan, Severini, and Winter \cite{duan_zeroerror_2013}, who introduced the \emph{confusability graph} of a quantum channel as a non-commutative generalisation of Shannon's classical confusability graph \cite{shannon2003zero}. The former approach introduces a notion of relations on von Neumann algebras, defining a quantum graph as a reflexive, symmetric quantum relation. The latter approach considers operator spaces of the form $\lspan\{K_i^* K_j \mid i, j \in [m]\}$, where $K_1, \dots, K_m$ are the Kraus operators of a quantum channel. Such an operator space encodes the channel's confusability structure and plays the role of the edge set of a graph. This perspective opened the door to extending classical zero-error information theory to the quantum setting \cite{stahlke2015quantum}, and led to the discovery of striking phenomena such as superactivation of zero-error capacity \cite{duan_superactivation_2009, cubitt_superactivation_2011}.

A systematic mathematical framework for quantum graphs was subsequently developed from two complementary perspectives. On the operator-algebraic side, Weaver \cite{weaver2021quantum} expanded on the theory of quantum relations on von Neumann algebras. This approach, rooted in operator space theory, naturally extends classical notions of independent sets and cliques \cite{weaver2017quantum} and has led to quantum analogues of Ramsey's theorem. On the categorical side, Musto, Reutter, and Verdon \cite{musto_compositional_2018} proposed a compositional approach in which a quantum graph is defined as an \emph{adjacency operator} on a quantum set, the latter being formalised as a special symmetric $\dagger$-Frobenius monoid in the category of finite-dimensional Hilbert spaces. This formulation comes equipped with a powerful \emph{string diagrammatic calculus} \cite{heunen_categories_2019, vicary_categorical_2011} that makes many proofs visual and intuitive. In the finite-dimensional case, which is the setting of this paper, both approaches are equivalent.

These foundational works have initiated a rapidly growing body of research; we refer the reader to the survey by Daws \cite{daws2024quantum} for a comprehensive account of the different perspectives on quantum graphs and their interrelations. Graph-theoretic properties have been generalised to the quantum setting: colouring and chromatic numbers via non-local games \cite{brannan_quantumclassical_2022, ganesan_spectral_2023}, capacity bounds and Lov\'asz-type parameters via sandwich theorems \cite{boreland2021sandwich}, connectedness and algebraic connectivity via spectral and operator-algebraic methods \cite{courtney2025connectivity, matsuda_algebraic_2024}, and graph homomorphisms and isomorphisms via categorical and game-theoretic frameworks \cite{atserias_quantum_2019, brannan2020bigalois, goldberg_quantum_2024}. Classical graph constructions have also been ported to the quantum setting, including the Mycielski transformation \cite{bochniak2025quantum}. At the same time, the question of classifying quantum graphs and constructing explicit examples has received increasing attention. Gromada \cite{gromada_examples_2022} classified all quantum graphs on $M_2$, finding exactly four non-isomorphic simple quantum graphs, and further constructed examples from Hadamard matrices \cite{gromada2024quantum}; see also the related work of Matsuda \cite{matsuda2022classification}.

Despite this progress, the theory has until now lacked a supply of concrete, parametric families of quantum graphs for which graph-theoretic parameters can be computed analytically. Classical graph theory benefits enormously from explicit families (circulant graphs, Kneser graphs, Paley graphs, etc.) that serve as examples, counterexamples, and testing grounds for conjectures. Recent work of Wasilewski \cite{wasilewski2024quantum} on quantum Cayley graphs is a step in this direction. The present paper aims to fill this gap more broadly for quantum graphs.

\bigskip

The contributions of this work are twofold: We introduce and then thoroughly study new families of quantum graphs on $M_n$ that are invariant under the action of some classical matrix groups.

\begin{mybox}
    On the conceptual side, we introduce and study families of quantum graphs on $M_n$ that are invariant under the action of classical matrix groups, such as the unitary group, the orthogonal group, and their respective diagonal subgroups.
\end{mybox}

The guiding principle is the same as in classical graph theory: Just as considering smaller subgroups of the symmetric group $S_n$ yields larger and richer families of classical graphs, considering smaller subgroups of the unitary group $U(n)$ yields larger families of quantum graphs. Apart from the unitary group itself, we look at the orthogonal group $O(n)$ and the hyperoctahedral group $\operatorname{Hyp}(n)$, as well as the diagonal fragments of $U(n)$ and $O(n)$: the \emph{diagonal unitary group} $DU(n)$ and the \emph{diagonal orthogonal group} $DO(n)$. These satisfy the following inclusions.
\begin{equation*}
  U(n) \supseteq O(n) \supseteq \operatorname{Hyp}(n) \supseteq DO(n) \subseteq DU(n).
\end{equation*}
Requiring invariance under the full unitary group $U(n)$ yields only the complete quantum graph $K_n$ and the edgeless quantum graph $\compl{K}_n$, while the orthogonal group $O(n)$ additionally permits the \emph{symmetric} and \emph{antisymmetric quantum graphs} $\Gsym$ and $\Gasym$ (\cref{prop:qgraphs-as-xabc}). These are already genuinely quantum objects with no classical counterpart.

The diagonal unitary group $DU(n)$ and the diagonal orthogonal group $DO(n)$ are considerably smaller, and their invariant quantum graphs form the main object of study of this paper. Building on the characterisation of $DU(n)$- and $DO(n)$-invariant linear maps \cite{singh2021diagonal,nechita2021graphical}, we show that these quantum graphs are parametrised by three $n \times n$ matrices $A$, $B$, $C$ with matching diagonals, giving rise to quantum graphs denoted $X_{A, B, C}$ (\cref{prop:xabc-qgraph-char}). The conditions for $X_{A, B, C}$ to be a quantum graph decompose cleanly: $B$ must be a projector, and for all $i \neq j$ the $2 \times 2$ blocks
\begin{equation*}
  \begin{pmatrix}
    A_{ij} & C_{ij} \\
    C_{ji} & A_{ji}
  \end{pmatrix}
\end{equation*}
must be projectors. Each of the three matrices plays a very distinct role:
\begin{itemize}
  \item The matrix $A$ encodes a \emph{classical graph} on $n$ vertices. When $B$ and $C$ are trivial, the quantum graph $X_{A, \diag A}$ is precisely the embedding of a classical graph into $M_n$ (\cref{xa-qgraph-char}).
  \item The matrix $C$ introduces a new type of edges, which we call \emph{strange edges}, annotated with a phase $\theta \in [0, 2\pi)$. Together, $A$ and $C$ determine a classical model called the \emph{strange graph} $\mathfrak{G}(A, C)$, which can be thought of as a classical graph whose edges are either classical or strange (see \cref{ssec:classicalmodels}).
  \item The matrix $B$ is a \emph{purely quantum contribution}: a projector on $\C^n$ with no classical analogue, which can be thought of as attaching a subspace to the quantum graph.
\end{itemize}
This decomposition captures an increasing level of ``quantumness'' as one moves from purely classical graphs $X_{A, \diag A}$ through AB graphs $X_{A, B}$ (which are $DU(n)$-invariant) to full ABC graphs $X_{A, B, C}$ having $DO(n)$ symmetry.

\medskip

Our second main contribution is the computation of graph-theoretic parameters for these newly introduced quantum graphs.

\begin{mybox}
    On the concrete side, we compute graph-theoretic parameters for the newly introduced families of quantum graphs, such as connectedness, colouring, independent sets, and cliques.
\end{mybox}
A key structural observation underlying our results is a recurring \emph{splitting principle}: For each of the graph parameters we consider, the defining conditions decompose into independent conditions on the $(A, C)$ part and the $B$ part (\cref{prop:conncompindep} and analogues for colouring, independent sets, and cliques). This effectively reduces problems about quantum graphs to a combination of a classical graph problem determined by the strange graph $\mathfrak{G}(A, C)$, and a linear-algebraic problem involving the projector $B$. We highlight five of our most important results:
\begin{enumerate}
  \item \textbf{Connectedness.} For $n \geq 3$, a quantum graph $X_{A, B, C}$ is connected if and only if its strange graph $\mathfrak{G}(A, C)$ is connected (\cref{prop:connectedequivn3}). This equivalence can fail for $n = 2$: Isolated strange edges with phase $\pi$ give rise to quantum graphs with strictly more connected components than their strange graph (\cref{prop:conncompcounterex}).

  \item \textbf{Independence number.} The independence number of quantum graphs of the form $X_{A, \diag A, C}$ is completely determined by the strange graph: $\alpha(X_{A, \diag A, C}) = \alpha(\mathfrak{G}(A, C))$ (\cref{prop:indepnumdetbystrange}). For the purely quantum part, the independence number of $X_{\diag B, B}$ can be bounded in both directions in terms of the rank of $B$. It is also lower bounded by $\operatorname{EqRows}(B)$, the maximum number of equal rows of $B$ in the standard basis.

  \item \textbf{Non-colourability.} There exist quantum graphs that are not classically colourable at all (\cref{prop:xabnotcolourable}); this includes the complete quantum graph $K_n$ and the symmetric quantum graph $\Gsym$ for $n \geq 3$. In contrast, quantum graphs of the form $X_{\diag B, B}$ are always $n$-colourable (\cref{prop:xb-ncolourable}), showing that the obstruction to colourability arises from the interaction between $A$, $C$, and $B$.

  \item \textbf{Clique number.} The quantum clique number of the quantum graph induced by a classical graph $\omega(X_{A, \diag A})$ can differ dramatically from the classical clique number $\omega(A)$, in both directions. While $\omega(X_{A, \diag A}) \geq \omega(A) - 1$ always holds, the classical complete bipartite graph on $n$ vertices satisfies $\omega(A) = 2$ but $\omega(X_{A, \diag A}) \geq n/2$ (\cref{prop:cliquenumclassbipart}). On the other hand, the classical complete graph satisfies $\omega(X_{A, \diag A}) = n - 1 < n = \omega(A)$ (\cref{prop:classcomplomega}). The symmetric and antisymmetric quantum graphs both have clique number $\lceil n / 2 \rceil$.

  \item \textbf{The role of $B$.} The matrix $B$ contributes a form of quantum density that interacts non-trivially with the classical structure encoded by $A$ and $C$. For example, $\omega(X_{\diag B, B}) \leq \sqrt{\rk B + 1}$ (\cref{prop:clique-bound-XB}), while $X_{\diag B, B}$ always has $n$ connected components and is $n$-colourable.
\end{enumerate}

Most of the results presented in this paper are summarized in \cref{fig:qgraphexamples-with-props}, which classifies quantum graphs according to their underlying classical symmetry and gathers all the exact formulas and bounds for their graph properties.

\bigskip

The paper is organized as follows. \Cref{sec:preliminaries} introduces the necessary background on quantum sets, quantum graphs, and their diagrammatic calculus, as well as the classical matrix groups and graph-theoretic properties studied throughout the paper. \Cref{sec:groupactions} develops the theory of quantum graphs invariant under classical matrix groups, culminating in the parametrisation of $DU(n)$- and $DO(n)$-invariant quantum graphs by triples of matrices $(A, B, C)$. \Cref{sec:graphtheory} computes graph-theoretic parameters for the families of quantum graphs introduced in the previous section: connected components, chromatic number, independence number, and clique number. Finally, \cref{sec:conclusion} summarises our findings and discusses open problems.

%% file: content/prelims.tex
First, let us fix some general conventions. We assume all projectors to be orthogonal projectors: A linear map $P$ on a Hilbert space is a projector if $P^2 = P\adjoint = P$. We use $\identity$ to denote different identities. It may refer to the identity linear map on some arbitrary vector space, or explicitly to the identity matrix in $M_n$. It will always be clear from the context which is meant. When talking about classical graphs, we assume the vertex set to be $[n] := \{1, \dots, n\}$. We write $i \sim j$ if the vertices $i$ and $j$ are adjacent.

\subsection{Quantum Sets and String Diagrams}
We will view quantum graphs as adjacency operators on a quantum set of vertices. As can be motivated through Gelfand duality of sets (viewed as finite spaces with the discrete topology), a \emph{quantum set} is a finite-dimensional \Cstar-algebra. It is known \cite[Theorem 4.7]{vicary_categorical_2011} that the category of \Cstar-algebras is equal to the category of \emph{special symmetric $\dagger$-Frobenius monoids} in $\fdHilb$.

\begin{definition}
  A \emph{monoid} $X$ in $\fdHilb$ is a finite-dimensional Hilbert space equipped with a linear map $m \colon X \otimes X \to X$, called \emph{multiplication}, and a linear map $u\colon \C \to X$, called \emph{unit}, satisfying
  \begin{enumerate}
    \item $m \circ (\id \otimes m) = m \circ (m \otimes \id)$ (associativity),
    \item $m \circ (\id \otimes u) = m \circ (u \otimes \id) = \id$ (unitality).
  \end{enumerate}
\end{definition}

Now the category $\fdHilb$ is canonically equipped with a contravariant endofunctor $\dagger$ that is the identity on objects, and maps a linear map $f\colon X \to Y$ between Hilbert spaces $X, Y$ to its adjoint $f\adjoint \colon Y \to X$. In particular, a monoid in $\fdHilb$ is automatically equipped with a \emph{comultiplication} $m\adjoint \colon X \to X \otimes X$ and a \emph{counit} $u\adjoint \colon X \to \C$. They satisfy the adjoints of conditions (1) and (2), called coassociativity and counitality.

\begin{definition}
  A \emph{$\dagger$-Frobenius monoid} $X$ is a monoid satisfying the \emph{Frobenius property},
  \begin{enumerate}
    \item $(\id \otimes m) \circ (m\adjoint \otimes \id) = m\adjoint \circ m = (m \otimes \id) \circ (\id \otimes m\adjoint)$.
  \end{enumerate}
  It is called
  \begin{enumerate}[resume]
    \item \emph{special} if $m \circ m\adjoint = \id$,
    \item \emph{symmetric} if $\sigma \circ m\adjoint \circ u = m\adjoint \circ u$,
  \end{enumerate}
  where $\sigma$ is the swap map.
\end{definition}

The advantage of working in a monoidal category like $\fdHilb$ is that we get access to a powerful graphical calculus that allows us to reason formally by manipulating \emph{string diagrams}. This approach of working with quantum sets and quantum graph was first introduced by Musto, Reutter, and Verdon in \cite{musto_compositional_2018}, where we also refer the reader for a more detailed introduction. Further exposition of the theory of string diagrams and their applications can be found in \cite{heunen_categories_2019,vicary_categorical_2011}.
The primitives of the calculus are \emph{strings} (or \emph{wires}) and \emph{boxes}.
\begin{equation*}
  \input{diagrams/string.tikz} \qquad\quad \input{diagrams/box.tikz}
\end{equation*}
Strings represent objects, while boxes represent morphisms. In our case the strings represent (elements of) Hilbert spaces, while boxes are linear maps. Composition $g \circ f$ of functions is represented by serial composition of diagrams, while tensor products $f \otimes h$ are represented by parallel composition. Diagrams are read from bottom to top.
\begin{equation*}
  \input{diagrams/composition.tikz} \qquad\qquad \input{diagrams/tensorprod.tikz}
\end{equation*}
The monoidal unit, $\C$ in the case of $\fdHilb$, is represented by the empty diagram. Linear maps $\C \to X$ are thus represented by a box that only has an outgoing wire. For reasons that will become apparent later, we generally represent such maps by a triangle with a corner cut off.
\begin{equation}\label{eq:ctox}
  \input{diagrams/ctox.tikz}
\end{equation}
The multiplication and unit maps of monoids are represented as fusion of wires and dots respectively.
\begin{equation*}
  \input{diagrams/mult.tikz} \qquad\qquad \input{diagrams/unit.tikz}
\end{equation*}
Taking adjoints in the diagrammatic calculus corresponds to mirroring the diagram across a horizontal axis. As such, the comultiplication and counit are represented as follows.
\begin{equation*}
  \input{diagrams/comult.tikz} \qquad\qquad \input{diagrams/counit.tikz}
\end{equation*}
The Frobenius property, for instance, can thus be expressed graphically as
\begin{equation*}
  \input{diagrams/frob-cond.tikz}\enspace.
\end{equation*}
A final important class of maps are the cups and caps.
\begin{equation*}
  \input{diagrams/cap.tikz} \qquad\qquad \input{diagrams/cup.tikz}
\end{equation*}
These are the duality morphisms for $X$, and their definition in terms of (co)multiplication and (co)unit shows that $\dagger$-Frobenius monoids are self-dual.\footnote{Otherwise, we would have to annotate our strings with arrows, as is done in \cite{vicary_categorical_2011}.} For our purposes, it suffices to know that they allow us to define the transpose of a linear map.

\begin{definition}\label{def:transpose}
  Let $f\colon X \to Y$ be a linear map between two $\dagger$-Frobenius monoids. Its \emph{transpose} is the map $f\transpose \colon Y \to X$ defined as
  \begin{equation}\label{eq:transpose}
    \input{diagrams/transpose.tikz}\enspace.
  \end{equation}
\end{definition}

The power of string diagrams is that they can be used in lieu of symbolic equations to prove equalities. Concretely, every equation that can be derived in the diagrammatic calculus by moving and bending strings, sliding boxes along wires, and rotating boxes also holds in $\fdHilb$. This can be stated formally.

\begin{theorem}[Theorem 3.28 in \cite{heunen_categories_2019}]
  In a pivotal category, a well-formed equation between morphisms follows from the axioms if and only if it holds in the graphical calculus up to planar oriented isotopy.
\end{theorem}

A pivotal category is a category with duals in which every object is naturally isomorphic to its double dual. Clearly, $\fdHilb$ is a pivotal category, so the theorem applies. Theorems of this form are often called \emph{coherence theorems}, and they also exist for other types of categories, cf. \cite{selinger_survey_2010}. One such planar oriented isotopy would be to consider the right-hand side of \cref{eq:transpose} and pulling the bent wires taut. This will result in a 180 degree rotation of the $f$ box. We conclude that taking the transpose of a linear map corresponds to rotating the corresponding diagram by 180 degrees. By using non-symmetric boxes for maps, we may thus get rid of superscripts to represent transposes and adjoints. We let
\begin{equation*}
  \input{diagrams/nonsym-func.tikz}~, \qquad \input{diagrams/nonsym-func-transpose.tikz}~, \qquad \input{diagrams/nonsym-func-adjoint.tikz}~, \qquad \input{diagrams/nonsym-func-conj.tikz}~.
\end{equation*}
The last equation corresponds to the conjugate of a linear map, defined as $\overline{f} = (f\adjoint)\transpose = (f\transpose)\adjoint$. The corresponding operation in the diagrammatic calculus must thus be mirroring across a horizontal axis, followed by a 180 degree rotation (or vice versa). This corresponds to mirroring across a vertical axis. Despite being useful, this asymmetric notation has the tendency to introduce visual clutter. Since we will rarely need adjoints or transposes of maps $X \to X$, we reserve this asymmetric notation for maps $\C \to X$, see \cref{eq:ctox}.

\begin{example}[Classical Sets]
  By Gelfand duality, the classical set with $n$ elements corresponds to the quantum set $\C^n$. This becomes a special symmetric $\dagger$-Frobenius monoid as follows. We choose an orthonormal basis $\ket{1}, \dots, \ket{n}$, the \emph{standard basis}, and define $m$ as the linear extension of
  \begin{equation*}
    \C^n \otimes \C^n \to \C^n, \quad \ket{i} \otimes \ket{j} \mapsto \delta_{ij} \ket{i}.
  \end{equation*}
  In other words, the multiplication on $\C^n$ is componentwise. The unit with respect to this multiplication must then be given by
  \begin{equation*}
    \C \to \C^n, \quad c \mapsto c \mathbf{1} ,
  \end{equation*}
  where
  \begin{equation*}
    \mathbf{1} \coloneqq \sum_{k=1}^n \ket{k}.
  \end{equation*}
  From this, the comultiplication can be determined to be the linear extension of
  \begin{equation*}
    \C^n \to \C^n \otimes \C^n, \quad \ket{i} \mapsto \ket{i} \otimes \ket{i},
  \end{equation*}
  while the counit is the sum-of-entries map $\C^n \to \C$. Note that the comultiplication only copies elements of the standard basis; this is the \emph{no-cloning theorem}. Equivalently, we could have started with the comultiplication and defined the standard basis as precisely the elements that are copied by the comultiplication, cf. \cite[Theorem 5.1]{coecke_new_2013}.
  It is not hard to verify using the above definitions that standard basis elements must be self-conjugate, $\overline{\ket{i}} = (\ket{i}\adjoint)\transpose = \ket{i}$. Recalling that taking the conjugate of a linear map corresponds to mirroring the diagram across a vertical axis, we want the boxes representing self-conjugate maps to have this symmetry. We thus represent standard basis elements as full triangles, as opposed to the cut triangles of \cref{eq:ctox}.
  \begin{equation*}
    \input{diagrams/stdbasis.tikz}
  \end{equation*}

  Finally, given some linear map $f\colon \C^n \to \C^n$, one may verify that the transpose as defined in \cref{def:transpose} is precisely given by the transpose of the matrix representation of $f$ in the standard basis.
\end{example}

In this paper, we will mainly be interested in one particular family of quantum sets, the \emph{full matrix algebras} $M_n(\C)$ for $n \geq 2$. These are obviously objects of $\fdHilb$ and they form a $\dagger$-Frobenius monoid by setting $u(c) = c\identity$, $m(X \otimes Y) = XY$.
%, and counit and comultiplication as their adjoints with respect to the Hilbert-Schmidt inner product.
They also admit a refinement of the string diagrammatic formalism. Namely, there is an isomorphism $M_n \cong \C^n \otimes \C^n$. This means that instead of representing $M_n$ as a single wire, we may represent it as two parallel $\C^n$ wires. The isomorphism is then given graphically as
\begin{equation*}
  M_n \to \C^n \otimes \C^n,\qquad \input{diagrams/mn-cncn-iso.tikz}\enspace.
\end{equation*}
This is sometimes called the ``vectorisation map'' and corresponds to stacking the columns of $x$ on top of each other. The multiplication and unit maps are then given by
\begin{equation*}
  \input{diagrams/mn-cncn-mult.tikz}, \qquad\qquad \input{diagrams/mn-cncn-unit.tikz}\enspace.
\end{equation*}
Concretely, this is the \emph{endomorphism monoid} of $\C^n$, cf. \cite[Definition 3.16]{vicary_categorical_2011}. Endomorphism monoids always satisfy the Frobenius property, and it is a good exercise to verify that $M_n$ is symmetric since $\C^n$ is. However, defined this way $M_n$ is not special. We have
\begin{equation*}
  \input{diagrams/comult-mult-in-endomonoid.tikz}\enspace.
\end{equation*}
This is closely related to so called $\delta$-forms, cf. e.g. \cite[Section 1]{gromada_examples_2022}. We can make $M_n$ special by normalising $u, u\adjoint, m, m\adjoint$. In fact, there are multiple ways to accomplish this. The normalisation most common in the context of string diagrams is
\begin{equation*}
  \tilde{m} = \frac{1}{\sqrt{n}} m, \qquad \tilde{u} = \sqrt{n} u, \qquad \tilde{m}\adjoint = \frac{1}{\sqrt{n}} m\adjoint, \qquad \tilde{u}\adjoint = \sqrt{n} u\adjoint.
\end{equation*}
This normalisation preserves the property that taking the adjoint corresponds to simply mirroring the diagram.

In what follows, we will generally suppress these normalisation factors. This is because the normalisation will lead to factors in unexpected places, at the cost of readability. Since we will mostly be working in the diagrammatic calculus, this will have almost no impact on our arguments. We only need to remember the normalisation when we want to compute scalars.

\subsection{Quantum Graphs}

There are many different equivalent ways of defining quantum graphs. We will be interested in two of them. Firstly, a quantum graph is an adjacency operator on a quantum set of vertices.

\begin{definition}[cf. Definition V.1 in \cite{musto_compositional_2018}]\label{def:qgraphsadj}
  A \emph{quantum graph} is a tuple $(X, G)$, where $X$ is a quantum set and $G\colon X \to X$ is a linear map satisfying
  \begin{equation*}
    \input{diagrams/adj-op-schur-idemp.tikz} \qquad\quad\text{and}\quad\qquad \input{diagrams/adj-op-real.tikz}
  \end{equation*}
\end{definition}

The two properties abstract the defining properties of adjacency matrices. In the classical case, where wires are $\C^n$, the first condition says that $G$ is idempotent under the entrywise (Schur) product. The second condition says that $G$ is self-conjugate, or in the classical case that its matrix representation in the standard basis has real entries.\footnote{Note that in the classical case, the second property is redundant, since being idempotent under the Hadamard product already requires the matrix entries in the standard basis to be in the set $\{0, 1\}$. However, adding the realness condition leads to a more coherent theory.} When the quantum set $X$ is clear from the context, we will also identify a quantum graph $G$ with its adjacency operator.

\begin{definition}
  A quantum graph $(X, G)$
  \begin{enumerate}
    \item is \emph{undirected} if
      \begin{equation*}
        \input{diagrams/adj-op-undirected.tikz}
      \end{equation*}
    \item has \emph{no loops} if
      \begin{equation*}
        \input{diagrams/adj-op-irreflexive.tikz}
      \end{equation*}
    \item has \emph{loops at every vertex} if
      \begin{equation*}
        \input{diagrams/adj-op-reflexive.tikz}~.
      \end{equation*}
  \end{enumerate}
\end{definition}

\begin{remark}
  Recall that we will generally suppress the normalisation of multiplication and comultiplication that makes the $\dagger$-Frobenius monoid $M_n$ special. This will have an impact on what constitutes an adjacency operator of a quantum graph. However, it is straightforward to recover the adjacency operators for special $\dagger$-Frobenius monoids: We simply multiply the adjacency operator by $n$. Indeed, let $A$ be a valid adjacency operator in the non-normalised setting, that is
  \begin{equation*}
    m \circ (A \otimes A) \circ m\adjoint = A.
  \end{equation*}
  Let $A' = nA$. The realness condition is invariant under multiplication by real scalars. Moreover, per definition we get
  \begin{equation*}
    \tilde{m} \circ (A' \otimes A') \circ \tilde{m}\adjoint = \frac{1}{n} m \circ (nA \otimes nA) \circ m\adjoint = nA = A',
  \end{equation*}
  so $A'$ is an adjacency operator in the normalised setting. The same reasoning applies to the undirectedness and loop conditions.
\end{remark}

The second important characterisation of quantum graphs will be in terms of operator spaces. From now on we will only consider quantum graphs on $M_n$. The isomorphism $M_n \cong \C^n \otimes \C^n$ allows us to define the \emph{realignment} map,
\begin{equation*}
  (-)^R \colon \End(M_n) \to \End(M_n), \quad \input{diagrams/realignment.tikz}\enspace,
\end{equation*}
which swaps the bottom-right and top-left tensor legs of a linear map $M_n \to M_n$. In the quantum information literature, the realignment map is used for entanglement detection \cite{rudolph2000separability,rudolph2003cross}, and the resulting map is called the \emph{Choi matrix} \cite{choi1975completely}; see also \cite[Section 10.2]{bengtsson2006geometry}. It is easy to see that the realignment is an involution, that is, $(F^R)^R = F$. We can thus give an equivalent definition of quantum graphs by taking the realignment of the conditions in \cref{def:qgraphsadj}.

\begin{proposition}\label{prop:adjoprealignedisproj}
  A linear map $G$ is the adjacency operator of a quantum graph if and only if $G^R$ is a projector.
  \begin{proof}
    Taking the realignment of the Schur idempotency condition and writing it in terms of $G^R$ yields
    \vspace{-.3cm}
    \begin{equation*}
      \input{diagrams/schur-idemp-realigned-lhs.tikz}\quad = \scalebox{.85}{\input{diagrams/schur-idemp-realigned.tikz}}\enspace = \quad\input{diagrams/schur-idemp-realigned-rhs.tikz}\enspace,
    \end{equation*}
    % \vspace{-.3cm}
    where the second equality follows from isomorphism of diagrams. The realness condition becomes
    \begin{equation*}
      \input{diagrams/real-realigned.tikz}\enspace.
    \end{equation*}
    These are precisely the conditions $(G^R)^2 = G^R$ and $(G^R)\adjoint = G^R$.
  \end{proof}
\end{proposition}

Every projector on $M_n$ uniquely defines a subspace $S \subseteq M_n$, so quantum graphs can equivalently be viewed as \emph{operator spaces} \cite{paulsen2002completely}. This is a perspective that is often taken in the operator algebra- and the quantum information literature. The latter is because one can naturally associate to a quantum channel $\Phi$ with Kraus operators $K_1, \dots, K_m$ the operator space $\lspan \{K_i\adjoint K_j \mid i,j \in [m] \}$, called the \emph{quantum confusability graph} of $\Phi$. This generalises classical confusability graphs \cite{shannon2003zero}, extending the theory of Shannon capacities to the quantum case \cite{duan_zeroerror_2013}.

Analogous to the proof of \cref{prop:adjoprealignedisproj} one may translate the conditions for undirectedness and having (no) loops into the operator space picture. This yields the following characterisation, see also \cite{weaver2021quantum}.

\begin{proposition}\label{prop:qgraphpropsopspace}
  Let $G$ be a quantum graph and $S = \img G^R$ the operator space associated to it. Then
  \begin{enumerate}
    \item $G$ is undirected if and only if $X \in S \implies X\adjoint \in S$,
    \item $G$ has loops at every vertex if and only if $\identity \in S$.
    \item $G$ has no loops if and only if $X \in S \implies \Tr(X) = 0$.
  \end{enumerate}
\end{proposition}

As is usual in classical graph theory, we will restrict our attention to \emph{simple} quantum graphs, that is undirected quantum graphs without loops. Note that in the literature, many authors instead work in a setting where quantum graphs have loops at every vertex. This setting is more natural in the context of confusability graphs, cf. \cite{duan_zeroerror_2013}. However, just as for classical graphs, there is a natural correspondence between loopless quantum graphs and quantum graphs with loops at every vertex, which is a direct corollary of \cref{prop:qgraphpropsopspace}.

\begin{proposition}
  Let $S$ be the operator space of a loopless quantum graph. Then $S \oplus \C\identity$ is the operator space of a quantum graph with loops at every vertex. Conversely, for every operator space $S'$ of a quantum graph with loops at every vertex, $S' \ominus \C\identity$ is the operator space of a loopless quantum graph.
\end{proposition}

This is equivalent to the following statement about adjacency operators, which directly generalises the classical case.

\begin{proposition}\label{prop:looplessallloopsequiv}
  Let $G$ be a loopless quantum graph. Then $G + \identity$ is a quantum graph with loops at every vertex. Conversely, for a quantum graph $H$ with loops at every vertex, $H - \identity$ is a loopless quantum graph.
  \begin{proof}
    By \cref{prop:qgraphpropsopspace}, $G^R$ projects onto a subspace orthogonal to $\C\identity$. This means that $G^R + \Omega\Omega\adjoint$ is a projector, where $\Omega\Omega\adjoint$ is the rank-$1$ projector onto the identity matrix $\identity \in M_n$.
    Letting $S = \img G^R$, this projects precisely onto $S \oplus \C\identity$. At the same time, the realignment of the identity on $M_n$ is given by
    \begin{equation*}
      \input{diagrams/identity-realigned.tikz}\enspace,
    \end{equation*}
    which is precisely $\Omega\Omega\adjoint$. It follows that $(G + I)^R = G^R + \Omega\Omega\adjoint$ as desired. Similarly, $H^R$ projects onto a subspace $S'$ with $\identity \in S'$. This implies that $(H - I)^R = H^R - \Omega\Omega\adjoint$ is a projector with image $S' \ominus \C\identity$.
  \end{proof}
\end{proposition}

\subsection{Properties of Quantum Graphs}\label{ssec:graphprops}

Most fundamental definitions of classical graph theory have been generalised to the quantum case. In this paper, we focus on four concepts: connected components, colouring, independent sets, and cliques.

\paragraph{Connected Components.} Let $G$ be a quantum graph. Courtney, Ganesan, and Wasilewski \cite{courtney2025connectivity} define $G$ to be \emph{connected} if
\begin{equation*}
  L_p G = G L_p \quad\Implies\quad p \in \{0, \identity\}
\end{equation*}
for all projectors $p \in M_n$, where $L_p$ denotes left multiplication by $p$. Note that this is equivalent to saying that $G$ is \emph{disconnected} if there exists a non-trivial projector $p$ such that
\begin{equation*}
  L_pGL_{p^\bot} = 0,
\end{equation*}
where $p^\bot$ is the projector onto $\ker p$. Left multiplication can be expressed graphically in a simple way:
\begin{equation*}
  \input{diagrams/leftmul-simplified.tikz}
\end{equation*}
We can thus restate the definition of connectedness nicely in graphical terms.
\begin{definition}
  A quantum graph $G$ is \emph{disconnected} if there exists a non-trivial projector $P\colon \C^n \to \C^n$ such that
  \begin{equation}\label{eq:conndef}
    \input{diagrams/conn-cond-mn.tikz}.
  \end{equation}
  Otherwise $G$ is called \emph{connected}.
\end{definition}

As mentioned in \cite{courtney2025connectivity}, $L_p$ should be thought of as the projector onto a connected component of $G$. This immediately suggests the following definition for connected components.

\begin{definition}
  A quantum graph $G$ has (at least) $k$ \emph{connected components} if there exist projectors $P_1, \dots, P_k$ on $\C^n$ with $\sum_s P_s = \identity$ such that
  \begin{equation*}
    \input{diagrams/conn-comp-condition-mn.tikz}
  \end{equation*}
  for all $s \neq t$.
\end{definition}

Observe that for $k = 1$ the condition is satisfied vacuously, so a connected graph still has at least $1$ connected component. Conversely, if $G$ is disconnected then there exists a projector $P$ satisfying \cref{eq:conndef}. In this case the choice $P_1 = P, P_2 = P^\bot$ witnesses that $G$ has at least two connected components.

\paragraph{Colouring.} A definition of graph colouring for quantum graphs was introduced by Brannan, Ganesan, and Harris \cite{brannan_quantumclassical_2022} in the context of non-local games. As common in the non-local games literature, they actually introduce multiple versions of colouring that differ in strength based on how correlated the answers of players in a certain non-local game are allowed to be. We will only consider \emph{classical} colourings, which correspond to classical winning strategies in the colouring game. Translated into our language, the definition reads as follows.

\begin{definition}[cf. Definition 2.15 in \cite{ganesan_spectral_2023}]
  Let $G$ be a quantum graph and $S = \img G^R$ its associated operator space. Then $G$ has a \emph{$k$-colouring} if there exist projections $P_1, \dots, P_k$ with $\sum_s P_s = \identity$ such that
  \begin{equation*}
    P_sXP_s = 0
  \end{equation*}
  for all $X \in S$ and $s \in [k]$.
\end{definition}

We can express this condition graphically as
\begin{equation*}
  \input{diagrams/colouring-cond.tikz},
\end{equation*}
which has to hold for all $x \in M_n$ and $s \in [k]$. We can get rid of the $x$ and write the condition in terms of the adjacency operator, which yields
\begin{equation*}
  \input{diagrams/colouring-cond-rewr.tikz} \quad\qquad\IFF\qquad\quad \input{diagrams/colouring-cond-rewr-2.tikz}.
\end{equation*}
Finally, we can take the conjugate of both sides, which recovers an equivalent condition analogous to the one for connected components.
\setcounter{theorem}{\value{theorem}-1}
\let\oldthedefinition\thedefinition
\renewcommand{\thedefinition}{\oldthedefinition'}
\begin{definition}
  A quantum graph $G$ has a \emph{$k$-colouring} if there exist projectors $P_1, \dots, P_k$ on $\C^n$ with $\sum_s P_s = \identity$ such that
  \begin{equation*}
    \input{diagrams/colouring-cond-mn.tikz}
  \end{equation*}
  for all $s \in [k]$.
\end{definition}
\let\thedefinition\oldthedefinition

The \emph{chromatic number} $\chi(G)$ of a quantum graph $G$ is defined as the smallest $k$ such that $G$ has a $k$-colouring.

\paragraph{Independent Sets.}
There are multiple possible generalisations of independent sets to quantum graphs---cf. \cite[Proposition 2]{duan_zeroerror_2013}---that are more or less natural depending on whether one approaches the problem from the perspective of classical graph theory, or operationally from the perspective of quantum channels and their confusability graphs. In this paper, we will define an independent set as follows.

\begin{definition}[cf. Definition 1.1 in \cite{weaver2017quantum}]\label{def:indepsets}
  Let $S \subseteq M_n$ be the operator space corresponding to a quantum graph $G$. A \emph{$k$-independent set} of $G$ is a projector $P \in M_n$ of rank $k$ such that $PSP \subseteq \C P$. The \emph{independence number} $\alpha(G)$ is the largest $k$ such that $G$ has a $k$-independent set.
\end{definition}

Note that our definition slightly differs from that given by Weaver in \cite{weaver2017quantum} and Duan, Severini, and Winter in \cite{duan_zeroerror_2013}. There they require that $PSP = \C P$, which is precisely the Knill-Laflamme error correction condition \cite{knill1997theory}. This means that if $S$ represents the confusability graph of a quantum channel, then an independent set is precisely a code that allows for zero-error communication. The reason why we modify this definition is that for confusability graphs, $S$ is an operator \emph{system}, that is it is closed under adjoints and, crucially, contains the identity. This means that $PSP$ always contains at least $P \identity P = P$, and thus $\C P \subseteq PSP$. We work in a more general framework, where quantum graphs are arbitrary operator spaces. Since we want to work with loopless quantum graphs, our definition needs to account for the possibility that there is no $X \in S$ such that $PXP = P$. This yields the condition $PSP \subseteq \C P$. Graphically, we can express \cref{def:indepsets} as follows, where $G^R$ is the projector onto the operator space and $c_x \in \C$ is a (potentially zero) scalar.
\begin{equation*}
  \input{diagrams/indepset-cond.tikz} \quad\qquad\IFF\qquad\quad \input{diagrams/indepset-cond-rewr.tikz}
\end{equation*}
The condition has to hold for all $x \in M_n$.

\paragraph{Cliques.}
To define cliques, we again stick to the definition given in \cite{weaver2017quantum}. While other natural definitions exist, this one has been studied in-depth and for example permits a generalisation of Ramsey's theorem to quantum graphs, as proved in the above-mentioned paper. Just as for independent sets, we adapt Weaver's definition to account for quantum graphs without loops.

\begin{definition}[cf. Definition 1.1 in \cite{weaver2017quantum}]\label{def:cliques}
  Let $S \subseteq M_n$ be the operator space corresponding to a quantum graph $G$. A \emph{$k$-clique} of $G$ is a projector $P \in M_n$ of rank $k$, such that $P(S \oplus \C\identity)P = PM_nP$. The \emph{clique number} $\omega(G)$ is the largest $k$ such that $G$ has a $k$-clique.
\end{definition}

%% file: diagrams/string.tikz
\begin{tikzpicture}
	\begin{pgfonlayer}{nodelayer}
		\node [style=none] (0) at (0, -0.75) {};
		\node [style=none] (2) at (0, 0.75) {};
	\end{pgfonlayer}
	\begin{pgfonlayer}{edgelayer}
		\draw [style=x-wire] (0.center) to (2.center);
	\end{pgfonlayer}
\end{tikzpicture}

%% file: diagrams/box.tikz
\begin{tikzpicture}
	\begin{pgfonlayer}{nodelayer}
		\node [style=map] (0) at (0, 0) {$f$};
		\node [style=none] (1) at (0, 0.75) {};
		\node [style=none] (2) at (0, -0.75) {};
	\end{pgfonlayer}
	\begin{pgfonlayer}{edgelayer}
		\draw [style=x-wire] (2.center) to (0);
		\draw [style=a-wire] (0) to (1.center);
	\end{pgfonlayer}
\end{tikzpicture}

%% file: diagrams/ctox.tikz
\begin{tikzpicture}
	\begin{pgfonlayer}{nodelayer}
		\node [style=cket] (0) at (0, -0.25) {$x$};
		\node [style=none] (1) at (0, 0.5) {};
	\end{pgfonlayer}
	\begin{pgfonlayer}{edgelayer}
		\draw [style=x-wire] (0) to (1.center);
	\end{pgfonlayer}
\end{tikzpicture}

%% file: diagrams/unit.tikz
\begin{tikzpicture}
	\begin{pgfonlayer}{nodelayer}
		\node [style=x-frob] (0) at (0, -0.25) {};
		\node [style=none] (2) at (0, 0.25) {};
	\end{pgfonlayer}
	\begin{pgfonlayer}{edgelayer}
		\draw [style=x-wire] (0) to (2.center);
	\end{pgfonlayer}
\end{tikzpicture}

%% file: diagrams/counit.tikz
\begin{tikzpicture}
	\begin{pgfonlayer}{nodelayer}
		\node [style=x-frob] (0) at (0, 0.25) {};
		\node [style=none] (2) at (0, -0.25) {};
	\end{pgfonlayer}
	\begin{pgfonlayer}{edgelayer}
		\draw [style=x-wire] (0) to (2.center);
	\end{pgfonlayer}
\end{tikzpicture}

%% file: diagrams/stdbasis.tikz
\begin{tikzpicture}
	\begin{pgfonlayer}{nodelayer}
		\node [style=ket] (0) at (0, -0.25) {$i$};
		\node [style=none] (1) at (0, 0.5) {};
	\end{pgfonlayer}
	\begin{pgfonlayer}{edgelayer}
		\draw [style=x-wire] (0) to (1.center);
	\end{pgfonlayer}
\end{tikzpicture}

%% file: content/groupactions.tex
To construct our examples of quantum graphs, we turn our attention to group actions. Classically, subgroups of the symmetric group $S_n$ act on a classical graph by permuting its vertices. Concretely, given such a subgroup $X \subseteq S_n$, realised as $n \times n$ permutation matrices, and a classical graph $G$ with adjacency matrix $A \in \{0, 1\}^{n \times n}$, we say that $X$ acts on $G$ if $PA = AP$ for all $P \in X$. Now it is easy to see that if $X = S_n$, the only loopless graphs that $X$ acts on are the complete graph $A = J - \identity$ and the empty graph $A = 0$. As one considers smaller and smaller $X \subseteq S_n$, one obtains larger and larger classes of graphs. The cyclic group $Z_n \subseteq S_n$, for instance, yields precisely the circulant graphs. We apply this reasoning to quantum graphs on $M_n$, where the unitary group $U(n)$ takes the role of $S_n$. Natural choices of subgroups $X \subseteq U(n)$ will yield different classes of quantum graphs, which have in common that they can be thought of in discrete terms and are amenable to study in the diagrammatic calculus.

We begin by establishing what it means for a classical matrix group to act on a quantum graph. Recall that just as in the classical case, a quantum graph is a (quantum) set of vertices together with some extra structure. In particular, a group acting on a quantum graph should first and foremost act on the set of vertices.
 Note that we have $M_n \cong \End(\C^n)$. This means that there is already a very natural action of subgroups of $\GL(\C^n) \subseteq \End(\C^n)$---conjugation. However, since we view $M_n$ as a \Cstar-algebra, we require the group action to preserve the \Cstar-involution. In the case of $M_n$ this is the Hermitian adjoint, so we must restrict to the unitary group.

\begin{observation}
  A subgroup of $U(n)$ acts on $M_n$ by conjugation, that is, $\alpha\colon U(n) \times M_n \to M_n$ defines a group action with $\alpha(g, x) = g x g\adjoint$.
\end{observation}

This group action takes a particularly nice graphical form.
\begin{equation*}
  \input{diagrams/groupaction-simplified.tikz}
\end{equation*}

Now just as in the classical case, we say that a group acts on a quantum graph if the adjacency matrix commutes with the group action on the underlying set.

\begin{definition}
  Let $G$ be a quantum graph. A group $X \subseteq U(n)$ \emph{acts on $G$} if
  \begin{equation*}
    (x \otimes \overline{x}) \circ G = G \circ (x \otimes \overline{x})
  \end{equation*}
  for all $x \in X$.
\end{definition}

\subsection{The Unitary and the Orthogonal Group}\label{sec:UO-symmetry}

Let us now investigate which graphs are acted on by different subgroups of $U(n)$, starting with $U(n)$ itself.

\begin{lemma}\label{lem:unitary-invariant-maps}
  The linear maps invariant under the group action of $U(n)$ are precisely those of the form
  \begin{equation*}
    \input{diagrams/unitary-invariant-maps.tikz}\enspace.
  \end{equation*}
  \begin{proof}
    The statement follows from an application of Schur-Weyl duality \cite[Chapter 9]{goodman2009symmetry}. A linear map $F \colon \C^n \otimes \C^n \to \C^n \otimes \C^n$ satisfies the group action condition for the unitary group if and only if
    \begin{equation*}
      \input{diagrams/groupaction-cond-unitary.tikz}
    \end{equation*}
    for all $u \in U(n)$. Taking the partial transpose of both sides and sliding boxes along wires yields
    \begin{equation*}
      \input{diagrams/groupaction-cond-unitary-ptrnsp.tikz}\enspace.
    \end{equation*}
    Schur-Weyl duality now tells us that this is satisfied if and only if $F^\Gamma$ is a linear combination of the identity and the swap map,
    \begin{equation*}
      \input{diagrams/unitary-invariant-ptrnsp.tikz}\enspace.
    \end{equation*}
    The identity is invariant under the partial transpose, and the partial transpose of the flip map is the rank-$1$ projection onto the identity matrix. Consequently, taking the partial transpose of both sides yields the desired result.
  \end{proof}
\end{lemma}

   This result is a slight generalisation of the well-known fact that the unitary covariant quantum channels are precisely the \emph{depolarizing channels}, see \cite{werner1989quantum,keyl1999optimal}, or \cite[Proposition 3.1]{nechita2025random} for a modern treatment. Now we have determined the linear maps that are invariant under the group action, but not every linear map $F \colon M_n \to M_n$ is the adjacency operator of a quantum graph. For that, the linear map has to be Schur idempotent and real. In order to classify the quantum graphs that $U(n)$ acts on, we need to determine the range of parameters $\alpha$, $\beta$ that lead to $F$ satisfying these properties. However, instead of treating $F$ as an adjacency operator, it will be more convenient to view it as the projector onto an operator space. In this case, we have to determine when $F$ is idempotent and self-adjoint, which is much easier to verify. As the following lemma shows, we can start from the same conditions.

\begin{lemma}\label{lem:projinvcond}
  Let $X \subseteq U(n)$ be a group and let $\alpha_x \colon M_n \to M_n$ be the action of a group element $x \in X$ on $M_n$. Then for any $F \colon M_n \to M_n$, $\alpha_x F \alpha_x\adjoint = F$ if and only if $\alpha_x F^R \alpha_x\adjoint = F^R$.
  \begin{proof}
    We know that $\alpha_x$ is of the form $x \otimes \overline{x} \in M_n \otimes M_n$. The statement thus follows by sliding boxes along wires.
    \begin{equation*}
      \input{diagrams/realignment-preserves-groupaction.tikz} \qedhere
    \end{equation*}
  \end{proof}
\end{lemma}

The conditions that make a linear map the adjacency operator of a loopless quantum graph restrict the possible values of the complex parameters $\alpha$ and $\beta$ from \cref{lem:unitary-invariant-maps} to only two possible values: $\alpha = \beta = 0$, leading to the zero map, and $\alpha = -1/n$, $\beta = 1$, leading to the map whose realignment is the projector onto the space of traceless matrices. The former is the empty graph $\compl{K}_n$, while the latter is known as the (irreflexive) quantum complete graph on $M_n$, which we denote by $K_n$.

\begin{proposition}\label{prop:unitary-invariant-qgraphs}
  The quantum graphs on $M_n$ acted on by $U(n)$ are precisely the edgeless graph $\compl{K}_n$ and the complete graph $K_n$.
  \begin{proof}
    We will determine which projectors correspond to quantum graphs acted on by $U(n)$, the adjacency operators are then obtained by realigning. By \cref{lem:unitary-invariant-maps,lem:projinvcond}, we may assume that the projectors are of the form
    \begin{equation*}
      \input{diagrams/unitary-invariant-maps-alt.tikz}\enspace.
    \end{equation*}
    Recall that the realignment of the identity is the rank-$1$ projector onto the identity matrix and vice versa, which simply swaps the roles of $\alpha$ and $\beta$ in this case. Composing a linear map of this form with itself yields
    \begin{equation*}
      \input{diagrams/unitary-invariant-maps-squared.tikz}\enspace,
    \end{equation*}
    so for the map to be an idempotent, it needs to hold that
    \begin{align*}
      \beta &= \beta^2\\
      \alpha &= 2\alpha\beta + \alpha^2n.
    \end{align*}
    The first condition requires that $\beta \in \{0, 1\}$. If $\beta = 0$, then $\alpha = 0$ or $\alpha = 1/n$. If $\beta = 1$, then $\alpha = 0$ or $\alpha = -1/n$. All these values are real, so they all make the map self-adjoint. It follows that they all correspond to quantum graphs. However, by plugging in these values and realigning, one finds that only the choices $\alpha = \beta = 0$ and $\alpha = -1/n, \beta = 1$ yield an adjacency operator that satisfies the looplessness condition. These choices correspond precisely to $\compl{K}_n$ and $K_n$.
  \end{proof}
\end{proposition}

We recover in this way the simplest possible quantum graphs on $M_n$. To obtain more interesting examples, one can consider the action of other subgroups of $U(n)$; the first such candidate is the orthogonal group $O(n)$.

\begin{lemma}
  The linear maps invariant under the group action of $O(n)$ are precisely those of the form
  \begin{equation*}
    \input{diagrams/orthogonal-invariant-maps.tikz}\enspace.
  \end{equation*}
  \begin{proof}
    The statement also follows from Schur-Weyl duality \cite[Chapter 10]{goodman2009symmetry}. In the orthogonal case, the role of the symmetric group is replaced by the Brauer algebra. A graphical basis of the Brauer algebra $\mathfrak{B}_2$ consists of all perfect matchings of the classical complete graph on $4$ vertices, which are precisely the components of the linear combination above.
  \end{proof}
\end{lemma}

This result is again closely related to an existing result in the quantum information literature. Namely that the space spanned by $O(n)$ covariant quantum channels is spanned by three maps: the identity channel, the completely depolarizing channel, and the transpose map, see \cite{werner1989quantum,eggeling2001separability} or \cite[Proposition 3.1]{nechita2025random}. Completely analogous to \cref{prop:unitary-invariant-qgraphs} one can determine the values of $\alpha, \beta, \gamma$ that make the linear map the adjacency operator of a loopless quantum graph. Choosing $\gamma = 0$ recovers the $U(n)$-invariant quantum graphs, and we obtain two additional quantum graphs for $\gamma = \pm 1/2$. We omit the (straightforward, but tedious) proof.

\begin{proposition}
  The quantum graphs on $M_n$ acted on by $O(n)$ are precisely $\compl{K}_n$, $K_n$, and the following two graphs:
  \begin{itemize}
    \item the \emph{antisymmetric quantum graph} $\Gasym$, having as realignment the projector onto the antisymmetric subspace of $\C^n \otimes \C^n$
    \item the \emph{symmetric quantum graph} $\Gsym$, having as realignment the projector onto the orthogonal complement of the identity inside the symmetric subspace of $\C^n \otimes \C^n$.
  \end{itemize}
  Diagrammatically, these quantum graphs are given by:
  \begin{equation*}
    \Gasym ~=~~ \input{diagrams/gasym.tikz} \qquad\&\qquad \Gsym ~=~~ \input{diagrams/gsym.tikz}\enspace.
  \end{equation*}
\end{proposition}

It was shown in \cite{gromada_examples_2022} that there exist precisely four non-isomorphic simple quantum graphs on $M_2$, which are uniquely determined by their number of edges. It turns out that these are exactly the graphs acted on by the orthogonal group $O(2)$.

\begin{proposition}
  There are four non-isomorphic quantum graphs on $M_2$. They are the edgeless graph $\compl{K}_2$, $\Gasym$, $\Gsym$, and the complete graph $K_2$.
  \begin{proof}
    The number of edges of a quantum graph $G$ is given by $u\adjoint G u$. We show that $\compl{K}_2$, $\Gasym$, $\Gsym$, and $K_2$ have $0$, $4$, $8$, and $12$ edges respectively. Clearly $\compl{K}_2$ has $0$ edges. The rest can be seen by a diagrammatic calculation. Note that we have to include the normalisation constants that we suppressed.
    \begin{equation*}
      \input{diagrams/ex-asym-edgecount.tikz} \quad = 2^3 - 2^2 = 4, \qquad\quad
      \input{diagrams/ex-complete-qgraph-edgecount.tikz} \quad = 2^4 - 2^2 = 12,
    \end{equation*}
    \begin{equation*}
      \input{diagrams/ex-sym-edgecount.tikz} \quad = 2^3 + 2^2 - 2^2 = 8.
    \end{equation*}
    \begin{equation*}
    \end{equation*}
    This implies that they are pairwise non-isomorphic and thus must be the four graphs in question.
  \end{proof}
\end{proposition}

\subsection{The Diagonal Unitary and the Diagonal Orthogonal Group}

One way to make the matrix groups even smaller and hence obtain more quantum graphs, is to look at their diagonal fragments. This gives rise to the \emph{diagonal unitary group} $DU(n)$ and the \emph{diagonal orthogonal group} $DO(n)$, which consist of diagonal unitary and orthogonal matrices respectively. The matrices in $DO(n)$ in particular are exactly the diagonal matrices that have $\pm 1$ on the main diagonal. In particular the group is finite!
The linear maps on $M_n$ that satisfy the corresponding relation have been characterised in \cite[Propositions 6.4 and 7.2]{nechita2021graphical}, where they are called CLDUI (\textbf{c}onjugate \textbf{l}ocal \textbf{d}iagonal \textbf{u}nitary \textbf{i}nvariant) and LDOI (\textbf{l}ocal \textbf{d}iagonal \textbf{o}rthogonal \textbf{i}nvariant) respectively. The CLDUI maps are given by
\begin{equation*}
  \input{diagrams/diag-unitary-invariant-maps.tikz}\enspace,
\end{equation*}
while the LDOI maps are given by
\begin{equation*}
  \input{diagrams/diag-orthogonal-invariant-maps.tikz}\enspace.
\end{equation*}
The $A, B, C$ are matrices in $M_n$ such that $\diag A = \diag B = \diag C$, and we write for arbitrary matrices $X$
\begin{equation*}
  \mathring{X} := X - \diag X.
\end{equation*}
Again, we have to restrict to the cases where the linear map corresponds to a quantum graph, and again this will be easiest if we view these maps as as projectors instead of adjacency operators. To relate the conditions to adjacency, we make use of the following lemma.

\begin{lemma}[Proposition 4.3 in \cite{singh2021diagonal}]
  For all $A, B, C \in M_n$ we have $(X_{A, B, C})^R = X_{B, A, C}$.
\end{lemma}

The conditions that make $X_{A, B, C}$ a quantum graph can now be stated nicely in terms of $A, B, C$.

\begin{proposition}\label{prop:xabc-qgraph-char}
  A linear map of the form $X_{A, B, C}$ is the adjacency operator of a quantum graph if and only if
  \begin{enumerate}
    \item $B$ is a projector,
    \item $\begin{pmatrix}
        A_{ij} & C_{ij} \\
        C_{ji} & A_{ji}
      \end{pmatrix}$ is a projector for all $i \neq j \in [n]$.
  \end{enumerate}
  It is undirected if and only if
  \begin{enumerate}[resume]
    \item $\mathring{A}$ is self-adjoint,
    \item $B = B\transpose$,
  \end{enumerate}
  It has no loops if and only if
  \begin{enumerate}[resume]
    \item $\mathbf{1} \in \ker B$.
  \end{enumerate}
  \begin{proof}
    Recall that the projector corresponding to an adjacency matrix is obtained by realigning. In order to show that $X_{A, B, C}$ is the adjacency operator of a quantum graph, we thus have to show that $(X_{A, B, C})^R = X_{B, A, C}$ is a projector. Now it can be shown, cf. \cite[Equation 41]{nechita2021graphical}, that
    \begin{equation}\label{eq:xbac-directsum}
      X_{B, A, C} = B \oplus \bigoplus_{i < j} \begin{pmatrix}
        A_{ij} & C_{ij} \\
        C_{ji} & A_{ji}
      \end{pmatrix}.
    \end{equation}
    Both the adjoint and multiplication are componentwise on the direct summands, so  $X_{B, A, C}$ is a  projector if and only if $B$ and all $AC$-blocks are.

    To derive the conditions for undirectedness, recall that a quantum graph with adjacency operator $G$ is undirected if $G = G\transpose$. Since $G$ is required to be real, this is equivalent to $G$ being self-adjoint. Taking the adjoint of $X_{A, B, C}$ yields
    \begin{equation*}
      \input{diagrams/xabc-adj.tikz}\enspace,
    \end{equation*}
    where we used that $B$ is a projector and hence self-adjoint. This is equal to $X_{A, B, C}$ if and only if both $\mathring{A}\adjoint = \mathring{A}$ and $\mathring{C}\adjoint = \mathring{C}$, as well as
    \begin{equation*}
      \input{diagrams/xabc-adj-b-case.tikz}\enspace.
    \end{equation*}
    The latter holds if and only if $B = B\transpose$: We have
    \begin{equation*}
      \input{diagrams/xabc-adj-b-case-condderiv.tikz}\enspace,
    \end{equation*}
    where the second equality holds by isomorphism of diagrams, or more explictly by the Frobenius equations. The converse can be seen by plugging units in the bottom left and top right. The self-adjointness condition for $\mathring{C}$ is satisfied automatically: As projectors, the $AC$-blocks in \cref{eq:xbac-directsum} have to be self-adjoint, which implies that $C_{ij} = \overline{C_{ji}}$ for all $i < j \in [n]$. This is equivalent to $\mathring{C}\adjoint = \mathring{C}$.

    Finally, recall that a quantum graph with adjacency operator $G$ has no loops if and only if
    \begin{equation*}
      \input{diagrams/adj-op-irreflexive.tikz}.
    \end{equation*}
      On $M_n$, the condition becomes
    \begin{equation*}
      \input{diagrams/noloops-cond-mn.tikz}\enspace.
    \end{equation*}
    For $X_{A, B, C}$ not to have loops, we thus want
    \begin{equation*}
      \input{diagrams/xabc-noloops-cond.tikz},
    \end{equation*}
    which is equivalent to
    \begin{equation*}
      \input{diagrams/xabc-noloops-cond-rephrased.tikz}.
    \end{equation*}
    The $A$ and the $C$ terms are the diagrams for the Schur products $\mathring{A} \bullet \identity$ and $\identity \bullet \mathring{C}\transpose$ respectively, which are both $0$ since $\mathring{A}$ and $\mathring{C}$ have zero diagonal by definition. The $B$ term thus has to be $0$ as well, which happens precisely if $B \circ u = 0$ or equivalently, since the wires are $\C^n$, $\mathbf{1} \in \ker B$.
  \end{proof}
\end{proposition}

\begin{corollary}\label{cor:xab-qgraph-char}
  A linear map of the form $X_{A, B}$ is the adjacency operator of a quantum graph if and only if
  \begin{enumerate}
    \item $B$ is a projector,
    \item $\mathring{A}$ is the adjacency matrix of a classical graph on $n$ vertices.
  \end{enumerate}
  It is undirected if
  \begin{enumerate}[resume]
    \item $B = B\transpose$,
    \item $\mathring{A}$ is the adjacency matrix of a loopless undirected classical graph on $n$ vertices.
  \end{enumerate}
  It has no loops if
  \begin{enumerate}[resume]
    \item $\mathbf{1} \in \ker B$.
  \end{enumerate}
  \begin{proof}
    We have $X_{A, B} = X_{A, B, C}$ with $C = \diag A = \diag B$, hence the conditions of \cref{prop:xabc-qgraph-char} apply. It follows that for $X_{A, B}$ to be a quantum graph, $B$ must be a projector. Moreover, the $2 \times 2$ matrices
    \begin{equation*}
      \begin{pmatrix}
        A_{ij} & C_{ij} \\
        C_{ji} & A_{ji}
      \end{pmatrix} =
      \begin{pmatrix}
        A_{ij} & 0\\
        0 & A_{ji}
      \end{pmatrix}
    \end{equation*}
    must be projectors for $i < j$. This is only the case if $A_{ij}, A_{ji} \in \{0, 1\}$. It follows that $\mathring{A}$ must be the adjacency matrix of a loopless classical graph on $n$ vertices. For the quantum graph to be undirected we know we must have $B = B\transpose$ and $\mathring{A}\adjoint = \mathring{A}$. The latter condition implies that $\mathring{A}$ is moreover the adjacency matrix of an undirected loopless classical graph. Finally, the loop condition stays equivalent to whether $\mathbf{1} \in \ker B$.
  \end{proof}
\end{corollary}

We will sometimes be interested in looking at special cases of $X_{A, B}$ and $X_{A, B, C}$ graphs where one or more of the matrices $A, B, C$ are trivial, i.e. zero except on the diagonal (which has to match that of the remaining matrices). We denote these special cases by replacing the trivial matrices by a dot. For example, we let
\begin{align*}
  X_{A, \ccdot} &\coloneqq X_{A, \diag A},\\
  X_{A, \ccdot, C} &\coloneqq X_{A, \diag A, C}.
\end{align*}
We also have $X_{A, B} = X_{A, B, \ccdot}$ and $X_{A, \ccdot} = X_{A, \ccdot, \ccdot}$. Graphically, the adjacency operators are obtained by simply dropping the terms corresponding to the trivial matrices. For quantum graphs of the form $X_{A, \ccdot, C}$, this requires some justification, since the $B$ term contributes the diagonal. However, note that a for loopless quantum graphs we have $B\mathbf{1} = 0$, and for a diagonal $B$ we have $B\mathbf{1} = \diag B$. It follows that $\diag B = 0$ and thus $B = 0$. We thus have
\begin{equation*}
  \input{diagrams/xac-adjop.tikz}
\end{equation*}
as desired. Now a particularly interesting special case occurs when $A$ is the only non-trivial matrix.

\begin{corollary}\label{xa-qgraph-char}
  A linear map of the form $X_{A, \ccdot}$ is the adjacency operator of a (loopless, undirected) quantum graph if and only if $A$ is the adjacency matrix of a (loopless, undirected) classical graph $G$.
  \begin{proof}
    Note that $X_{A, \ccdot} = X_{A, B}$ for $B = \diag A$. \cref{cor:xab-qgraph-char} implies that $X_{A, B}$ is a quantum graph if only if $B$ is a projector and $\mathring{A}$ is the adjacency matrix of a classical graph. Since $B$ is diagonal, it follows that $\diag A = \diag B \in \{0,1\}^n$. Consequently, $A = \mathring{A} + \diag A$ is also the adjacency matrix of a classical graph, potentially containing loops. Now for diagonal $B$, we have $B = B\transpose$, so $X_{A, \ccdot}$ is undirected if and only if $A$ is the adjacency matrix of an undirected classical graph. Finally, $X_{A, \ccdot}$ is loopless if and only if $B\mathbf{1} = 0$. For a diagonal matrix $B$ we have $B\mathbf{1} = 0$ if and only if $\diag B = 0$. Since $\diag A = \diag B$, it follows that $A$ has zero diagonal and is thus the adjacency matrix of a loopless classical graph.
  \end{proof}
\end{corollary}

Since $U(n) \supseteq DU(n)$ and $O(n) \supseteq DO(n)$, the $X_{A, B}$ and $X_{A, B, C}$ graphs must subsume the complete-, edgeless-, symmetric-, and antisymmetric quantum graphs. In particular, $K_n$ and $\compl{K_n}$ must be of the form $X_{A, B}$, while $\Gsym$ and $\Gasym$ must be of the form $X_{A, B, C}$.

\begin{proposition}\label{prop:qgraphs-as-xabc}
  It holds that
  \begin{enumerate}
    \item $\compl{K}_n = X_{\ccdot, \ccdot}$
    \item $K_n = X_{\tfrac{nJ - I}{n}, \tfrac{nI - J}{n}}$
    \item $\Gasym = X_{\tfrac{J - I}{2}, \ccdot, \tfrac{I - J}{2}}$
    \item $\Gsym = X_{\tfrac{J + I}{2} - \tfrac{I}{n}, \tfrac{nI - J}{n}, \tfrac{J + I}{2} - \tfrac{I}{n}}$
  \end{enumerate}
  \begin{proof}
    Can be verified by a straightforward calculation. We show this representatively for the case of $K_n$. We have $J = uu\adjoint$. The $\mathring{A}$ part is thus given graphically by
    \begin{equation*}
      \input{diagrams/kn-xa.tikz}\enspace,
    \end{equation*}
    since we have to remove the diagonal. Together with the $B$ part, we get
    \begin{equation*}
      \input{diagrams/kn-deriv.tikz}\enspace,
    \end{equation*}
    which is precisely the adjacency operator of $K_n$.
  \end{proof}
\end{proposition}

\subsection{The Hyperoctahedral Group}\label{ssec:hypgr}

We now consider an intermediate group that sits between the diagonal orthogonal group and the full orthogonal group: the \emph{hyperoctahedral group}. This group arises naturally as a semi-direct product construction and imposes strong structural constraints on the matrices $A$, $B$, and $C$ parameterizing invariant linear maps.

\begin{definition}
  The \emph{hyperoctahedral group} $\operatorname{Hyp}(n)$ is defined as the semi-direct product
  \begin{equation*}
    \operatorname{Hyp}(n) \coloneqq S_n \ltimes DO(n),
  \end{equation*}
  where $S_n$ denotes the symmetric (permutation) group on $n$ elements and $DO(n)$ denotes the diagonal orthogonal group. Concretely, $\operatorname{Hyp}(n)$ consists of all $n \times n$ matrices that have exactly one nonzero entry in each row and column, and that entry is $\pm 1$. Equivalently, $\operatorname{Hyp}(n)$ is the group of \emph{signed permutation matrices}.
\end{definition}

The hyperoctahedral group fits into the following chain of inclusions:
\begin{equation*}
  DU(n) \supseteq DO(n) \subseteq \operatorname{Hyp}(n) \subseteq O(n) \subseteq U(n).
\end{equation*}
On the diagonal side, we have the containments
\begin{equation*}
  DO(n) \subseteq DU(n) \qquad \text{and} \qquad DO(n) \subseteq \operatorname{Hyp}(n),
\end{equation*}
where $DO(n)$ is a subgroup of both. On the other hand, $DU(n)$ and $\operatorname{Hyp}(n)$ are in general \emph{not} comparable: $DU(n)$ contains diagonal unitaries with arbitrary phases, which do not belong to $\operatorname{Hyp}(n)$, while $\operatorname{Hyp}(n)$ contains non-diagonal permutation matrices, which do not belong to $DU(n)$. Since the invariance group is larger than $DO(n)$, the family of $\operatorname{Hyp}(n)$-invariant linear maps is a \emph{subclass} of the LDOI maps from \cite{singh2021diagonal}. At the same time, since $\operatorname{Hyp}(n) \subseteq O(n)$, this family contains the $O(n)$-invariant maps from \cref{sec:UO-symmetry}.

\begin{proposition}\label{prop:hyp-abc-structure}
  A linear map $X \colon M_n \to M_n$ of the LDOI form $X_{A, B, C}$ satisfies
  \begin{equation*}
    (O \otimes O)\, X_{A,B,C}\, (O \otimes O)^\top = X_{A,B,C} \qquad \text{for all } O \in \operatorname{Hyp}(n)
  \end{equation*}
  if and only if the matrices $A$, $B$, and $C$ each lie in $\operatorname{span}\{I, J\}$, where $I$ denotes the $n \times n$ identity matrix and $J$ denotes the $n \times n$ all-ones matrix.
\end{proposition}

\begin{proof}
  Since $DO(n) \subseteq \operatorname{Hyp}(n)$, any $\operatorname{Hyp}(n)$-invariant map is in particular $DO(n)$-invariant and hence of the LDOI form $X_{A,B,C}$, where $A, B, C \in M_n$ satisfy $\diag A = \diag B = \diag C$. It remains to determine what the additional invariance under all permutation matrices $P_\sigma$ ($\sigma \in S_n$) imposes on $A$, $B$, and $C$.

  By the same argument as in \cite[Section~3]{singh2021diagonal} (see also \cite[Section IV.A]{gulati2025entanglement}), the invariance condition
  \begin{equation*}
    (P_\sigma \otimes P_\sigma)\, X_{A,B,C}\, (P_\sigma^{-1} \otimes P_\sigma^{-1}) = X_{A,B,C} \qquad \text{for all } \sigma \in S_n
  \end{equation*}
  is equivalent to requiring
  \begin{equation*}
    P_\sigma\, A\, P_\sigma^{-1} = A, \qquad P_\sigma\, B\, P_\sigma^{-1} = B, \qquad P_\sigma\, C\, P_\sigma^{-1} = C
  \end{equation*}
  for every permutation matrix $P_\sigma$ with $\sigma \in S_n$. Indeed, expanding the invariance condition for $X_{A, B, C}$ and plugging in the standard basis, we get
  \begin{equation*}
    \input{diagrams/xabc-perm-invariant.tikz}\enspace.
  \end{equation*}
  Using that $P_\sigma \ket{k} = \ket{\sigma(k)}$, the definitions of multiplication and comultiplication, and the fact that $\sigma$ is a bijection one finds this to be equivalent to
  \begin{alignat*}{4}
    &(1 - \delta_{ij}) \delta_{ik} \delta_{\ell j} \cdot A_{\sigma(i)\sigma(j)}
    & {}+{} & \delta_{ij} \delta_{k\ell} B_{\sigma(i)\sigma(j)}
    & {}+{} & (1 - \delta_{ij}) \delta_{i\ell} \delta_{jk} \cdot C_{\sigma(j)\sigma(i)} \\
    =\enspace &(1 - \delta_{ij}) \delta_{ik} \delta_{\ell j} \cdot A_{ij}
    & {}+{} & \delta_{ij} \delta_{k\ell} B_{ij}
    & {}+{} & (1 - \delta_{ij}) \delta_{i\ell} \delta_{jk} \cdot C_{ji}
  \end{alignat*}
  Comparing the two sides of the equation for all choices of $i, j, k, \ell \in [n]$ yields $A_{\sigma(i)\sigma(j)} = A_{ij}$ for all $i, j$ and all $\sigma \in S_n$, where the case $i = j$ follows from the $B$-term and the condition $\diag A = \diag B$. The same can be concluded for $B$ and $C$. This is precisely the condition that $A$, $B$, and $C$ commute with every permutation matrix (via conjugation), i.e., $P_\sigma M P_\sigma^{-1} = M$ for $M \in \{A, B, C\}$ and all $\sigma \in S_n$.

  We now show that a matrix $M \in M_n$ satisfies $P_\sigma M P_\sigma^{-1} = M$ for all $\sigma \in S_n$ if and only if $M \in \operatorname{span}\{I, J\}$, i.e., $M = \alpha I + \beta J$ for some $\alpha, \beta \in \mathbb{C}$. The condition $P_\sigma M P_\sigma^{-1} = M$ means $M_{\sigma(i)\,\sigma(j)} = M_{ij}$ for all $i,j \in [n]$ and all $\sigma \in S_n$. In particular:
  \begin{itemize}
    \item Taking $i \neq j$: For any two pairs $(k, l)$ with $k \neq l$, the $2$-transitivity of $S_n$ provides a permutation $\sigma$ with $\sigma(i) = k$ and $\sigma(j) = l$. It follows that $M_{kl} = M_{ij}$. Hence all off-diagonal entries of $M$ are equal; call this common value $\beta$.
    \item Taking $i = j$: For any two diagonal indices $k, l \in [n]$, there exists $\sigma \in S_n$ with $\sigma(i) = k$ and $\sigma(i) = l$ only if $k = l$. Instead, for any $k$, choose $\sigma$ with $\sigma(i) = k$; then $M_{kk} = M_{ii}$. Hence all diagonal entries of $M$ are equal; call this common value $\alpha + \beta$.
  \end{itemize}
  Consequently, $M_{ij} = \alpha \delta_{ij} + \beta$ for all $i,j$, i.e., $M = \alpha I + \beta J$. The converse is immediate: Any matrix $\alpha I + \beta J$ commutes with every permutation matrix, since $P_\sigma I P_\sigma^{-1} = I$ and $P_\sigma J P_\sigma^{-1} = J$. Applying this to each of $A$, $B$, and $C$ completes the proof.
\end{proof}

It follows that a hyperoctahedral-invariant linear map $X_{A,B,C}$ is described by at most four free parameters, compared to the $3n^2 - 2n$ parameters needed for a general LDOI map. Explicitly, writing $A = a I + a' \mathring{J}$, $B = b_{\mathrm{d}} I + b \mathring{J}$, and $C = c_{\mathrm{d}} I + c \mathring{J}$, the diagonal-matching condition $\diag A = \diag B = \diag C$ forces $a = b_{\mathrm{d}} = c_{\mathrm{d}}$, leaving precisely $4$ real parameters $(a, a', b, c)$. This exactly matches the parameterization of hyperoctahedral states $X^{\mathsf{Hyp}}_{a,a',b,c}$ from \cite{park2024universal}.

\begin{proposition}\label{prop:x-hyp-aa'bc-qgraph-char}
  A linear map of the form $X^{\mathsf{Hyp}}_{a,a',b,c}$ is the adjacency operator of a quantum graph if and only if
  \begin{enumerate}
    \item $(a, b) \in \left\{ (0,0), \left(1/n, 1/n\right), (1,0), \left(1-1/n, -1/n\right) \right\}$,
    \item $(a', c) \in \left\{ (0,0), (1,0), \left(1/2, 1/2\right), \left(1/2, -1/2\right) \right\}$.
  \end{enumerate}
  Every such quantum graph is automatically undirected. It has no loops if and only if
  \begin{enumerate}[resume]
    \item $(a,b)=(0,0)$ or $(a,b)=\left(1-1/n, -1/n\right)$.
  \end{enumerate}

  \begin{proof}
    By \cref{prop:xabc-qgraph-char}, $X^{\mathsf{Hyp}}_{a,a',b,c} = X_{A,B,C}$ for $A = aI + a'\mathring{J}$, $B = aI + b\mathring{J}$, and $C = aI + c\mathring{J}$ is the adjacency operator of a quantum graph if and only if $B$ is a projector and the off-diagonal $2 \times 2$ blocks
    \begin{equation*}
      X_{ij} = \begin{pmatrix}
        A_{ij} & C_{ij} \\
        C_{ji} & A_{ji}
      \end{pmatrix}
    \end{equation*}
    are projectors for all $i \neq j$. The only projections in the span of $I$ and $J$ are $0, J/n, I-J/n$, and $I$. This yields the four cases in the first point.
    For the second condition, observe that
    \begin{equation*}
      X_{ij} = \begin{pmatrix}
        a' & c \\
        c & a'
      \end{pmatrix}.
    \end{equation*}
    For $X_{ij}$ to be a projector, this real symmetric matrix must be zero, the identity, or a rank-$1$ projector. The latter condition is equivalent to $a'=1/2$ and $c = \pm 1/2$. This produces the four cases in the second point.
    Since $(a, a', b, c)$ are real parameters and $\mathring{J}$ is symmetric, the matrices $\mathring{A} = a'\mathring{J}$, $B = aI + b\mathring{J}$, and $\mathring{C} = c\mathring{J}$ are inherently symmetric and self-adjoint. By \cref{prop:xabc-qgraph-char}, this guarantees every such quantum graph is automatically undirected.
    Finally, the graph has no loops if and only if $\mathbf{1} \in \ker B$. This is the case for two of the four options in the first point, finishing the proof.
  \end{proof}
\end{proposition}

Note that the four graphs in \cref{prop:qgraphs-as-xabc} have automatically hyperoctahedral symmetry and correspond to the following choices of parameters:
\begin{enumerate}
  \item $\compl{K_n}$ $\to$ $(a, a', b, c) = (0, 0, 0, 0)$
  \item $K_n$ $\to$ $(a, a', b, c) = (1-1/n, 1, -1/n, 0)$
  \item $\Gsym$ $\to$ $(a, a', b, c) = (1 - 1/n, 1/2, -1/n, 1/2)$
  \item $\Gasym$ $\to$ $(a, a', b, c) = (0, 1/2, 0, -1/2)$.
\end{enumerate}

Each of the four graphs above corresponds to a different choice of parameters $(a',c)$. The other four graphs correspond to the same four choices of $(a', c)$ with value of $(a, b) \in \{(0,0),(1-1/n, -1/n)\}$ swapped. In other words, these extra four quantum graphs are $\compl{K_n}$, $K_n$, $\Gsym$, and $\Gasym$ with the $B = I - J/n$ part added (for $\compl{K_n}$ and $\Gasym$) or removed (for $K_n$ and $\Gsym$).

\subsection{Classical Models of Quantum Graphs}\label{ssec:classicalmodels}
The structure of the $X_{A, B}$ and $X_{A, B, C}$ quantum graphs nicely captures an increasing level of ``quantumness'' as we relax the symmetry requirements. The easiest class of quantum graphs acted on by the diagonal unitary group are those of the form $X_{A, \ccdot}$. By \cref{cor:xab-qgraph-char}, quantum graphs of this form satisfy that $\mathring{A}$ is the adjacency matrix of a classical graph. In fact, $X_{A, \ccdot}$ is precisely the adjacency operator of a classical graph $G$ under the diagonal embedding of $\C^n$ into $M_n$. Moreover, it is an enlightening exercise to see that the corresponding projector, obtained by realigning, is precisely the projector onto the operator space $S = \lspan \{\ket{i}\bra{j} \mid i \sim j\}$.

Moving from purely classical graphs into the quantum realm, our first stop are the quantum graphs of the form $X_{A, B}$, which we will simply call AB graphs. As one can see in \cref{cor:xab-qgraph-char}, the conditions on $A$ and $B$ that make $X_{A, B}$ a quantum graph are almost completely independent. Except for the diagonal---which has to coincide with that of $B$---$A$ is still the adjacency matrix of a classical graph. $B$ has to be a projector. We may hence think of $X_{A, B}$ simply as a classical graph to which a subspace of $\C^n$ is attached.

If instead of requiring full diagonal unitary symmetry, we only require diagonal orthogonal symmetry, we end up with ABC graphs: quantum graphs of the form $X_{A, B, C}$. Taking this step, we lose the direct connection to classical graphs. As we have seen in \cref{prop:xabc-qgraph-char}, the conditions on $A$ are much less strict, in particular $\mathring{A}$ no longer has to be an adjacency matrix. There is, however, still enough structure to think of these quantum graphs in terms of a classical model, at least whenever $X_{A, B, C}$ is undirected. The central observation is that condition (2) of \cref{prop:xabc-qgraph-char} restricts how $A$ and $C$ must interact, while $B$ is still almost entirely independent. Condition (2) requires that
\begin{equation*}
  X_{ij} \coloneqq \begin{pmatrix}
    A_{ij} & C_{ij}\\
    C_{ji} & A_{ji}
  \end{pmatrix}
\end{equation*}
be a projector for all $i < j$. In particular, one of three cases must apply.

\paragraph{$X_{ij}$ has rank 0.} In this case $X_{ij} = 0$ and thus $A_{ij} = A_{ji} = C_{ij} = C_{ji} = 0$.

\paragraph{$X_{ij}$ has rank 2.} In this case, $X_{ij}$ has full rank and is thus equal to the identity. It follows that $A_{ij} = A_{ji} = 1$.

\paragraph{$X_{ij}$ has rank 1.} In this case, we must have $\Tr X_{ij} = \rk X_{ij} = 1$. Since $X_{ij}$ is a projector and hence positive-semidefinite, the diagonal elements $A_{ij}$, $A_{ji}$ must be non-negative reals. Assuming $X_{A, B, C}$ is undirected, we have $\mathring{A}\adjoint = \mathring{A}$ and hence $A_{ij} = A_{ji}$. Consequently $A_{ij} = A_{ji} = 1/2$. Since $X_{ij}$ is rank-deficient, we have $\det X_{ij} = A_{ij}A_{ji} - C_{ij}C_{ji} = 0$, and since $X_{ij}\adjoint = X_{ij}$, we have $C_{ji} = \overline{C_{ij}}$. It follows that $A_{ij}A_{ji} = \abs{C_{ij}}^2$ and thus $\abs{C_{ij}}^2 = 1/4$. This occurs precisely when $C_{ij} = e^{\theta i}/2$. It follows that every rank-1 $X_{ij}$ is uniquely determined by a phase $\theta$.
\\

Consequently, we may interpret an undirected quantum graph of the form $X_{A, \ccdot, C}$ as an undirected classical graph on $n$ vertices with two kinds of edges: Classical edges, which correspond to rank-2 $X_{ij}$, and ``strange'' edges, which are annotated with a phase $\theta$ and correspond to rank-1 $X_{ij}$.

\begin{definition}[Strange Graphs]
  Let $G$ be an undirected quantum graph on $M_n$ of the form $X_{A, B, C}$. The \emph{strange graph} $\mathfrak{G}(A, C)$ associated to $G$ is an undirected graph with vertex set $[n]$ and edge set $E \cup \bigcup_{\theta \in [0, 2\pi)}\tilde{E}_\theta$, where
  \begin{equation*}
    E = \{ \{i, j\} \subseteq [n] \mid X_{ij} = \identity \} \qquad\&\qquad \tilde{E}_\theta = \{\{i, j\} \subseteq [n] \mid i < j, X_{ij} = \frac{1}{2}\begin{pmatrix}
        1 & e^{\theta i} \\
        e^{-\theta i} & 1
    \end{pmatrix}\}.
  \end{equation*}
\end{definition}

A quantum graph of the form $X_{A, B, C}$ can then be thought of as a strange graph to which a subspace of $\C^n$ is attached. We have seen that in particular $\Gsym$ and $\Gasym$ are of this form.

\begin{proposition}\label{prop:gsymgasymstrange}
  The strange graph associated to $\Gsym$ is the complete graph consisting only of strange edges with phase $0$. The strange graph associated to $\Gasym$ is the complete graph consisting only of strange edges with phase $\pi$.
  \begin{proof}
    By \cref{prop:qgraphs-as-xabc}, we have $\Gsym = X_{A, B, C}$ for $A = (J + \identity)/2$ and $C = (J + \identity)/2$ and thus
    \begin{equation*}
      X_{ij} = \frac{1}{2} \begin{pmatrix}
        1 & 1 \\
        1 & 1
        \end{pmatrix} = \frac{1}{2} \begin{pmatrix}
        1 & e^{0 i} \\
        e^{0 i} & 1
      \end{pmatrix}.
    \end{equation*}
    Similarly, we have $\Gasym = X_{A, B, C}$ for $A = (J - \identity)/2 - \identity/n$ and $C = (\identity - J)/2 - \identity/n$. It follows that
    \begin{equation*}
      X_{ij} = \frac{1}{2} \begin{pmatrix}
        1 & -1 \\
        -1 & 1
        \end{pmatrix} = \frac{1}{2} \begin{pmatrix}
        1 & e^{\pi i} \\
        e^{-\pi i} & 1
      \end{pmatrix}.\qedhere
    \end{equation*}
  \end{proof}
\end{proposition}

To be precise, $\Gsym$ additionally has a subspace attached. A counterpart without the subspace contribution also exists and has hyperoctahedral symmetry instead of full orthogonal symmetry, see \cref{ssec:hypgr}.
We have also seen that $X_{A, \ccdot}$ is a classical graph embedded into $M_n$. It turns out that this classical graph is precisely the strange graph associated to $X_{A, \ccdot}$.

\begin{proposition}\label{prop:xastrangeisa}
  Let $A$ be an undirected classical graph. Then the strange graph associated to $X_{A, B}$ has no strange edges and satisfies $\mathfrak{G}(A, \ccdot) \cong A$.
  \begin{proof}
    We have $X_{A, B} = X_{A, B, C}$ for some $C \in M_n$ with $C_{ij} = 0$ for all $i \neq j$. This implies that $X_{ij} = A_{ij} \identity$ for all $i \neq j \in [n]$ and thus $\mathfrak{G}(A, C) \cong A$.
  \end{proof}
\end{proposition}

The strange graph formalism is moreover nicely compatible with complementation of quantum graphs.

\begin{definition}
  Let $G$ be a loopless quantum graph with corresponding operator space $S$. The \emph{complement} $\compl{G}$ of $G$ is given by the operator space $S^\bot \ominus \C\identity$.
\end{definition}

As is usual in classical graph theory, we consider the \emph{loopless} complement. This is why the complement quantum graph is not given by the orthogonal complement $S^\bot$ of $S$, but by the orthogonal complement of the identity matrix inside $S^\bot$. Indeed, the operator space $S^\bot \ominus \C\identity$ consists precisely of the traceless matrices in $S^\bot$ and thus corresponds to a loopless quantum graph, cf. \cref{prop:qgraphpropsopspace}. In contrast, if $G$ is loopless then we have $\identity \in S^\bot$, which corresponds to a quantum graph with loops at every vertex. Now by realigning the projector onto $S^\bot \ominus \C\identity$, we can write the adjacency operator of $\compl{G}$ explicitly as
\begin{equation*}
  \input{diagrams/adjop-compl.tikz}\enspace.
\end{equation*}
It is not hard to see this way that $\compl{G}$ has the same symmetries as $G$. In particular, this implies that if $G$ is of the form $X_{A, B, C}$ then so is $\compl{G}$. It turns out that we can compute this representation directly.

\begin{proposition}
  It holds that $\compl{X_{A, B, C}} = X_{J - \identity - \mathring{A} + \diag, \identity - B - J/n, -\mathring{C} + \diag}$.
  \begin{proof}
    We derive the equation graphically. We have
    \begin{align*}
      \compl{G} &\quad=\quad \input{diagrams/xabc-compl-deriv.tikz}\\
                &\quad=\quad \input{diagrams/xabc-compl-deriv-2.tikz}\\
                &\quad=\quad \input{diagrams/xabc-compl-deriv-3.tikz}\\
                &\quad=\quad \input{diagrams/xabc-compl-deriv-4.tikz}\enspace.
    \end{align*}
    The $A$ and $C$ terms both have zero diagonal, so adding $\diag (\identity - B - J/n)$ to both yields three matrices with identical diagonals, as desired.
  \end{proof}
\end{proposition}

Observe that ignoring the diagonal correction terms, the $A$ contribution is given by $J - \mathring{A}$, while the $C$ contribution is given by $-\mathring{C}$. On the level of the $X_{ij}$ blocks, this means that rank-$2$ blocks become rank-$0$ blocks and vice versa, while rank-$1$ blocks transform as
\begin{equation*}
  \begin{pmatrix}
    A_{ij}  & C_{ij} \\
    C_{ji} & A_{ji}
  \end{pmatrix} 
  \mapsto
  \begin{pmatrix}
    1 - A_{ij}  & -C_{ij} \\
    -C_{ji} & 1 -A_{ji}
  \end{pmatrix} 
  =
  \begin{pmatrix}
    1 - 1/2  & -e^{\theta i} \\
    -e^{-\theta i} & 1 - 1/2
  \end{pmatrix} 
  =
  \begin{pmatrix}
    1/2  & e^{(\theta + \pi) i} \\
    e^{-(\theta + \pi) i} & 1/2
  \end{pmatrix}. 
\end{equation*}
Concretely, this implies that we can construct the strange graph of $\compl{G}$ by taking the strange graph associated to $G$ and performing the following operations on edges:
\begin{enumerate}
  \item Complement classical edges: Classical edges become non-edges and vice versa.
  \item Invert the phases of strange edges: If a strange edge had phase $\theta$ before, it becomes $\theta + \pi$.
\end{enumerate}
As expected, by \cref{prop:xastrangeisa}, we recover this way the usual complement of classical graphs. Using \cref{prop:qgraphs-as-xabc}, we find that $\compl{(\Gsym)} = \Gasym$ and vice versa, which nicely matches \cref{prop:gsymgasymstrange}.

%% file: content/graphtheory.tex
With our examples in hand, we now want to take a closer look at their graph theoretic properties. This section is structured as follows. To each graph theoretic property introduced in \cref{ssec:graphprops}, we dedicate a subsection. We start with the number of connected components, continue with the chromatic number, and conclude with the independence- and the clique number. Within each subsection we look at the corresponding property for each of the above-mentioned graphs. In the first subsection, to get used to the diagrammatic arguments involved, we start easy with the empty and complete graphs, moving on through the symmetric and antisymmetric graphs, AB graphs, and finally ABC graphs in full generality. In the subsequent subsections, we invert this order, starting with ABC graphs, the most general case, to avoid proving redundant results.

\subsection{Connected Components}
\input{content/graphtheory/conncomp}

\subsection{Colouring}
\input{content/graphtheory/colouring}

\subsection{Independent Sets}
\input{content/graphtheory/indepsets}

\subsection{Cliques}
\input{content/graphtheory/cliques}

We summarise our findings in the table below. Note that while we only list graph properties, we have in certain cases proved even stronger results, such as correspondences between the quantum graph structure and structure of matrices or classical graphs. We have also proved stronger bounds in certain special cases that we do not list for the sake of legibility.

\pagestyle{empty}
\begin{landscape}
    \def\arraystretch{2}
    \begin{table}
    \begin{center}
  % \parbox[c][\textwidth][s]{\linewidth}{%
      \renewcommand\tabularxcolumn[1]{m{#1}}
      \begin{tabularx}{\linewidth}{|Y|Y|Y|c|c|c|}
        \hline
        Group Action & Quantum Graph & Connected \mbox{Components} & $\chi$ & $\alpha$ & $\omega$ \\
        \hline
        \multirow{2}{*}{$U(n)$}
                     & \cellcolor{gray!15}$\compl{K}_n$ &\cellcolor{gray!15}$n$ &\cellcolor{gray!15}$1$ &\cellcolor{gray!15}$n$ &\cellcolor{gray!15}$1$ \\

                     & $K_n$ & $1$ & not colourable & $1$ & $n$ \\

                     \hdashline

        \multirow{2}{*}{$O(n)$} &\cellcolor{gray!15}$\Gsym$ &\cellcolor{gray!15}1 &\cellcolor{gray!15}\makecell{\pad{1.35cm}{$2$ for $n=2$} \\not colourable for $n \geq 3$} &\cellcolor{gray!15}$1$ &\cellcolor{gray!15}$\lceil n/2 \rceil$\\
                                & $\Gasym$ & \makecell{$2$ for $n = 2$\\ \mbox{$1$ for $n \geq 3$}} & $n$ & $1$ & $\lceil n/2 \rceil$\\

                                \hdashline

        \multirow{3}{*}{$DU(n)$} &\cellcolor{gray!15}$ X_{A, \ccdot}$ &\cellcolor{gray!15}same as $A$ &\cellcolor{gray!15}$\chi(A)$ &\cellcolor{gray!15}$\alpha(A)$ &\cellcolor{gray!15}\makecell{\pad{1.4cm}{$\geq \omega(A) - 1$} \\$\leq n - 1$} \\
                                 & $X_{\ccdot, B}$ & $n$ & \makecell{$\leq n$ \\ $\geq n / \dim\ker B$} & \makecell{$\leq n - \rk B$ \\$\geq \max \{ \operatorname{EqRows}(B), 1 + \left\lfloor \frac{n-1}{\rk B + 1} \right\rfloor \}$} &$\leq \sqrt{\rk B + 1}$ \\
                                 &\cellcolor{gray!15}$X_{A, B}$ &\cellcolor{gray!15}same as $A$ &\cellcolor{gray!15}\makecell{\pad{.3cm}{not colourable in general} \\ $\geq \max \{\chi(A), \chi(X_{\ccdot, B}) \}$} &\cellcolor{gray!15}$\leq \min\{\alpha(A), \alpha(X_{\ccdot, B})\}$ &\cellcolor{gray!15}$\geq \max \{\omega(X_{A, \ccdot}), \omega(X_{\ccdot, B})\}$ \\

                                 \hdashline

        \multirow{2}{*}{$DO(n)$} & $X_{A, \ccdot, C}$ & $\geq \mathfrak{G}(A, C)$ & $\leq \chi(\mathfrak{G}(A, C))$ & $\alpha(\mathfrak{G}(A, C))$ &  \\
                                 &\cellcolor{gray!15} $X_{A, B, C}$ &\cellcolor{gray!15}$\geq \mathfrak{G}(A, C)$ &\cellcolor{gray!15}$\geq \max \{\chi(X_{A, \ccdot, C}), \chi(X_{\ccdot, B})\}$ &\cellcolor{gray!15}$\leq \min \{\alpha(\mathfrak{G}(A, C)), \alpha(X_{\ccdot, B})\}$ &\cellcolor{gray!15}$\geq \max \{\omega(X_{A, \ccdot, C}), \omega(X_{\ccdot, B}))\}$ \\

                                 \hline
      \end{tabularx}
    \end{center}
  \caption{Summary of quantum graph properties.}
  \label{fig:qgraphexamples-with-props}
\end{table}
\end{landscape}
\pagestyle{plain}

%% file: content/graphtheory/conncomp.tex
Recall that a quantum graph has at least $k$ connected components if there exist $k$ non-zero projectors $P_1, \dots, P_k$ on $\C^n$ that sum to the identity, such that
\begin{equation}\label{eq:conn-comp-cond-mn}
  \input{diagrams/conn-comp-condition-mn.tikz}.
\end{equation}
for all $s \neq t \in [k]$. This condition can obviously be satisfied vacuously, so every quantum graph has at least one connected component. On the other hand, the non-zero requirement on the projectors shows that a quantum graph can have at most $n$ connected components. This upper bound is achieved by the empty graph: Since the adjacency operator of $\compl{K}_n$ is the zero map, any complete collection of pairwise orthogonal rank-1 projectors trivially satisfies \cref{eq:conn-comp-cond-mn}. We hence conclude the following.

\begin{proposition}
  $\compl{K}_n$ has $n$ connected components.
\end{proposition}

We would also expect the complete graph $K_n$ to be connected, and indeed this is what we find.

\begin{proposition}
  $K_n$ is connected.
  \begin{proof}
    For $K_n$, \cref{eq:conn-comp-cond-mn} reads
    \begin{equation*}
      \input{diagrams/kn-conn-comp-cond.tikz}.
    \end{equation*}
    Since the $P_s$ sum to the identity, $P_s$ and $P_t$ must project onto orthogonal subspaces, and hence $P_sP_t = 0$. The condition thus simplifies to
    \begin{equation*}
      \input{diagrams/kn-conn-comp-cond-simpl.tikz}.
    \end{equation*}
    Observe that for non-zero $P_s, P_t$, the left-hand side is a rank-1 matrix and hence cannot be equal to zero. It follows that any collection of projectors satisfying \cref{eq:conn-comp-cond-mn} contains only the identity and hence $K_n$ is connected.
  \end{proof}
\end{proposition}

\subsubsection{The Symmetric and Antisymmetric Cases}

The number of connected components of the Symmetric and Antisymmetric quantum graphs is also limited by rank constraints. Generally, both graphs are connected, although $n = 2$ is a special case for $\Gasym$.

\begin{proposition}\label{prop:gasymconncomp}
  For $n \geq 3$, $\Gasym$ is connected. For $n = 2$, $\Gasym$ has $2$ connected components.
  \begin{proof}
    For $\Gasym$, \cref{eq:conn-comp-cond-mn} is equivalent to
    \begin{equation}\label{eq:gasymconncompcond}
      \input{diagrams/gasym-conn-comp-cond.tikz}\enspace.
    \end{equation}
    Note that the left-hand side has rank $1$, while the right-hand side has rank $\rk P_s \cdot \rk P_t$. Consequently, the condition is satisfiable only if $\rk P_s = \rk P_t = 1$. We hence must have $P_s = \ket{\psi_s}\bra{\psi_s}$, where $\braket{\psi_s | \psi_t} = \delta_{st}$ for all $s, t$. Plugging these definitions into \cref{eq:gasymconncompcond} yields
    \begin{equation*}
      \input{diagrams/gasym-conn-comp-deriv.tikz}\enspace.
    \end{equation*}
    This is true if and only if $\ket{\overline{\psi_s}} = c\ket{\psi_t}$ and $\ket{\overline{\psi_t}} = d\ket{\psi_s}$ where $c\overline{d} = 1$. Now if $n \geq 3$, then there exist pairwise orthogonal $\ket{\psi_1}$, $\ket{\psi_2}$, $\ket{\psi_3}$, while the condition implies that $\ket{\psi_1} \sim \ket{\overline{\psi_2}} \sim \ket{\psi_3}$. This contradicts their orthogonality, so the only way to satisfy \cref{eq:gasymconncompcond} is vacuously and it follows that $\Gasym$ is connected for $n \geq 3$. On the other hand, if $n = 2$, then $\ket{\psi_1} = (\ket1 + i\ket2)/\sqrt{2}$ and $\ket{\psi_2} = (\ket1 - i\ket2)/\sqrt{2}$ form a maximal set of orthonormal vectors, and we have $\ket{\psi_1} = \ket{\overline{\psi_2}}$. It follows that $\Gasym$ has $2$ connected components for $n = 2$.
  \end{proof}
\end{proposition}

For $\Gsym$, the same reasoning shows that it is connected even for $n = 2$.

\begin{proposition}
  $\Gsym$ is connected.
  \begin{proof}
    For $\Gsym$, \cref{eq:conn-comp-cond-mn} becomes
    \begin{equation*}
      \input{diagrams/gsym-conn-comp-cond.tikz}\enspace.
    \end{equation*}
    Again, since the $P_s$ sum to the identity and are thus mutually orthogonal, the condition simplifies to
    \begin{equation}\label{eq:gsymconncompcond}
      \input{diagrams/gsym-conn-comp-cond-simpl.tikz}\enspace.
    \end{equation}
    Just as for $\Gasym$, the condition is satisfiable only if $\rk P_s = \rk P_t = 1$. We hence must have $P_s = \ket{\psi_s}\bra{\psi_s}$, where $\braket{\psi_s | \psi_t} = \delta_{st}$ for all $s, t$. Plugging these definitions into \cref{eq:gsymconncompcond} yields
    \begin{equation*}
      \input{diagrams/gsym-conn-comp-deriv.tikz}\enspace.
    \end{equation*}
    This is true if and only if $\ket{\overline{\psi_s}} = c\ket{\psi_t}$ and $\ket{\overline{\psi_t}} = d\ket{\psi_s}$ where $c\overline{d} = -1$. But at the same time we have
    \begin{equation*}
      \ket{\psi_t} = \overline{d} \ket{\overline{\psi_s}} = c \overline{d} \ket{\psi_t},
    \end{equation*}
    so $c\overline{d} = 1 \neq -1$. It follows that the only way to satisfy \cref{eq:gsymconncompcond} is vacuously and hence $\Gsym$ is connected.
  \end{proof}
\end{proposition}

\subsubsection{A Splitting Principle}\label{ssec:indepres}
We now move on to the classes of quantum graphs parameterised by matrices $A, B, C$ with $\diag A = \diag B = \diag C$. We have seen in \cref{ssec:classicalmodels} that we can think of these quantum graphs in discrete terms: A quantum graph of the form $X_{A, B, C}$ is uniquely determined by a classical graph on $n$ vertices, whose edges may be annotated with phases. This graph only depends on $A$ and $C$, whereas $B$---always a projector---can be thought of as attaching some subspace of $\C^n$. If $\mathring{C} = 0$, we even recover embeddings of classical graphs. It turns out that these models are not just ways to visualise quantum graphs, they also are a useful model for graph-theoretic investigations. Concretely, we will be able to give tight bounds on graph-theoretic parameters in terms of the strange graphs $\mathfrak{G}(A, C)$, and the subspace $B$.

The first such result is that the conditions for the existence of connected components split into independent conditions for $A, C$---which we may phrase in terms of $\mathfrak{G}(A, C)$---and $B$. We will later find that analogous results hold for the other graph parameters as well. We call this the \emph{splitting principle}. We begin by expanding the connected component condition for $X_{A, B, C}$.
\begin{equation*}
  \input{diagrams/conn-comp-cond-xabc.tikz}
\end{equation*}
We can plug the standard basis into the right tensor legs,
\begin{equation*}
  \input{diagrams/conn-comp-cond-xabc-stdbas.tikz}\enspace.
\end{equation*}
Evaluating the (co)multiplications yields two distinct conditions:
\begin{equation}\label{eq:xacconncondsimpl}
  \input{diagrams/conn-comp-cond-xabc-xac.tikz}
\end{equation}
for $i \neq j$, and since per definition $\mathring{A}_{ii} = \mathring{C}_{ii} = 0$,
\begin{equation*}
  \input{diagrams/conn-comp-cond-xabc-xb.tikz}
\end{equation*}
for $i = j$.
These conditions need to hold for all $i, j \in [n]$. It follows that a set of projectors $P_1, \dots, P_k$ witnesses the existence of $k$ connected components of a quantum graph $X_{A, B, C}$ if and only if they form $k$ connected components of both $X_{A, \ccdot, C}$ and $X_{\ccdot, B}$ individually. We may thus conclude the following.

\begin{proposition}\label{prop:conncompindep}
  Let $k_{AC}$ be the number of connected components of a quantum graph $X_{A, \ccdot, C}$ and $k_B$ the number of connected components of a quantum graph $X_{\ccdot, B}$. Then $X_{A, B, C}$ has at most $\min(k_{AC}, k_B)$ connected components.
\end{proposition}

\subsubsection{AB Graphs}
It turns out that the number of connected components of $X_{A, B}$ only depends on $A$. We first show that $X_{A, B}$ has at least as many connected components of $A$.

\begin{proposition}\label{prop:aconncomptoxab}
  For all $k \in \N$, it holds that if $A$ has $k$ connected components then so does $X_{A, B}$.
  \begin{proof}
    Let $C_1, \dots, C_k$ be the connected components of $A$. For all $s \in [k]$, we let
    \begin{equation*}
      P_s = \sum_{i \in C_s} \ket{i}\bra{i}.
    \end{equation*}
    Since the $C_s$ form a partition of $[n]$, the $P_s$ sum to the identity. If we plug in the definition of $P_s$ into the connected component condition, we get a sum with summands
    \begin{equation*}
      \input{diagrams/conn-comp-cond-xab-deriv.tikz}
    \end{equation*}
    for all $i \in C_s$ and $j \in C_t$. These are equal to
    \begin{equation*}
      \input{diagrams/conn-comp-cond-xab-deriv-2.tikz}
    \end{equation*}
    Since $C_s$ and $C_t$ are disjoint, we always have $i \neq j$, so the $B$ term vanishes. Since $i$ and $j$ are in different connected components by assumption, we have that $i \not\sim j$ and thus $A_{ij} = 0$. This makes the remaining terms vanish, which completes the proof.
  \end{proof}
\end{proposition}

The converse is slightly more involved.

\begin{proposition}\label{prop:aconncompfromxab}
  For all $k \in \N$, it holds that if $X_{A, B}$ has $k$ connected components then so does $A$.
  \begin{proof}
    Suppose that $P_1, \dots, P_k$ are projectors witnessing that $X_{A, B}$ has $k$ connected components. By the splitting principle of \cref{ssec:indepres}, we know that $P_1, \dots, P_k$ also form $k$ connected components of $X_{A, \ccdot}$. The projectors thus satisfy
    \begin{equation}\label{eq:conncompxaedgecond}
      \input{diagrams/conn-comp-cond-xa-simpl.tikz}
    \end{equation}
    for all $i, j \in [n]$, cf. \cref{eq:xacconncondsimpl}. This implies that whenever we have an edge $i \sim j$, that is $A_{ij} \neq 0$, then we must have
    \begin{equation*}
      P_s\ket{i} = 0 \qquad\textsc{or}\qquad P_t\ket{j} = 0.
    \end{equation*}

    Consider the following family of sets
    \begin{equation*}
      C_s \coloneqq \{i \in [n] \mid P_s \ket{i} \neq 0 \}.
    \end{equation*}
    First note that all the $C_s$ are non-empty. Indeed, suppose there was an $s \in [k]$ with $C_s = \varnothing$. Then $P_s\ket{i} = 0$ for all $i \in [n]$ and hence $P_s = 0$, contradicting the assumption that all $P_s$ are non-zero.
    Along similar lines, we can prove that the $C_s$ cover $[n]$. Assume for contradiction that there is some $i \in [n]$ such that $i \not\in C_s$ for all $s \in [k]$. Then $P_s \ket{i} = 0$ for all $s$, and since the $P_s$ sum to the identity, we get
    \begin{equation*}
      \ket{i} = \identity \ket{i} = \sum_{s \in [k]} P_s \ket{i} = 0,
    \end{equation*}
    contradiction. Next, we claim that the $C_s$ almost partition $[n]$.

    \begin{claim}\label{cl:conncompxab3}
      Let $s_1 \neq s_2 \in [k]$. If $i \in C_{s_1} \cap C_{s_2}$ then $i$ is an isolated vertex in $A$.
      \begin{subproof}
        Suppose for contradication that there exists some $i \in C_{s_1} \cap C_{s_2}$ that is not isolated. Then there exists some $j \neq i \in [n]$ such that $\{i, j\}$ is an edge in $A$. By \cref{eq:conncompxaedgecond}, we have
        \begin{align*}
          P_{s_1}\ket{i} &= 0 \qquad\textsc{or}\qquad P_{t}\ket{j} = 0 \qquad\text{for all } t \neq s_1\\
          P_{s_2}\ket{i} &= 0 \qquad\textsc{or}\qquad P_{t}\ket{j} = 0 \qquad\text{for all } t \neq s_2
        \end{align*}
        By assumption $P_{s_1}\ket{i} \neq 0$ and $P_{s_2}\ket{i} \neq 0$, so we have $P_t\ket{j} = 0$ for all $t \neq s_1$ and for all $t \neq s_2$. But $s_1 \neq s_2$, so we have $P_t\ket{j} = 0$ for all $t$, contradicting that the $P_s$ sum to the identity.
      \end{subproof}
    \end{claim}

    Most importantly, there do not exist edges between different $C_s$.

    \begin{claim}\label{cl:conncompxabnoedges}
      Let $i \neq j \in [n]$ and $s_1 \neq s_2 \in [k]$ such that $i \in C_{s_1}$ and $j \in C_{s_2}$. Then $A_{ij} = 0$.
      \begin{subproof}
        If $\{i, j\}$ was an edge in $A$, then by \cref{eq:conncompxaedgecond}
        \begin{equation*}
          P_{s_1}\ket{i} = 0 \qquad\textsc{or}\qquad P_{t}\ket{j} = 0
        \end{equation*}
        for all $t \neq s_1$. Since $i \in C_{s_1}$, we have $P_{s_1}\ket{i} \neq 0$ and thus $P_t\ket{j} = 0$ for all $t \neq s_1$. In particular this holds for $t = s_2 \neq s_1$, so $P_{s_2}\ket{j} = 0$, contradicting $j \in C_{s_2}$.
      \end{subproof}
    \end{claim}

    We now show that $A$ has at least $k$ connected components. Define the sets $B_s := C_s \setminus \bigcup_{t \neq s} C_t$ and the set $F \coloneqq [n] \setminus \bigcup_s B_s$. By \cref{cl:conncompxab3}, every non-isolated vertex of $C_s$ belongs to $B_s$, and every vertex in $F$ lies in at least two distinct $C_t$'s and is therefore isolated. Since the $B_s$ are pairwise disjoint and there are no edges between different $C_s$ by \cref{cl:conncompxabnoedges}, each non-empty $B_s$ is a union of connected components of $A$, and each vertex in $F$ forms an isolated connected component. It follows that $A$ has at least $m + |F|$ connected components, where $m$ is the number of non-empty sets $B_s$.

    It remains to show that $m + |F| \geq k$. For every $i \in B_s$ and $t \neq s$, we have $P_t\ket{i} = 0$ since $i \notin C_t$. Since $\sum_t P_t = \identity$, this gives $\ket{i} = P_s\ket{i}$, so $\ket{i} \in \img P_s$. The vectors $\{\ket{i} \mid i \in B_s\}$ are orthonormal, so $|B_s| \leq \rk P_s$. Setting $r_s \coloneqq \rk P_s \geq 1$, we have $\sum_s r_s = n$ and $\sum_s |B_s| + |F| = n$, so
        \begin{equation*}
          |F| = \sum_{s=1}^k (r_s - |B_s|) \geq \sum_{\substack{s \in [k] \\ B_s = \varnothing}} r_s \geq k - m.
        \end{equation*}
        Consequently, $A$ has at least $m + |F| \geq k$ connected components as desired.
  \end{proof}
\end{proposition}

Since \cref{prop:aconncomptoxab,prop:aconncompfromxab} hold for any $A$ and $B$---subject to the condition that $X_{A, B}$ be a quantum graph---we obtain the following corollaries.

\begin{corollary}
  For all $k \in \N$, it holds that $X_{A, \ccdot}$ has $k$ connected components if and only if $A$ does.
\end{corollary}

\begin{corollary}\label{cor:conncompxb}
  $X_{\ccdot, B}$ has $n$ connected components.
  \begin{proof}
    It is $X_{\ccdot, B} = X_{A, B}$ for $\mathring{A} = 0$, which is a graph with $n$ isolated vertices.
  \end{proof}
\end{corollary}

\subsubsection{ABC Graphs}
It turns out that \cref{prop:aconncomptoxab} can be generalised to ABC graphs using the same ideas.

\begin{proposition}\label{prop:conncompstrangetoxabc}
  A quantum graph of the form $X_{A, B, C}$ has at least as many connected components as the strange graph $\mathfrak{G}(A, C)$.
  \begin{proof}
    Let $C_1, \dots, C_k$ be a partition of $\mathfrak{G}(A, C)$ into $k$ connected components. We show that the projectors $P_1, \dots, P_k$ with
    \begin{equation*}
      P_s = \sum_{i \in C_s} \ket{i}\bra{i}
    \end{equation*}
    witness that $X_{A, B, C}$ also has $k$ connected components. The condition
    \begin{equation*}
      \input{diagrams/conn-comp-cond-xabc.tikz}
    \end{equation*}
    becomes
    \begin{equation*}
      \input{diagrams/conn-comp-cond-xabc-deriv-2.tikz}
    \end{equation*}
    for all $i \in C_s$ and $j \in C_t$. Since $C_s$, $C_t$ are disjoint, the middle term vanishes. Since $i$ and $j$ are in different connected components by assumption, we have that $i \not\sim j$ in $\mathfrak{G}(A, C)$ and thus $A_{ij} = C_{ji} = 0$.
  \end{proof}
\end{proposition}

However, there is no hope to generalise the converse direction. Take the Antisymmetric quantum graph $\Gasym$. We have shown that its corresponding strange graph is the complete graph, whose edges all have phase $\pi$. For $n = 2$ in particular, we get a single edge with phase $\pi$. This graph is obviously connected, but we have shown in \cref{prop:gasymconncomp} that for $n = 2$, $\Gasym$ has $2$ connected components. While $\Gasym$ is connected for $n \geq 3$, the problem persists for arbitrary $n$. The reason for this is given by the next lemma, which shows that connected components of two quantum graphs are connected components of their disjoint union.

\begin{definition}
    The disjoint union of two quantum graphs $G_1 = (M_k, A_1)$ and $G_2 = (M_\ell, A_2)$ is the quantum graph  $G_1 \sqcup G_2$ on $M_{k + \ell}$ with adjacency operator $A_1 \oplus A_2$.\footnote{In the literature, the vertex set of the disjoint union is usally defined as $M_k \oplus M_\ell$, but since we work with quantum graphs on $M_n$ we view the direct sum as a suitable subspace of $M_{k + \ell}$.}
\end{definition}

Recalling that
\begin{equation*}
  X_{A, B, C}^R = B \oplus \bigoplus_{i < j} \begin{pmatrix}
    A_{ij} & C_{ij} \\
    C_{ji} & A_{ji}
  \end{pmatrix},
\end{equation*}
it is not hard to see that the strange graph of disjoint union of two ABC graphs is given by the disjoint union of their strange graphs.

\begin{lemma}\label{lem:conncompsofdisjunion}
  Let $X_{A_1, \ccdot, C_1}, \dots, X_{A_k, \ccdot, C_k}$ be quantum graphs that respectively decompose into connected components $P_1^{(i)}, \dots, P_{m_i}^{(i)}$ for $i \in [k]$. Then there exists a decomposition of the disjoint union $\bigsqcup_i X_{A_i, \ccdot, C_i}$ into connected components $R_1^{(1)}, \dots, R^{(1)}_{m_1}, R^{(2)}_1, \dots, R^{(k)}_{m_k}$.
  \begin{proof}
    Without loss of generality suppose $X_{A_i, \ccdot, C_i}$ is a quantum graph on $M_{n_i}$. Let $n = \sum_i n_i$ and define $n_i^\downarrow = \sum_{j < i} n_j$ and $n_i^\uparrow = \sum_{j > i} n_j$ for all $i \in [k]$. The projectors witnessing the connected components of the disjoint union are then given by
    \begin{equation*}
      R^{(i)}_s = 0_{n_i^\downarrow} \oplus P^{(i)}_s \oplus 0_{n_i^\uparrow}.
    \end{equation*}
    Per assumption, $P_1^{(i)}, \dots, P_{m_i}^{(i)}$ sum to $\identity_{n_i}$, so $R_1^{(1)}, \dots, R^{(1)}_{m_1}, R^{(2)}_1, \dots, R^{(k)}_{m_k}$ sum to $\identity_n$. Now recall that the connected component condition is equivalent to
    \begin{equation*}
      \input{diagrams/conn-comp-cond-xabc-xac.tikz},
    \end{equation*}
    for all $i, j \in [n]$ and $s \neq t$. This is always satisfied if $i \not\sim j$ in the strange graph of $\bigsqcup_i X_{A_i, \ccdot, C_i}$. Suppose then that $i \sim j$. Then $i$ and $j$ must both be in the part of the strange graph belonging to some $X_{A_a, \ccdot, C_a}$. This implies that if any $P \in \{P_s, P_t\}$ is equal to $R^{(b)}_s$ for $b \neq a$, then $P\ket{i} = P\ket{j} = 0$ and the condition is satisfied. In the remaining case, where $P_s$ and $P_t$ are equal to $R^{(a)}_s$ and $R^{(a)}_t$ respectively, we recover one of the connected component conditions of $X_{A_a, \ccdot, C_a}$ (up to padding), which is satisfied by assumption.
  \end{proof}
\end{lemma}

We can use this lemma to prove the following proposition, showing that the equivalence between the number of connected components of $X_{A, B, C}$ and of the strange graph $\mathfrak{G}(A, C)$ fails for every $n$. The case for odd $n$ follows by taking the disjoint union with an isolated vertex.

\begin{proposition}\label{prop:conncompcounterex}
  For every even $n \in \N$ there exists a quantum graph on $M_n$ of the form $X_{A, \ccdot, C}$ such that $X_{A, \ccdot, C}$ has twice as many connected components as $\mathfrak{G}(A, C)$.
  \begin{proof}
    Consider the strange graph obtained by taking the disjoint union of $n / 2$ strange edges with phase $\pi$. The strange graph has $n/2$ connected components, while by \cref{lem:conncompsofdisjunion}, the corresponding quantum graph has $n$ connected components.
  \end{proof}
\end{proposition}

However, if $n \geq 3$ then we can still recover the equivalence for the question of connectedness. To show this, we first prove the following lemma, which might be interesting in its own right.

\begin{lemma}\label{lem:connectedequivclassicaledge}
  If $\mathfrak{G}(A, C)$ contains at least one classical edge, then $X_{A, B, C}$ is connected if and only if $\mathfrak{G}(A, C)$ is.
  \begin{proof}
    The forward direction follows contrapositively from \cref{prop:conncompstrangetoxabc}: If $\mathfrak{G}(A, C)$ is disconnected it has at least $2$ connected components, so $X_{A, B, C}$ also has at least $2$ connected components and is thus also disconnected.

    Conversely, suppose that $X_{A, B, C}$ is disconnected. Then in particular $X_{A, \ccdot, C}$ is disconnected. There thus exists a projector $P$ that satisfies
    \begin{equation*}
      \input{diagrams/connected-cond-xac-simpl.tikz},
    \end{equation*}
    cf. \cref{eq:xacconncondsimpl}. We show that there exists a decomposition of $\mathfrak{G}(A, C)$ into $2$ connected components. Consider the following sets of vertices
    \begin{align*}
      C &\coloneqq \{i \in [n] \mid P\ket{i} = \ket{i}\},\\
      C^\bot &\coloneqq \{i \in [n] \mid P\ket{i} = 0\},\\
      C^\sim &\coloneqq [n] \setminus (C \cup C^\bot).
    \end{align*}
    Clearly, $C, C^\bot, C^\sim$ partition $[n]$. We show that edges of $\mathfrak{G}(A, C)$ can only exist inside these sets. Let $i \sim j$. We have to exclude the following cases.
    \begin{itemize}
      \item $i \in C$, $j \in C^\bot$: In this case we have $P\ket{i} = \ket{i} \neq 0$ and $P^\bot\ket{j} = \ket{j}$, so the $A$-term is non-zero. This already excludes classical edges. We moreover have $P\ket{j} = 0$, so the $C$-term is zero and thus cannot cancel out the $A$ contribution. This excludes strange edges as well.
      \item $i \in C$, $j \in C^\sim$: In this case we have $P\ket{i} = \ket{i} \neq 0$ and $P^\bot\ket{j} \neq 0$, which excludes classical edges. We also have $P^\bot\ket{i} = 0$, so the $C$-term is zero. This excludes strange edges.
      \item $i \in C^\bot$, $j \in C^\sim$: Follows from the previous case by considering the condition with $P \leftrightarrow P^\bot$ exchanged.
    \end{itemize}
    The cases with $i$ and $j$ exchanged are follow by undirectedness.
    Finally, we show that there cannot be any classical edges inside $C^\sim$. Indeed, $i, j \in C^\sim$ implies $P\ket{i} \not\in \{0, \ket{i}\}$ and $P\ket{j} \not\in \{0, \ket{j}\}$, and thus also $P^\bot\ket{j} \not\in \{0, \ket{j}\}$. It follows that the $A$-term is non-zero and thus the condition unsatisfiable for classical edges.

    Now by assumption, $\mathfrak{G}(A, C)$ has at least one classical edge. This implies that either $C$ or $C^\bot$ must be non-empty. Without loss of generality assume $C$ is non-empty. Then $C$ and $C^\bot \cup C^\sim$ partition $[n]$ and by our previous observations there is no edge between the two sets. It follows that $\mathfrak{G}(A, C)$ is disconnected, which completes the proof.
  \end{proof}
\end{lemma}

Note that there is no hope to extend \cref{lem:connectedequivclassicaledge} to the number of connected components: The disjoint union of a classical $(n-2)$-clique and a strange edge of phase $\pi$ has $2$ connected components, but the corresponding quantum graph has at least $3$ by \cref{lem:conncompsofdisjunion}. We now prove that connectedness is preserved in general for $n \geq 3$.

\begin{proposition}\label{prop:connectedequivn3}
  For $n \geq 3$, a quantum graph of the form $X_{A, B, C}$ on $M_n$ is connected if and only if $\mathfrak{G}(A, C)$ is.
  \begin{proof}
    The forward direction again follows contrapositively. For the backward direction we may moreover assume by the previous proposition that $\mathfrak{G}(A, C)$ does not have any classical edges. Assume for contradiction that $\mathfrak{G}(A, C)$ is connected but $X_{A, B, C}$ has at least $2$ connected components. By the splitting principle, $X_{A, \ccdot, C}$ also has at least $2$ connected components. There thus exists a projector $P$ such that
    \begin{equation}\label{eq:xacconncondsimpln3}
      \input{diagrams/connected-cond-xac-simpl.tikz}
    \end{equation}
    for all $i, j \in [n]$. We claim that every such projector must have full-coordinate support. Otherwise there exists an $i \in [n]$ such that $P\ket{i} = 0$. Since the strange graph is connected and $P$ is non-zero, there also exists a $j \in [n]$ with $i \sim j$ such that $P\ket{j} \neq 0$. But then \cref{eq:xacconncondsimpln3} cannot be satisfied. Indeed, the $A$-term is $0$ since $P\ket{i} = 0$, but the $C$-term is not: Per assumption, $\mathfrak{G}(A, C)$ contains no classical edges, so $\mathring{C}_{ij} \neq 0$. At the same time $P\ket{j} \neq 0$ and $P^\bot\ket{i} = \ket{i} \neq 0$.

    Since per assumption $\mathfrak{G}(A, C)$ contains no classical edges, $\mathring{C}_{ij}$ is always non-zero, so we can rewrite \cref{eq:xacconncondsimpln3} as
    \begin{equation*}
      \input{diagrams/connected-cond-simpl-lhs.tikz} \quad=\quad e^{(\pi - \gamma_{ij})i}\quad \input{diagrams/connected-cond-simpl-rhs.tikz}\enspace,
    \end{equation*}
    where we substituted $\mathring{A}_{ji}$ and $\mathring{C}_{ij}$ by their values $1/2$ and $1/2$ times a phase, multiplied both sides by $2$, rotated the input wire to the output, and moved the $C$-term to the right-hand side. This requires that
    \begin{align}
      P\ket{i} &= \alpha P\ket{j}\label{eqs:pconds1}\\
      \overline{P^\bot}\ket{j} &= \beta \overline{P^\bot}\ket{i}\label{eqs:pconds2}
    \end{align}
    with $\alpha\beta = e^{(\pi - \gamma_{ij})i} \eqqcolon c_{ij}$.

    Choose an orthonormal basis $\psi_1, \dots, \psi_k$ of $\img P$ and complete it to an orthonormal basis of $\C^n$ by vectors $\psi_{k + 1}, \dots, \psi_n$. We define a unitary matrix $H$ by $H_{is} = \braket{\psi_s | i}$. The fact that $P$ has full coordinate support implies that every row of $H$ has at least one non-zero entry in the first $k$ columns.
    \cref{eqs:pconds1} now implies
    \begin{equation*}
      \sum_{s = 1}^k \braket{\psi_s | i} \ket{\psi_s} = P \ket{i} = \alpha P\ket{j} = \alpha \sum_{s = 1}^k \braket{\psi_s | j} \ket{\psi_s}.
    \end{equation*}
    Since the $\psi_s$ are orthogonal, this is only the case if
    \begin{equation*}
      H_{is} = \braket{\psi_s | i} = \alpha \braket{\psi_s | j} = \alpha H_{js}.
    \end{equation*}
    for all $s \leq k$. Similarly, by \cref{eqs:pconds2} we have
    \begin{equation*}
      \sum_{t = k + 1}^n \braket{\overline{\psi_t} | j} \ket{\overline{\psi_t}} = \overline{P^\bot} \ket{j} = \beta \overline{P^\bot}\ket{i} = \beta \sum_{t = k+1}^n \braket{\overline{\psi_t} | i} \ket{\overline{\psi_t}},
    \end{equation*}
    which implies
    \begin{equation*}
      \overline{H}_{jt} = \braket{\overline{\psi_t} | j} = \beta \braket{\overline{\psi_t} | i} = \beta \overline{H}_{it}.
    \end{equation*}
    for all $k < t \leq n$. We thus conclude that
    \begin{equation*}
      \frac{H_{is}}{H_{js}} = \frac{\overline{H}_{it}}{\overline{H}_{jt}} c_{ij} \qquad\textsc{and}\qquad \frac{H_{jt}}{H_{it}} = \frac{\overline{H}_{js}}{\overline{H}_{is}} \overline{c_{ij}}
    \end{equation*}
    for all $s \leq k$, $k < t \leq n$ whenever $i \sim j$. This means that if $i \sim j$, the rows $H_{i*}$ and $H_{j*}$ split into two parts: the first $k$ columns and the last $n - k$ columns, which are respectively scalar multiples of each other. Call this property (M). Since $\mathfrak{G}(A, C)$ is connected and $n \geq 3$, there are at least three vertices connected by a path. It follows that the corresponding rows must satisfy property (M). We show that no unitary matrix satisfies this constraint.

    \begin{claim}\label{cl:x2alwaysnonzero}
      Let $x, x'$ be a pair of rows of a unitary matrix that satisfy property (M). Suppose $x$ splits into $x_1$ and $x_2$. Then $x_2 \neq 0$.
      \begin{subproof}
        Suppose otherwise. Since $H$ is unitary, $x$ and $x'$ are orthogonal. We thus have
        $x'_1 x_1\adjoint = x'_1 x_1\adjoint + x'_2 x_2\adjoint = x' x\adjoint = 0$, but $x_1$ is a scalar multiple of $x'_1$ by property (M), so their inner product cannot be $0$.
      \end{subproof}
    \end{claim}

    \begin{claim}\label{cl:nosuchunitary}
      Let $k < n$ and $H$ be a unitary matrix such that every row has at least one non-zero entry in the first $k$ columns. Then $H$ does not contain three rows that pairwise satisfy property (M).
      \begin{subproof}
        Assume for contradiction that there are rows $u, v, w$ pairwise satisfying property (M). This means that
        \begin{align*}
          v_1 = \alpha u_1 &\quad\textsc{and}\quad v_2 = \beta u_2  \\
          w_1 = \gamma u_1 &\quad\textsc{and}\quad w_2 = \delta u_2
        \end{align*}
        Since $H$ is unitary, they must be mutually orthogonal. It follows that
        \begin{align*}
          \overline{\alpha} \norm{u_1}^2 + \overline{\beta} \norm{u_2}^2 = u_1 (\alpha u_1)\adjoint + u_2 (\beta u_2)\adjoint = uv\adjoint = 0 \\
          \overline{\gamma} \norm{u_1}^2 + \overline{\delta} \norm{u_2}^2 = u_1 (\gamma u_1)\adjoint + u_2 (\delta u_2)\adjoint = uw\adjoint = 0
        \end{align*}
        Set $\xi = \norm{u_1}^2 / \norm{u_2}^2$. Note that this is well-defined since $u_2 \neq 0$ by \cref{cl:x2alwaysnonzero}. It is moreover positive, since by assumption $u_1$ contains at least one non-zero entry. Substituting $\xi$ into the orthogonality equations yields
        \begin{equation}\label{eq:betadeltadefxi}
          \beta = -\xi\alpha \quad\textsc{and}\quad \delta = -\xi\gamma.
        \end{equation}
        Now consider the remaining orthogonality constraint, which yields
        \begin{equation*}
          \alpha\overline{\gamma} \norm{u_1}^2 + \beta\overline{\delta} \norm{u_2}^2 = \alpha u_1 \cdot (\gamma u_1)\adjoint + \beta u_2 \cdot (\delta u_2)\adjoint = vw\adjoint = 0.
        \end{equation*}
        By \cref{eq:betadeltadefxi} this is equivalent to
        \begin{equation*}
          \alpha \overline{\gamma} \norm{u_1}^2 + (-\xi\alpha)(-\xi \overline{\gamma})\norm{u_2}^2 = \alpha \overline{\gamma} \norm{u_1}^2 + \xi^2 \alpha\overline{\gamma} \norm{u_2}^2 = 0
        \end{equation*}
        Dividing both sides by $\alpha\overline{\gamma}\norm{u_2}^2$ leaves us with
        \begin{equation*}
          \xi + \xi^2 = 0,
        \end{equation*}
        but this is unsatisfiable for positive $\xi$, contradiction.
      \end{subproof}
    \end{claim}

    We have seen that \cref{eq:xacconncondsimpln3} is only satisfied for all $i, j \in [n]$ if the unitary matrix $H = (\braket{\psi_s | i})_{is}$ contains at least three rows that pairwise satisfy property (M). But $H$ satisfies the constraints of \cref{cl:nosuchunitary} and thus cannot contain such a collection of rows. This is a contradiction, so $X_{A, B, C}$ must be connected.
  \end{proof}
\end{proposition}

% Observe that \cref{prop:connectedequivn3} implies that it is really only isolated strange edges that the break the correspondence between the number of connected components of $X_{A, B, C}$ and the strange graph.

%% file: content/graphtheory/colouring.tex
We now move on to colouring. Recall that a quantum graph is $k$-colourable if there exist $k$ non-zero projectors $P_1, \dots, P_k$ on $\C^n$ that sum to the identity, such that
\begin{equation*}
  \input{diagrams/colouring-cond-mn.tikz}
\end{equation*}
for all $s \in [k]$. There is a trivial lower bound: Every quantum graph can be at most $1$ colourable, which is satisfied if and only if $G$ is the zero map. However, contrary to the case of connected components, there is no trivial upper bound on the number of colours that are needed. We will see that there are quantum graphs that are not colourable at all.

\subsubsection{ABC Graphs}

Since the colouring condition is very similar to that for connected components, it is unsurprising that we can also prove a splitting principle for colouring. In fact, it is proved the exact same way we did for connected components in \cref{ssec:indepres}, so we refrain from repeating the proof here. Merely the conclusion for the chromatic number is different. The splitting principle for colouring states that a collection of projectors is a colouring for $X_{A, B, C}$ if and only if it is a colouring for both $X_{A, \ccdot, C}$ and $X_{\ccdot, B}$. While for the number of connected components we were interested in the maximum possible number of projectors, the chromatic number is equal to the minimum number of projectors needed. We thus obtain the following result.

\begin{proposition}
  We have $\chi(X_{A, B, C}) \geq \max \{\chi(X_{A, \ccdot, C}), \chi(X_{\ccdot, B})\}$.
\end{proposition}

Similarly, an analogous argument as for connected components shows that every $k$-colouring of the strange graph $\mathfrak{G}(A, C)$ induces a $k$-colouring of $X_{A, \ccdot, C}$.

\begin{proposition}\label{prop:xabc-colouring-ub}
  For all $k \in \N$ it holds that if $\mathfrak{G}(A, C)$ is $k$-colourable then so is $X_{A, \ccdot, C}$. In particular, we have $\chi(X_{A, \ccdot, C}) \leq \chi(\mathfrak{G}(A, C))$.
  \begin{proof}
    Let $C_1, \dots, C_k$ be the colour classes of a $k$-colouring of $\mathfrak{G}(A, C)$. We set
    \begin{equation*}
      P_s = \sum_{i \in C_i} \ket{i}\bra{i}.
    \end{equation*}
    Then the $P_s$ are orthogonal projectors summing to the identity and the condition
    \begin{equation*}
      \input{diagrams/xac-colouring-cond.tikz}
    \end{equation*}
    becomes
    \begin{equation*}
      \input{diagrams/xac-colouring-cond-exp.tikz}\enspace.
    \end{equation*}
    Since $C_s$ is a colour class, we have $ii' \not\in E(\mathfrak{G}(A, C))$ for all $i \neq i' \in C_s$ and thus $A_{ii'} = C_{i'i} = 0$ as desired. For $i = i'$, we have $\mathring{A}_{ii'} = \mathring{C}_{ii'} = 0$ by definition. This completes the proof.
  \end{proof}
\end{proposition}

Note that unlike the corresponding result for connected components, this result only holds for ABC graphs with trivial $B$. This is related to the splitting principle: In the case of connected components we were lucky that our translation of the connected components of $\mathfrak{G}(A, C)$ also happened to be a decomposition into connected components of $X_{\ccdot, B}$. We have no such luck for colouring. While it would generally be possible that there exists a different translation that does work, the next section will show that this is not the case.

\subsubsection{AB Graphs}

The first observation we make is that AB graphs can be arbitrarily difficult to colour.

\begin{proposition}\label{prop:xabnotcolourable}
  There exist quantum graphs of the form $X_{A, B}$ that are not colourable.
  \begin{proof}
    Recall that the complete graph $K_n$ satisfies
    \begin{equation*}
      K_n = X_{\tfrac{nJ - I}{n}, \tfrac{nI - J}{n}}
    \end{equation*}
    and is thus of the form $X_{A, B}$. However, it is not colourable: The colouring condition reads
    \begin{equation*}
      \input{diagrams/kn-colouring-cond.tikz}\enspace,
    \end{equation*}
    but the left-hand side has rank $1$, while the right-hand side has rank $n \cdot \rk P_s$. This is unsatisfiable for $n \geq 2$.
  \end{proof}
\end{proposition}

Since every strange graph is at least $n$-colourable, this shows that \cref{prop:xabc-colouring-ub} cannot be generalised to non-trivial $B$. This is a symptom of the colourings for $X_{A, \ccdot, C}$ and $X_{\ccdot, B}$ generally being incompatible; it is not exclusively caused by $B$. Indeed, $X_{\ccdot, B}$ is always at least $n$-colourable.

\begin{proposition}\label{prop:xb-ncolourable}
  Every quantum graph of the form $X_{\ccdot, B}$ is $n$-colourable.
  \begin{proof}
    For $X_{\ccdot, B}$ the colouring condition becomes
    \begin{equation*}
      \input{diagrams/xb-colouring-cond.tikz},
    \end{equation*}
    which is equivalent to
    \begin{equation}\label{eq:xbcolouringderiv}
      \input{diagrams/xb-colouring-cond-deriv.tikz}.
    \end{equation}
    Indeed, the forward direction follows by bending down the top-left wire and plugging a unit in the bottom right. The backward direction follows from postcomposing with the comultiplication and bending the resulting top-right wire down, while bending the bottom-left wire up.
    Now let $v_1, \dots, v_n$ be the orthonormal DFT basis of $\C^n$, that is the $k$th entry of $v_j$ is $\exp(jk \cdot 2\mathrm{i}\pi/n)/\sqrt n$. Let $P_s = \ket{v_s}\bra{v_s}$. Plugging in $\overline{v_i}$ into the bottom left- and $v_j$ into the bottom right of \cref{eq:xbcolouringderiv} yields $0$ if $i \neq s$ or $j \neq s$, and $B \circ m(\overline{v_s} \otimes v_s)$ otherwise. Since $m$ is entrywise multiplication on $\C^n$ and the entries of $v_s$ have norm $1$, we have $m(\overline{v_s} \otimes v_s) = \mathbf{1}$. Since we assume $B$ to be loopless, we have $B\mathbf{1} = 0$, and thus $B \circ m(\overline{v_s} \otimes v_s) = 0$. Since both the $\overline{v_i}$ and the $v_i$ form a basis, we conclude that \cref{eq:xbcolouringderiv} holds.
  \end{proof}
\end{proposition}

In this case, the colourability is limited by the dimension of the kernel of $B$.

\begin{proposition}
  It holds that $\chi(X_{\ccdot, B}) \geq n / \dim\ker B$.
  \begin{proof}
    Recall that the colouring condition for $X_{\ccdot, B}$ is equivalent to
    \begin{equation*}
      \input{diagrams/xb-colouring-cond-deriv.tikz}.
    \end{equation*}
    Let $T_s = \img P_s$. The condition requires that the subspaces $T_s^2 = \lspan \{m(\overline{v} \otimes w) \mid v, w \in T_s\}$ must be contained in the kernel of $B$. Now suppose there existed a $k$-colouring $P_1, \dots, P_k$ of $X_{\ccdot, B}$. Since the $P_s$ sum to the identity, there needs to exist an $s$ such that $\rk P_s \geq n/k$. We claim that $\dim T_s^2 \geq \dim T_s$, and thus $\dim \ker B \geq n/k$.

    Without loss of generality assume that for all $j \in [n]$ there exists a $v \in T_s$ such that $v_j \neq 0$. Otherwise we consider $T_s$ as a subspace of $\C^m$ for some suitable $m < n$.

    \begin{claim}
      There exists a $w \in T_s$ such that $w_j \neq 0$ for all $j \in [n]$.
      \begin{subproof}
        Choose some arbitrary $w^{(0)} \in T_s$. Let $i$ be minimal with $w^{(0)}_i = 0$. By assumption, there exists a $u \in T_s$ such that $u_i \neq 0$. Let
        \begin{equation*}
          \alpha \coloneqq \max_i \abs{w^{(0)}_i}
          \qquad\quad
          \beta \coloneqq \min \{\abs{u_i} \mid u_i \neq 0\}
          \qquad\quad
          \gamma \coloneqq \inf \{r \in R^+ \mid r\beta > \alpha\},
        \end{equation*}
        and set $w^{(1)} = w^{(0)} + ru$. Then $w^{(1)} \in T_s$, $w^{(1)}_i \neq 0$ and for all $j \in [n]$, $w^{(0)}_j \neq 0$ implies $w^{(1)}_j \neq 0$. Iterating this procedure for at most $n$ steps yields the desired $w$.
      \end{subproof}
    \end{claim}

    Let $w$ be as in the claim and observe that $T_s^2 \supseteq \{m(\overline{v} \otimes w) \mid v \in T_s\}$. Recall that $m$ is componentwise multiplication, and thus $m(\overline{v} \otimes w) = \overline{v} \bullet w$. This in turn is equal to $W\overline{v}$ for $W = \diag w$. Now $\dim T_s = \dim \overline{T_s}$ and $W$ is invertible. It follows that
    \begin{equation*}
      \dim T_s^2 \geq \dim \{m(\overline{v} \otimes w) \mid v \in T_s\} = \dim \{Wv \mid v \in \overline{T_s}\} = \dim \overline{T_s} = \dim T_s
    \end{equation*}
    as desired. As shown above, this implies that $\dim \ker B \geq \dim T_s^2 \geq \dim T_s \geq n/k$ for any $k$-colouring. In particular, for an optimal $\chi(X_{\ccdot, B})$-colouring, we get
    \begin{equation*}
      \chi(X_{\ccdot, B}) \geq \frac{n}{\dim\ker B}
    \end{equation*}
    as desired.
  \end{proof}
\end{proposition}

For $A$ alone, we recover precisely the chromatic number.

\begin{proposition}
  For all $k \in \N$, it holds that $X_{A, \ccdot}$ is $k$-colourable if and only if $A$ is.
  \begin{proof}
    \cref{prop:xabc-colouring-ub} and \cref{prop:xastrangeisa} establish that $X_{A, \ccdot}$ is $k$-colourable if $A$ is. Conversely, suppose that $P_1, \dots, P_k$ are a $k$-colouring of $X_{A, \ccdot}$. Then per definition, we have
    \begin{equation*}
      \input{diagrams/xa-colouring-cond.tikz}
    \end{equation*}
    for all $s \in [k]$. This is equivalent to
    \begin{equation*}
      \input{diagrams/xa-colouring-cond-ij.tikz}
    \end{equation*}
    for all $i, j \in [n]$. We can simplify the (co)multiplications, which yields
    \begin{equation*}
      \input{diagrams/xa-colouring-cond-ij-rewr.tikz}.
    \end{equation*}
    This in turn is only satisfied if $A_{ij} \neq 0$ implies $P\ket{i} = 0$ or $P\ket{j} = 0$. Let $C_1, \dots, C_k \subseteq [n]$ be defined as
    \begin{equation*}
      C_s = \compl{\{i \in [n] \mid P_s\ket{i} = 0\}}
    \end{equation*}
    The $C_s$ are valid colour classes: Assume for contradiction that there were $i, j \in C_s$ with $i \sim j$. Then $A_{ij} \neq 0$ and thus $P_s\ket{i} = 0$ or $P_s\ket{j} = 0$. This implies that $i \not\in C_s$ or $j \not\in C_s$, contradicting our assumption. They also cover $[n]$. Suppose there was some $x \in [n]$ such that $x \not\in C_s$ for all $s$. Then we have $P_s\ket{x} = 0$ for all $s$ and thus
    \begin{equation*}
      \ket{x} = \identity\ket{x} = \sum_s P_s\ket{x} = 0 \neq \ket{x}.
    \end{equation*}
    We have proved that there exist colour classes $C_1, \dots, C_k$ that jointly cover $[n]$. To obtain a graph colouring, we let
    \begin{equation*}
      \tilde{C}_s = C_s \setminus \bigcup_{t > s} C_t
    \end{equation*}
    Then the $\tilde{C}_s$ form a partition of $[n]$ and we have $\tilde{C}_s \subseteq C_s$, so the colouring condition is still satisfied. Consequently, $\tilde{C}_1, \dots, \tilde{C}_k$ is a $k$-colouring of $A$.
  \end{proof}
\end{proposition}

\subsubsection{The Symmetric and Antisymmetric Cases}

We now turn our attention to the Symmetric and Antisymmetric quantum graphs. The Antisymmetric quantum graph turns out to be $n$-colourable, which we show by another rank-based argument.

\begin{proposition}
  It holds that $\chi(\Gasym) = n$.
  \begin{proof}
    In \cref{prop:qgraphs-as-xabc} we have seen that $\Gasym$ is of the form $X_{A, \ccdot, C}$, and by \cref{prop:gsymgasymstrange} its corresponding strange graph is the complete graph of strange edges with phase $\pi$. \cref{prop:xabc-colouring-ub} thus implies that $\chi(\Gasym) \leq n$. On the other hand, the colouring condition for $\Gasym$ is equivalent to
    \begin{equation*}
      \input{diagrams/gasym-colouring-cond.tikz}\enspace.
    \end{equation*}
    The left-hand side has rank $1$, while the right-hand side has rank $(\rk P_s)^2$. The only way to satisfy the equation is thus when $\rk P_s = 1$ for all $s$, which implies that $\chi(\Gasym) \geq n$.
  \end{proof}
\end{proposition}

  For the Symmetric quantum graph we find a similar situation as for the connected components of $\Gasym$, where $n = 2$ is a special case. We also find that for $n \geq 3$, $\Gsym$ is not colourable at all.

\begin{proposition}
  For $n = 2$, $\Gsym$ is $2$-colourable. For $n \geq 3$, $\Gsym$ is not colourable.
  \begin{proof}
    The colouring condition for $\Gsym$ is equivalent to
    \begin{equation*}
      \input{diagrams/gsym-colouring-cond.tikz}\enspace.
    \end{equation*}
    Note that the right-hand side has rank $n \cdot \rk P_s$, while the left-hand side has rank at most $1 + (\rk P_s)^2$. An elementary computation shows that this is only compatible if $n = 2$ or $\rk P_s = n$. Suppose $n \neq 2$ and $\rk P_s = n$. The only projector of rank $n$ is the identity, in which case the right-hand side becomes a multiple of the identity on $\C^n \otimes \C^n$, while the left-hand side becomes the projector onto the symmetric subspace of $\C^n \otimes \C^n$. These are obviously different maps, so the equality cannot be satisfied for $n \neq 2$. It follows that $\Gsym$ is not colourable for $n \geq 3$.

    On the other hand, suppose $n = 2$. We know that $P_s \not\in \{0, \identity\}$, so we must have $\rk P_s = 1$. This means that we can rewrite the condition equivalently as
    \begin{equation*}
      \input{diagrams/gsym-colouring-cond-rk1.tikz}\enspace.
    \end{equation*}
    Moving wires and simplifying yields
    \begin{equation*}
      \input{diagrams/gsym-colouring-cond-rk1-rewr.tikz}\enspace.
    \end{equation*}
    We can factor the $P_s$ term, in which case the condition simplifies to
    \begin{equation*}
      \frac{1}{2} (\ket{\overline{\psi_s}}\bra{\overline{\psi_s}} + \ket{\psi_s}\bra{\psi_s}) = \frac{1}{2} \identity,
    \end{equation*}
    which is satisfied for $\psi_s = (\ket{1} + i\ket{2})/\sqrt 2$.
  \end{proof}
\end{proposition}

\subsubsection{The Complete and Empty Graphs}

The cases of the complete and empty graphs are now straightforward. Since the adjacency operator of $\compl{K_n}$ is the zero map, the identity satisfies the colouring condition, which implies that $\chi(\compl{K_n}) = 1$. On the other hand, we have shown in \cref{prop:xabnotcolourable} that the complete graph $K_n$ is not colourable at all. This concludes our investigation of the chromatic number.

\begin{proposition}
  We have $\chi(\compl{K_n}) = 1$, while $K_n$ is not colourable.
\end{proposition}

\subsubsection{Notes on (Non-)Colourability of Quantum Graphs}

We have seen that there exist quantum graphs that are not classically colourable. On first glance, this might seem like a flaw in the definition of $k$-colourability. We argue that this is not the case: the obstruction is intrinsic to the non-commutativity of the quantum edge space. A classical $k$-colouring requires orthogonal projections $P_1, \dots, P_k$ summing to the identity and satisfying $P_sXP_s = 0$ for every $X \in S$. For quantum graphs whose adjacency operator has large rank (see \cref{prop:xabnotcolourable}, where the colouring condition leads to an irreconcilable rank mismatch) no non-zero projection can satisfy this simultaneously for all edges. This is natural: In the classical setting, adjacency matrices have entries in $\{0,1\}$ and one can always assign one colour per vertex. For quantum graphs, however, the edge space $S$ can be non-commutatively too dense for any classical partition of the identity to separate.

In \cite{brannan_quantumclassical_2022}, see also \cite{ganesan_spectral_2023}, the authors also propose a definition of quantum colouring. While the colouring notion we used so far corresponds to the existence of a \emph{classical} winning strategy in a suitably chosen non-local game played on the quantum graph, quantum colourings correspond to \emph{quantum} strategies. More precisely, the classical chromatic number $\chi$ restricts to an ancilla algebra $\mathcal{N} = \mathbb{C}$, whereas the quantum chromatic number $\chi_q$ allows any finite-dimensional von Neumann algebra $\mathcal{N}$. The entanglement encoded in $\mathcal{N}$ provides the players with shared quantum resources, enabling correlations that are impossible classically and making every quantum graph on $M_n$ finitely colourable.

\begin{proposition}[Theorem 6.6. in \cite{brannan_quantumclassical_2022}]
  For every quantum graph $G$ on $M_n$, we have $\chi_q(G) \leq n^2$.
\end{proposition}

The bound $n^2$ is tight: It holds that $\chi_q(K_n) = n^2 = \dim M_n$ \cite{brannan_quantumclassical_2022, ganesan_spectral_2023}. By constrast, we have seen in \cref{prop:xabnotcolourable} that $K_n$ is not classically colourable. In particular, $K_n$ exhibits a strict separation
\begin{equation*}
  \chi(K_n) = \infty \;>\; n^2 = \chi_q(K_n),
\end{equation*}
demonstrating that quantum entanglement yields a decisive advantage in the colouring game for quantum graphs. The $n^2$ projections achieving the quantum colouring are given by a ``shift and multiply'' unitary error basis for $M_n$, see \cite[Theorem 6.6]{brannan_quantumclassical_2022} for details.

%% file: content/graphtheory/indepsets.tex
Recall that the notion of independent sets for quantum graphs was introduced in \cref{def:indepsets} via the operator space point of view: A \emph{$k$-independent set} of a quantum graph $G$ is a projector $P \in M_n$ of rank $k$ such that $PSP \subseteq \C P$, where $S$ is the operator space associated to $G$. Clearly, any rank one projector $P = \ket{v}\bra{v}$ has this property, hence every quantum graph $G$ has $\alpha(G) \geq 1$.

\medskip

Before moving on to the study of the independent sets of quantum graphs with symmetry, let us comment on the fact that the independent set condition from \cref{def:indepsets} is very close to the notion of \emph{isotropic subspace} introduced in \cite{bei2021independent}: A subspace $U \subseteq \mathbb C^n$ is an isotropic subspace of an alternating matrix space
\begin{equation*}
  \mathcal A \subseteq \{X \in M_n \, : \, X^\top = -X\}
\end{equation*}
if, for any $A \in \mathcal A$ and any $u,u' \in U$, we have $u\transpose A u' = 0$. Comparing this definition with the independent set definition in our work, we spot two important differences:
\begin{itemize}
  \item On the one hand, we only require that $PSP$ is a multiple of $P$, not necessarily $P S P = 0$ as in the definition of the isotropic subspace from \cite{bei2021independent}.
  \item On the other hand, we work exclusively in the complex case scenario, so the forms we are considering are sesquilinear, not bilinear: We would work with conditions of the form $u\adjoint A u' = 0$ instead of $u\transpose A u' = 0$.
\end{itemize}
Although not strictly equivalent, these conditions are closely related: In \cite[Theorem 1.3]{bei2021independent}, the authors show that the independence number of a graph $G$ is equal to the maximum dimension of an isotropic subspace of the alternating matrix space $\mathcal A(G) = \{\ket{i}\bra{j} - \ket{j}\bra{i} \mid i \sim j\}$, a result similar to our \cref{prop:alpha-equal-strange-AC}. Note that $\mathcal{A}(G)$ is precisely the operator space associated to $X_{G/2, \ccdot, -G/2}$. A similar comment applies to the case of the coloring graph parameters.

\subsubsection{ABC Graphs}

Just as for colourings and connected components, it turns out that the independent set condition for $X_{A, B, C}$ splits into independent conditions for the $AC$ part and the $B$ part. Indeed, a projector $P$ witnesses an independent set of $X_{A, B, C}$ if
\begin{equation*}
  \input{diagrams/xabc-indepset-cond.tikz}\enspace.
\end{equation*}
for all $x \in M_n$. Plugging in the standard basis for $x$ and simplifying, we get
\begin{equation*}
  \input{diagrams/xabc-indepset-cond-exp.tikz}\enspace.
\end{equation*}
For $i \neq j$, the $B$-term is zero and we recover the independent set condition for $X_{A, \ccdot, C}$. For $i = j$, we have $\mathring{A}_{ji} = \mathring{C}_{ij} = 0$, and we recover the independent set condition for $X_{\ccdot, B}$. It thus follows that an independent set for $X_{A, B, C}$ must simultaneously be an independent set of $X_{A, \ccdot, C}$ and $X_{\ccdot, B}$. We thus get the following result.

\begin{proposition}\label{prop:xabcindepsetjointbounds}
  We have $\alpha(X_{A, B, C}) \leq \min\{\alpha(X_{A, \ccdot C}), \alpha(X_{\ccdot, B}))\}$.
\end{proposition}

For the independence number, this result is particularly useful, since we can determine $\alpha(X_{A, \ccdot, C})$ exactly.

\begin{proposition}\label{prop:indepnumdetbystrange}
  For quantum graphs of the form $X_{A, \ccdot, C}$ we have $\alpha(X_{A, \ccdot, C}) = \alpha(\mathfrak{G}(A, C))$.
  \begin{proof}
    We first show that $\alpha(X_{A, \ccdot, C}) \geq \alpha(\mathfrak{G}(A, C))$. We show the stronger results that every independent set of $\mathfrak{G}(A, C)$ induces an independent set of $X_{A, \ccdot C}$ of the same size. By realigning, we find that the projector onto the operator space corresponding to $X_{A, \ccdot, C}$ is given by
    \begin{equation*}
      \input{diagrams/xac-proj.tikz}.
    \end{equation*}
    Now let $C$ be an independent set of $\mathfrak{G}(A, C)$ of size $k$ and let
    \begin{equation*}
      P = \sum_{c \in C} \ket{c} \bra{c}.
    \end{equation*}
    Then $P$ is a rank-$k$ projector and the independent set condition becomes
    \begin{equation*}
      \input{diagrams/xac-indepset-cond.tikz}\enspace,
    \end{equation*}
    which is equivalent to
    \begin{equation}\label{eq:xac-indepset}
      \input{diagrams/xac-indepset-cond-rewr.tikz}\enspace.
    \end{equation}
    Now suppose $c \neq c'$. Since $c, c'$ are part of an independent set, we have $c \not\sim c'$ and thus $\mathring{A}_{c'c} = \mathring{C}_{c'c} = 0$. The same holds for $c = c'$, since $\mathring{A}$ and $\mathring{C}$ have no diagonal. It follows that \cref{eq:xac-indepset} is satisfied for $c_x = 0$, and thus $P$ witnesses an independent set of $X_{A, \ccdot, C}$.

    Now suppose $P$ is a rank-$k$ projector witnessing an independent set of $X_{A, \ccdot, C}$.
    Without loss of generality, we assume that $k > 1$.
    We show that there exists a corresponding independent set of $\mathfrak{G}(A, C)$.
    To see this, we rewrite the independent set condition by plugging in a basis of $M_n$ for $x$, concretely the outer products formed by the standard basis of $\C^n$.
    \begin{equation*}
      \input{diagrams/xac-indepset-cond-ij.tikz}
    \end{equation*}
    Simplifying the (co)multiplications and moving tensor legs yields
    \begin{equation}\label{eq:indepsetijrewr}
      \input{diagrams/xac-indepset-cond-ij-rewr.tikz}\enspace.
    \end{equation}
    We now choose our independent set as $X = \compl{\{i \in [n] \mid P\ket{i} = 0 \}}$. Then we have $\abs{X} \geq n - \dim \ker P = \rk P = k$.  Suppose then that $i \sim j$ in $\mathfrak{G}(A, C)$. If \cref{eq:indepsetijrewr} is satsfied for $c_{ij} = 0$ then it is straightforward to see that either $P\ket{i} = 0$ or $P\ket{j} = 0$ and thus $i \not\in X$ or $j \not\in X$. Suppose then that the equation is satisfied for $c_{ij} \neq 0$. Dividing by $c_{ij}$ yields
    \begin{equation*}
      \alpha P\ket{i}\bra{j}P + \beta P\ket{j}\bra{i}P = P.
    \end{equation*}
    Since $P$ is a projector, we must have
    \begin{align*}
      (\alpha P\ket{i}\bra{j}P + \beta P\ket{j}\bra{i}P)^2 &= \alpha^2 P_{ji} P\ket{i}\bra{j}P + \alpha\beta P_{jj} P\ket{i}\bra{i}P + \beta\alpha P_{ii} P\ket{j}\bra{j}P + \beta^2 P_{ij} P\ket{j}\bra{i}P\\
                                                           &= \alpha P\ket{i} \big(\alpha P_{ji} \bra{j}P + \beta P_{jj} \bra{i}P\big) + \beta P\ket{j} \big(\beta P_{ij} \bra{i}P + \alpha P_{ii} \bra{j}P\big)\\
                                                           &\overset{!}{=} \alpha P\ket{i} \bra{j}P + \beta P \ket{j}\bra{i}P
    \end{align*}
    Since $i \sim j$, we know that $\alpha = \mathring{A}_{ji} / c_{ij} \neq 0$. The equation is thus only satisfied if $\beta = 0$, $P_{jj} = 0$, or $\bra{i}P = 0$. If $\bra{i}P = 0$ we have $P\ket{i} = 0$ and thus $i \not\in X$. If $P_{jj} = 0$, then $\norm{P\ket{j}}^2 = \bra{j}PP\ket{j} = \bra{j}P\ket{j} = 0$ and thus $P\ket{j} = 0$ and $j \not\in X$. Suppose then that $\beta = 0$. Then we have
    \begin{equation*}
      P = \alpha P\ket{i}\bra{j}P,
    \end{equation*}
    which implies that $\rk P = 1$, contradicting our assumption. It follows that $X$ is indeed an independent set, which completes the proof.
  \end{proof}
\end{proposition}

\subsubsection{AB Graphs}

An immediate corollary of \cref{prop:indepnumdetbystrange} is that the independence number of $X_{A, \ccdot}$ coincides with that of $A$.

\begin{proposition}\label{prop:alpha-equal-strange-AC}
  It holds that $\alpha(X_{A, \ccdot}) = \alpha(A)$.
  \begin{proof}
    By \cref{prop:xastrangeisa} we have $A \cong \mathfrak{G}(A, \ccdot)$ and by \cref{prop:indepnumdetbystrange} we have $\alpha(X_{A, \ccdot}) = \alpha(\mathfrak{G}(A, \ccdot))$.
  \end{proof}
\end{proposition}

While we cannot determine the independence number of $X_{\ccdot, B}$ precisely, we can give strong bounds. We will see that the bounds are not only tight, but the upper and lower bounds coincide for certain $X_{\ccdot, B}$. Concretely, it turns out that the independence number is closely related to both the rank of $B$ and the matrix representation of $B$ in the standard basis.

\begin{definition}
  Let $X$ be an $n \times n$ matrix. We denote by
  \begin{equation*}
    \operatorname{EqRows}(X) \coloneqq \max \{k \in [n] \mid \exists j_1 <  \cdots < j_k \in [n], X_{j_1*} = X_{j_2*} = \dots = X_{j_k*} \}
  \end{equation*}
  the maximum number of equal rows of $X$.
\end{definition}

Towards proving our bounds, we begin by rephrasing the independent set condition for $X_{\ccdot, B}$. By realigning the adjacency operator, we find that the projector onto the corresponding operator space is given by
\begin{equation*}
  \input{diagrams/xb-proj.tikz}.
\end{equation*}
The independent set condition thus becomes
\begin{equation*}
  \input{diagrams/xb-indepset-cond.tikz}\enspace.
\end{equation*}
Like we did for $X_{A, \ccdot, C}$, we can plug in the standard basis for $x$, which yields
\begin{equation*}
  \input{diagrams/xb-indepset-cond-rewr.tikz}\enspace.
\end{equation*}
For $i \neq j$, this is satisfied for $c_{ij} = 0$. Otherwise, by bending the right wire downwards, we find that the condition is equivalent to
\begin{equation}\label{eq:xbindebdi}
  \input{diagrams/xb-indepset-cond-leftmul.tikz}\enspace.
\end{equation}
Let $D_i \colon \C^n \to \C^n$ be left-multiplication by $B\ket{i}$. Then in the standard basis, we have $D_i = \diag(B\ket{i})$, and equation \cref{eq:xbindebdi} becomes
\begin{equation}\label{eq:xbindepdirest}
  PD_iP = \beta_i P.
\end{equation}

This phrasing of the independence condition suggests to construct $P$ explicitly as follows.

\begin{proposition}
  Let $B$ be a projector. Every set of equal rows of $B$ induces an independent set of $X_{\ccdot, B}$.
  \begin{proof}
    Suppose $B$ contained $k$ equal rows, indexed by $j_1, \dots, j_k$. This means that $B_{j_s i} = B_{j_{t} i}$ for all $i \in [n]$ and $s, t \in [k]$. The rank-$k$ projector
    \begin{equation*}
      P = \sum_{s \in [k]} \ket{j_s} \bra{j_s}
    \end{equation*}
    then satisfies
    \begin{equation*}
      P D_i P = \sum_{s,t \in [k]} \ket{j_s}\braket{j_s | D_i | j_{t}}\bra{j_{t}} = \sum_{s \in [k]} B_{j_si} \ket{j_s}\bra{j_s} = \beta_i P
    \end{equation*}
    as desired.
  \end{proof}
\end{proposition}

We hence get the following corollary regarding the independence number of $X_{\ccdot, B}$.

\begin{corollary}\label{cor:eqrowsindepbound}
  We have $\alpha(X_{\ccdot, B}) \geq \operatorname{EqRows}(B)$.
\end{corollary}

We can also lower bound $\alpha(X_{\ccdot, B})$ in terms of $\rk B$. The proof is a nice application of \emph{Tverberg's theorem} \cite{tverberg1966generalization}, which states that any set of at least $1+(k-1)(r+1)$ vectors in $\R^r$ can be partitioned into $k$ groups whose convex hulls share a common point. For what follows, it will be convenient to rephrase \cref{eq:xbindepdirest} in terms of isometries instead of projectors. Concretely, $X_{\ccdot, B}$ has a $k$-independent set if and only if there exists an isometry $V\colon \C^k \to \C^n$ such that
\begin{equation}\label{eq:xbindepirestiso}
  V\adjoint D_i V = \beta_i \identity_k
\end{equation}
\cref{eq:xbindepdirest} follows from \cref{eq:xbindepirestiso} by multiplying both sides from the left with $V$ and from the right with $V\adjoint$. Conversely, any isometry $V$ with $\img V = \img P$ satisfies $P = VV\adjoint$, so multiplying \cref{eq:xbindepdirest} from the left with $V\adjoint$ and from the right with $V$ recovers \cref{eq:xbindepirestiso}.

% In the case of the quantum graph $X_{\cdot, B}$, the corresponding operator system contains only diagonal matrices, with the vectors in the image of $B$ on the diagonal. We can immediately apply \cite[Lemma 2.9]{weaver2019quantum} to our setting to obtain the upper bound below. Note that the proof of this result makes use of \emph{Tverberg's theorem},  Tverberg's theorem \cite{tverberg1966generalization} is a central result in combinatorics \cite{barany2016tverbergs} that has applications in many areas, including quantum graphs \cite{weaver2017quantum,weaver2019quantum} and quantum information theory \cite{knill2000theory, li2011generalized,loulidi2021compatibility}.

\begin{proposition}\label{prop:indep-number-LB}
  It holds that
  \begin{equation*}
    \alpha(X_{\ccdot, B}) \geq 1 + \left\lfloor \frac{n-1}{\rk B + 1} \right\rfloor.
  \end{equation*}
  \begin{proof}
    First, note that since we assume $X_{\ccdot, B}$ to be undirected, it holds by \cref{prop:xabc-qgraph-char} that $\overline{B} = B$. Let $r \coloneqq \rk B$ and let $k$ be the right hand side of the ineqality in the statement. Our goal is to construct an isometry $V: \C^k \to \C^n$ such that
    \begin{equation}\label{eq:isometry-indep-set}
      V^\dag \diag(B_{*i})V = \beta_i I_k \qquad \forall i \in [r],
    \end{equation}
    for some $\beta_i \in \R$; actually we are going to construct below a \emph{real} isometry $V$. The condition above implies, for all $s \in [k]$,
    $$\sum_{j=1}^n B_{ji}\abs{V_{js}}^2 = \beta_i \qquad \forall i \in [r].$$
    Since for a fixed $s$, the vector $(|V_{js}|^2)_{j \in [n]}$ is a probability vector, we can interpret the equation above as $\beta \in \R^r$ lying in the convex hull of the vectors $B_{j*} \in \R^r$ for $j \in [n]$.

    To construct the (real) orthonormal vectors $v_s$ and the vector $\beta \in \R^r$, we are going to apply Tverberg's theorem. From the assumption in the statement we have $n \geq 1+(k-1)(r+1)$, hence we can partition the set $[n]$ into $k$ parts
    \begin{equation*}
      [n] = J_1 \sqcup J_2 \sqcup \cdots \sqcup J_k
    \end{equation*}
    such that
    \begin{equation*}
      \bigcap_{s=1}^k \operatorname{conv} \{ B_{j*} \mid j \in J_s \} \neq \emptyset.
    \end{equation*}
    Let $\beta \in \R^r$ be an arbitrary element of the intersection above. For all $s \in [k]$, we can write
    \begin{equation*}
      \beta_i = \sum_{j \in J_s} p_j^{(s)}B_{ji} \qquad \forall i \in [r]
    \end{equation*}
    for some probability vector $p^{(s)}$. Now we define the isometry $V$ as follows: for all $s \in [k]$ and $j \in [n]$,
    \begin{equation*}
      V_{js} \coloneqq \begin{cases}
        \sqrt{p^{(s)}_j} &\qquad \text{ if } j \in J_s\\
        0 &\qquad \text{ if } j \notin J_s.
    \end{cases}
    \end{equation*}
    First, note that the $k$ column vectors $V_{*s} \in \R^n$ have unit norm by construction. Moreover, for $s \neq t$, it holds that $V_{*s} \perp V_{*t}$ because they have disjoint supports. For any $s \in [k]$ we have
    \begin{equation*}
      [V\adjoint \diag(B_{*i}) V]_{ss} = \sum_{j=1}^n B_{ji}|V_{js}|^2 = \sum_{j \in J_s} B_{ji} p^{(s)}_j = \beta_i \qquad \forall i \in [r].
    \end{equation*}
    Using one more time the disjoint support property, we have for $s \neq t \in [k]$:
    \begin{equation*}
      [V\adjoint \diag(B_{*i}) V]_{st} = \sum_{j=1}^n B_{ji} V_{js} V_{jt} = 0.
    \end{equation*}
    We have constructed thus an isometry as in \eqref{eq:isometry-indep-set}, finishing the proof.
  \end{proof}
\end{proposition}

We remark that the proof is conceptually very similar to that of Lemma 2.9 in \cite{weaver2019quantum}. In fact, we could have applied their result directly to obtain the bound: As required by their lemma, the operator systems corresponding to quantum graphs of the form $X_{\ccdot, B}$ consist only of diagonal matrices. The only difference is that we need to account for the fact that our graphs do not have loops, which results in the $+ 1$ in the denominator. We now proceed to prove a simple upper bound for the independence number that also depends on the rank of the matrix $B$.

\begin{proposition}\label{prop:indep-number-UB}
  It holds that $\alpha(X_{\cdot, B}) \leq n - \rk B$.
    \begin{proof}
        Let $V\colon\C^k \to \C^n$ be an isometry witnessing the independence number of the quantum graph $X_{\ccdot, B}$ with $k = \alpha(X_{\ccdot, B})$. Let $b_1, \ldots, b_r \in \R^n$ be vectors that span the image of the projection $B$. The isometry $V$ satisfies
        \begin{equation}\label{eq:indep-number-isometry}
            \forall i \in [r] \ \exists \beta_i \in \R \qquad V^* \diag(b_i) V = \beta_i I_k.
        \end{equation}
        Writing $c_i := b_i - \beta_i \mathbf{1}$, the relation above can be re-written as $V^* \diag(c_i) V = 0$ for $i \in [r]$. Setting $\mathcal V := \img V$, we can rephrase the condition as
        \begin{equation*}
          \diag(c_i) \mathcal V \subseteq \mathcal V^\perp.
        \end{equation*}
        Pick a vector $v \in \mathcal V$ having non-zero coordinates; if such a vector $v$ does not exist, then $V$ is effectively an isometry $V\colon \C^k \to \C^m$ with $m < n$ and the bound in the statement can be improved to $\alpha(X_{\cdot, B}) \leq m-r$. We have
        \begin{equation*}
          d_i:=\diag(c_i) v = \diag(v) c_i.
        \end{equation*}
        Since $\mathbf{1} \notin \img B$, the vectors $c_1, \ldots c_r$ are linearly independent, hence so are the vectors $d_1, \ldots, d_r \in \mathcal C^\perp$, proving that $r \leq n-k$.
    \end{proof}
\end{proposition}
 
Note that this upper bound nicely complements the lower bound in \cref{cor:eqrowsindepbound}: A natural inequality is that $\rk B \leq n - \operatorname{EqRows}(B) + 1$. For loopless $X_{\ccdot, B}$, we have $\mathbf{1} \in \ker B$, which implies that the column-sums of $B$ must be $0$. The inequality thus becomes $\rk B \leq n - \operatorname{EqRows}(B)$, or equivalently $\operatorname{EqRows}(B) \leq n - \rk B$.  Crucially, there exist projectors $B$ that saturate this inequality. Take, for instance, the rank-$1$ projection onto a vector $v$ with $v_1 = \dots = v_{n - 1} = c$ and $v_n = -(n-1)c$. It follows that there exist quantum graphs $X_{\ccdot, B}$ for which our upper and lower bounds coincide.

Finally, note that a naive dimension count for the existence of a complex isometry $V\colon \C^k \to \C^n$ witnessing the independence number of $X_{\cdot, B}$ indicates that the real dimension of the complex Stiefel manifold of isometries $\C^k \to \C^n$ needs to be larger than the total number of constraints in \cref{eq:indep-number-isometry}:
\begin{equation*}
  2nk - k^2 \geq r(k^2-1) \Implies k \leq  \frac{n}{r+1} + \sqrt{\frac{n^2+r^2+r}{(r+1)^2}}.
\end{equation*}
We plot below the difference between the lower bound in \cref{prop:indep-number-LB} and, respectively, the upper bound in \cref{prop:indep-number-UB} (left panel) and the parameter cound in the equation above (right panel).
%We conjecture that for arbitrary (or at least generic) projections $B$, the value of the indepennce number of $X_{\cdot, B}$ is given by the lower bound from \cref{eq:indep-number-LB}.
\begin{center}
\includegraphics[width=0.45\textwidth]{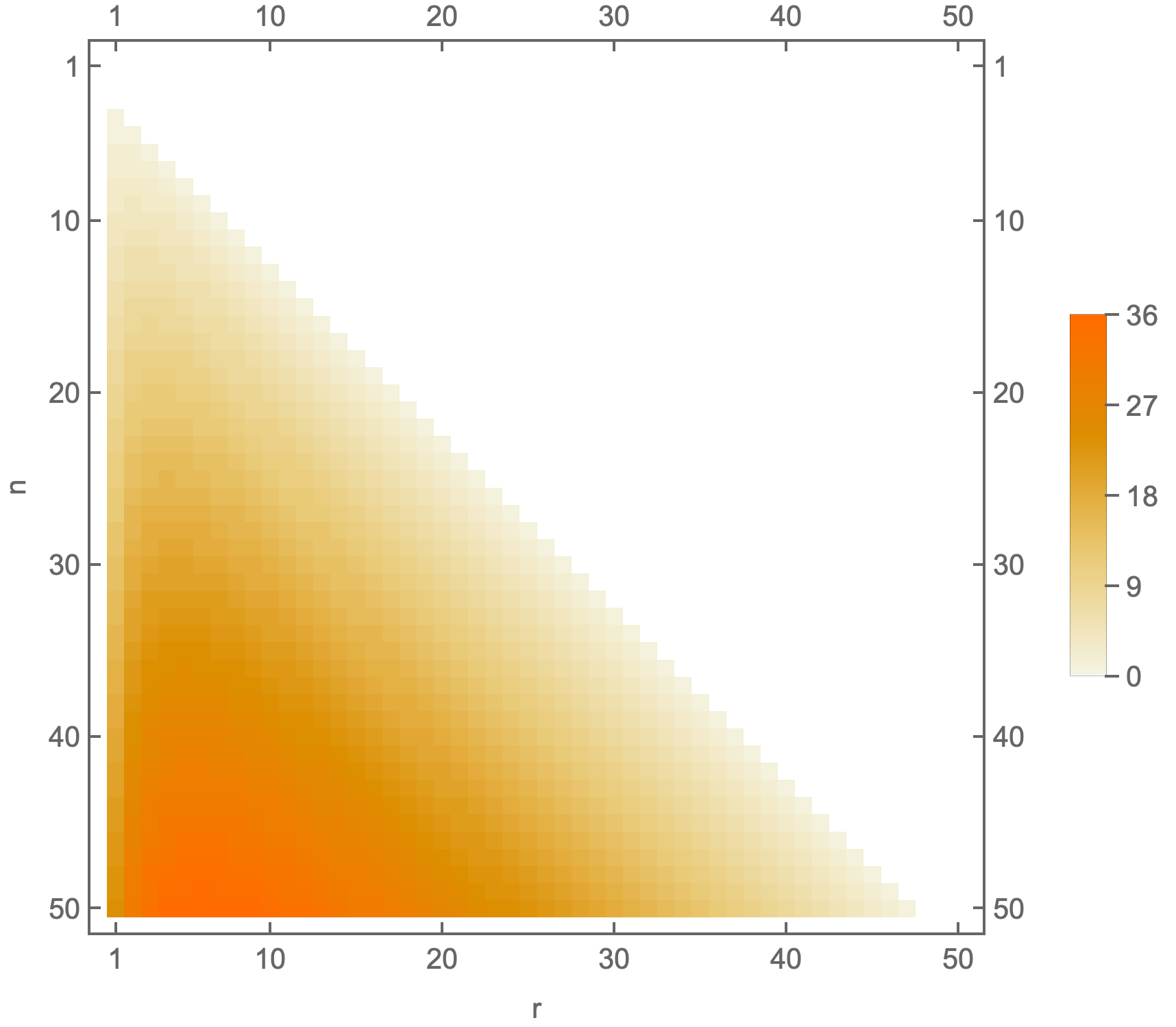} \qquad
\includegraphics[width=0.45\textwidth]{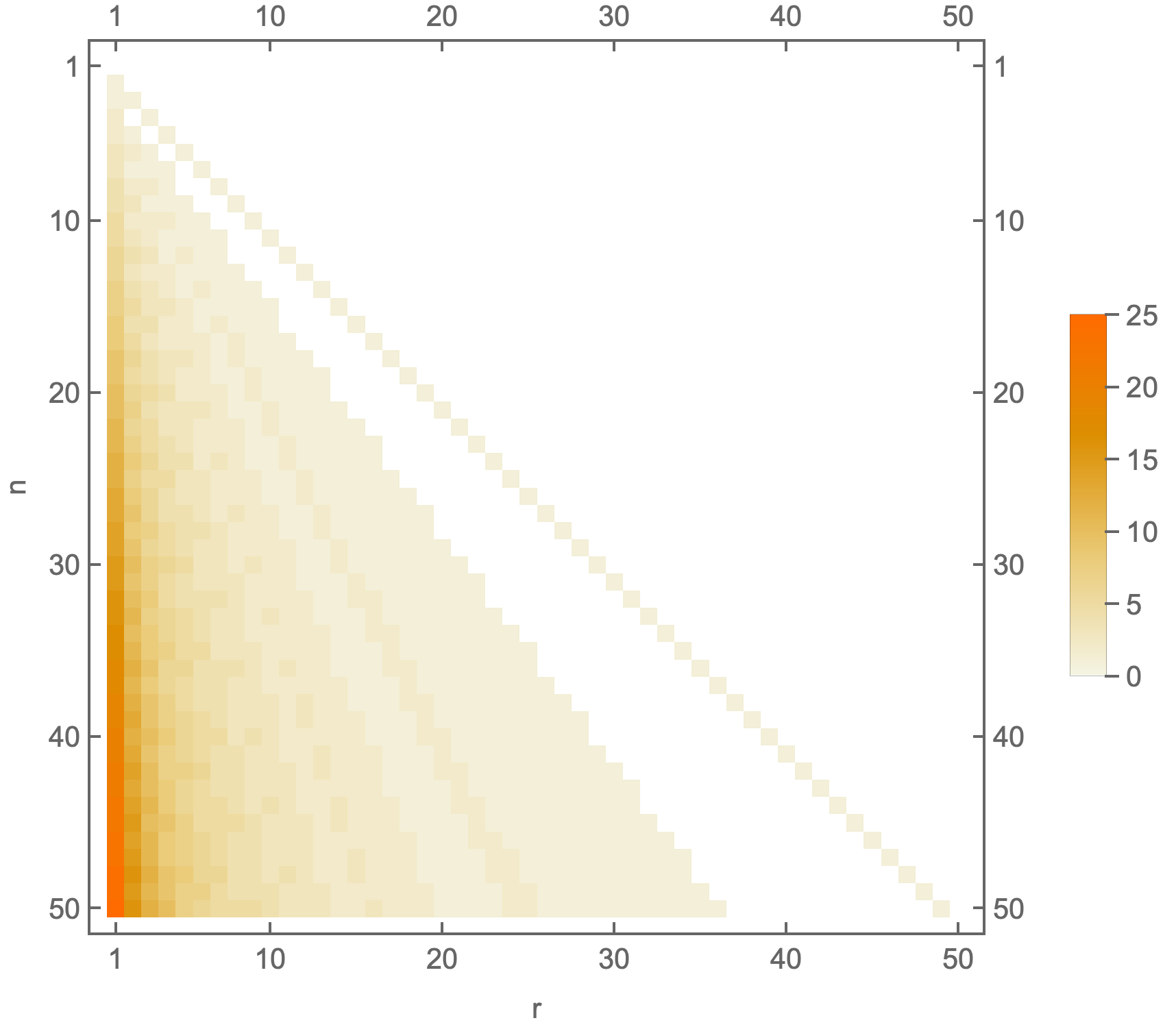}
\end{center}

\subsubsection{The Symmetric and Antisymmetric Cases}

Given our knowledge of $\alpha(X_{A, \ccdot, C})$ and $\alpha(X_{\ccdot, B})$, the treatment of the Symmetric and Antisymmetric quantum graphs is straightforward. The independence number of $\Gsym$ can be determined using \cref{prop:indepnumdetbystrange} and the splitting principle, but it also follows from the more elementary proof below.

\begin{proposition}
  It holds that $\alpha(\Gsym) = 1$.
  \begin{proof}
    Every non-empty quantum graph has an independent set of size $1$, so it suffices to show that there exists no independent set of size $\geq 2$. Suppose that $P$ is a projector witnessing an independent set. Then $P \neq 0$, so there exists an $i \in [n]$ such that $P\ket{i} \neq 0$. Now the operator space corresponding to $\Gsym$ plus $\C\identity$ is spanned by $\{\ket{i}\bra{j} + \ket{j}\bra{i} \mid i, j \in [n]\}$, so $P$ in particular satisfies
    \begin{equation*}
      P\ket{i}\bra{i}P = c_i P.
    \end{equation*}
    Since $P\ket{i} \neq 0$, the left-hand side has rank $1$, which implies that $c_i \neq 0$ and $\rk P = 1$.
  \end{proof}
\end{proposition}

The case for the Antisymmetric quantum graph is again an immediate corollary of \cref{prop:indepnumdetbystrange}.

\begin{proposition}
  It holds that $\alpha(\Gasym) = 1$.
  \begin{proof}
    $\Gasym$ is of the form $X_{A, \ccdot, C}$ so \cref{prop:indepnumdetbystrange} applies. By \cref{prop:gsymgasymstrange} the strange graph corresponding to $\Gasym$ is a complete graph, which has independence number $1$.
  \end{proof}
\end{proposition}

\subsubsection{The Complete and Empty Graphs}

We conclude with the independence numbers of the complete and the empty quantum graphs.

\begin{proposition}
  We have $\alpha(\compl{K_n}) = n$ and $\alpha(K_n) = 1$.
  \begin{proof}
    The operator space corresponding to $\compl{K_n}$ is the trivial space, so any projector $P$ satisfies $PSP \subseteq P$, in particular this holds for $P = \identity$. The complete graph satisfies $K_n = X_{\tfrac{nJ - I}{n}, \tfrac{nI - J}{n}}$, so its corresponding strange graph is the complete graph with only classical edges. \cref{prop:xabcindepsetjointbounds} thus implies that $\alpha(K_n) \leq 1$. The trivial lower bound $\alpha(K_n) \geq 1$ then establishes the equality.
  \end{proof}
\end{proposition}

%% file: content/graphtheory/cliques.tex
Finally, we want to turn our attention to cliques. We will find it more convenient to phrase the definition of cliques in terms of isometries instead of projectors. As such, a $k$-clique of a loopless quantum graph with operator space $S$ is given by an isometry $V\colon \C^k \to \C^n$ such that $V\adjoint (S \oplus \C\identity)V = M_k$. This rephrasing suggests the following strategy to prove the (non-)existence of $k$-cliques. Let $P_S \colon M_n \to M_n$ be the projector onto $S$. Then an isometry $V\colon \C^k \to \C^n$ gives a $k$-clique if and only if the map
\begin{equation*}
  L_V \colon M_n \oplus \C \to M_k, (x, \alpha) \mapsto V\adjoint(P_S(x) + \alpha \identity)V
\end{equation*}
has full image. Since $(\img L_V)^\bot = \ker L_V\adjoint$, this is happens if and only if $L_V\adjoint$  has trivial kernel. Now we have
\begin{equation*}
  \lrangle{(x, \alpha), L_V\adjoint(y)} = \lrangle{L_V(x, \alpha), y} = \lrangle{V\adjoint P_S(x)V, y} + \overline{\alpha}\lrangle{\identity, y} = \lrangle{x, P_S(V y V\adjoint)} + \overline{\alpha} \Tr y
\end{equation*}
for all $x \in M_n$, $\alpha \in \C$. This implies that an arbitrary $y \in M_k$ is in the kernel of $L_V\adjoint$ if and only if $P_S(V y V\adjoint) = 0$ and $\Tr y = 0$. To find a $k$-clique it is thus necessary and sufficient to find an isometry $V \colon \C^k \to \C^n$ such that there are no traceless $y \in M_k$ in the kernel of $x \mapsto P_S(V x V\adjoint)$.

\begin{proposition}\label{prop:cliquecond}
  Given a quantum graph with operator space $S$, the clique number $\omega$ is given by
  $$\omega = \max \{ k \mid \exists V \colon \C^k \to \C^n \text{ isometry such that } \big[\Tr y = 0 \text{ and } P_S(V y V\adjoint) = 0\big] \implies y = 0\}.$$
\end{proposition}

\subsubsection{ABC Graphs}
We will not be able to prove bounds for $X_{A, B, C}$ in terms of $A, B, C$, but we are able to extend the splitting principle proved in the case of independent sets (see \cref{prop:xabcindepsetjointbounds}) to cliques.

\begin{proposition}\label{prop:clique-splitting}
  Let $X_{A, B, C}$ be a quantum graph. Then we have $\omega(X_{A, B, C}) \geq \max \{\omega(X_{A, \ccdot, C}), \omega(X_{\ccdot, B})\}$.
  \begin{proof}
    Denote the operator spaces of $X_{A, B, C}$, $X_{A, \ccdot, C}$, and $X_{\ccdot, B}$ by $S_{ABC}$, $S_{AC}$, and $S_B$ respectively. Recall from the proof of \cref{prop:xabc-qgraph-char} that the projector onto $S_{ABC}$ decomposes as
    \begin{equation*}
      (X_{A, B, C})^R = X_{B, A, C} = B \oplus \bigoplus_{i < j} \begin{pmatrix}
        A_{ij} & C_{ij} \\
        C_{ji} & A_{ji}
      \end{pmatrix}.
    \end{equation*}
    The $B$-summand is the projector onto $S_B$, while the remaining summands form the projector onto $S_{AC}$. Since the image of a direct sum is the direct sum of the images, we have $S_{ABC} = S_{AC} \oplus S_B$, and in particular $S_{ABC} \supseteq S_{AC}$ and $S_{ABC} \supseteq S_B$. Now suppose $V\colon \C^k \to \C^n$ witnesses a $k$-clique of $X_{A, \ccdot, C}$, i.e.\ $V\adjoint(S_{AC} \oplus \C\identity)V = M_k$. Since $S_{ABC} \supseteq S_{AC}$, we have
    \begin{equation*}
      M_k = V\adjoint(S_{AC} \oplus \C\identity)V \subseteq V\adjoint(S_{ABC} \oplus \C\identity)V \subseteq M_k,
    \end{equation*}
    so $V$ also witnesses a $k$-clique of $X_{A, B, C}$. The same argument with $S_B$ in place of $S_{AC}$ yields the result. \qedhere
  \end{proof}
\end{proposition}

\subsubsection{AB Graphs}

Moving on to AB graphs, we start with the case of $X_{\ccdot, B}$. It turns out that we can upper bound $\omega(X_{\ccdot, B})$ in terms of $\rk B$.

\begin{proposition}\label{prop:clique-bound-XB}
  It holds that $\omega(X_{\ccdot, B}) \leq \sqrt{\rk B + 1}$.
  \begin{proof}
    Recall that the projector $P_S$ onto the operator space associated with $X_{\ccdot, B}$ is given by
    \begin{equation*}
       \input{diagrams/xb-proj.tikz}.
    \end{equation*}
    We want to find out under which conditions on an isometry $V\colon \C^k \to \C^n$ there exists a traceless $y \in M_k$ in the kernel of $P_S(V y V\adjoint)$. We have
    \begin{equation*}
      \input{diagrams/proj-xb-kernel-cond-1.tikz} = 0 \quad\quad\IFF\quad \input{diagrams/proj-xb-kernel-cond-2.tikz} = 0 \quad\quad\IFF\quad\quad \input{diagrams/proj-xb-kernel-cond-3.tikz}~ = 0,
    \end{equation*}
    where $d_y \coloneqq \diag V y V\adjoint$. Let $T \coloneqq \{ \diag V y V\adjoint \mid y \in M_k, \Tr y = 0\}$. The question then becomes under which conditions we can find an isometry $V$ such that
    (i) $\diag VyV\adjoint \neq 0$
    for all non-zero traceless $y \in M_k$ and (ii) $T \cap \ker B = \{0\}$.

    Note that condition (ii) is satisfied if and only if for every basis $t_1, \dots, t_\ell$ of $T$ and $b_1, \dots, b_{\ell'}$ of $\ker B$ the set $\{t_1, \dots, t_\ell, b_1, \dots, b_{\ell'}\}$ is linearly independent. It is thus necessary that $\dim T \leq n - \dim \ker B = \rk B$. Condition (i) on the other hand is equivalent to the map $\Phi_V\colon \mathfrak{sl}_k(\C) \to \C^n, x \mapsto \diag V x V\adjoint$ being injective. It follows that $\dim T = \rk \Phi_V = \dim \mathfrak{sl}_k(\C) = k^2 - 1$, and thus $k \leq \sqrt{\rk B + 1}$.
  \end{proof}
\end{proposition}

It turns out that this upper bound can generally be achieved, and this question is closely related to the existence of specific \emph{informationally complete POVMs} (IC-POVMs, see e.g.~\cite[Section~3.3]{heinosaari2011mathematical}). The easiest case is $B = I - J/n$.

\begin{proposition}
  Let $B = I - J/n$. Then $\omega(X_{\ccdot, B}) = \lfloor \sqrt{n} \rfloor$.
\end{proposition}
\begin{proof}
  We have $\ker B = \lspan\{\1\}$, where $\1 = (1, \dots, 1)\transpose \in \C^n$. The upper bound on the clique number follows from \cref{prop:clique-bound-XB}. For the lower bound, we shall use the same proof strategy as in \cref{prop:clique-bound-XB}: It suffices to show that for $k=\lfloor \sqrt{n} \rfloor$ there exists an isometry $V\colon \C^k \to \C^n$ satisfying the following conditions:
  \begin{enumerate}[label=(\roman*)]
    \item $\diag VyV\adjoint \neq 0$ for all non-zero traceless $y \in M_k$
    \item $\{ \diag V y V\adjoint \mid y \in M_k, \Tr y = 0\} \cap \ker B = \{0\}$.
  \end{enumerate}

  We first observe that condition~(ii) is automatically satisfied whenever condition~(i) holds. Indeed, let $y \in M_k$ be traceless and let $v_1, \dots, v_n \in \C^k$ denote the rows of $V$, so that $(\diag VyV\adjoint)_i = v_i\adjoint y v_i$. Then
  \begin{equation*}
    \sum_{i=1}^n (\diag VyV\adjoint)_i = \sum_{i=1}^n v_i\adjoint y v_i = \Tr\left(y \sum_{i=1}^n v_iv_i\adjoint \right) = \Tr(y) = 0,
  \end{equation*}
  where we used that $\sum_i v_iv_i\adjoint = I_k$ since $V$ is an isometry. In other words, the image of $\Phi_V$ on traceless matrices is contained in the hyperplane $\1^\bot = (\ker B)^\bot$. Thus for any non-zero traceless $y$ satisfying $\diag VyV\adjoint \neq 0$, the vector $\diag VyV\adjoint$ cannot lie in $\ker B = \lspan\{\1\}$, and condition~(ii) is automatic.
  The problem therefore reduces to finding the largest~$k$ for which condition~(i) can be satisfied: The map $\Phi_V \colon \mathfrak{sl}_k(\C) \to \C^n$, $x \mapsto \diag VxV\adjoint$, must be injective.

  It remains to show that for $k = \lfloor \sqrt{n} \rfloor$, an isometry satisfying condition~(i) exists. Set $m = k^2 \leq n$. The construction proceeds in two steps: We first build only $m = k^2$ vectors, which is the number needed to span $M_k$, and we add the remaining $n - m$ vectors at the very end. We will construct $n$ vectors $v_1, \dots, v_n \in \C^k$ with $\sum_i v_iv_i\adjoint = I_k$ whose rank-one operators $v_iv_i\adjoint$ span~$M_k$. Note that this is the defining property of an IC-POVM. Consider the standard basis $\ket{1}, \dots, \ket{k}$ of $\C^k$, and define the following $m = k^2$ vectors:
  \begin{enumerate}[(a)]
    \item The $k$ basis vectors $\ket{1}, \dots, \ket{k}$.
    \item For each pair $1 \leq s < t \leq k$, the vector $(\ket{s} + \ket{t})/\sqrt{2}$.
    \item For each pair $1 \leq s < t \leq k$, the vector $(\ket{s} + i \ket{t})/\sqrt{2}$.
  \end{enumerate}
  This gives $k + 2\binom{k}{2} = k^2$ vectors. Their rank-one outer products span~$M_k$: The vectors in~(a) yield the diagonal matrix units $\ket{s}\bra{s}$; the vectors in~(b) yield operators whose off-diagonal part is proportional to $\ket{s}\bra{t} + \ket{t}\bra{s}$ (recovering the real parts of off-diagonal matrix units); and the vectors in~(c) yield operators whose off-diagonal part is proportional to $i(\ket{s}\bra{t} - \ket{t}\bra{s})$ (recovering the imaginary parts). Together these span all of $M_k$.
  Denote these $m$ vectors by $w_1, \dots, w_m$ and let
  \begin{equation*}
   F \coloneqq \sum_{l=1}^m w_lw_l\adjoint.
  \end{equation*}
   Since the operators $w_lw_l\adjoint$ span $M_k$, the matrix $F$ is positive definite. Set $u_l \coloneqq F^{-1/2} w_l$ for each $l$. Then
  \begin{equation*}
   \sum_{l=1}^m u_lu_l\adjoint = F^{-1/2}\left(\sum_l w_lw_l\adjoint\right)F^{-1/2} = I_k.
  \end{equation*}
  Moreover, the operators $u_lu_l\adjoint = F^{-1/2}w_lw_l\adjoint F^{-1/2}$ still span $M_k$, since conjugation by the invertible matrix $F^{-1/2}$ preserves linear independence.

  If $m = n$, we set $v_l = u_l$ and are done. If $m < n$, we extend the collection to $n$ vectors as follows. Choose any $u_1 \neq 0$ and split it into $n - m + 1$ copies: Set $v_l = u_1 / \sqrt{n - m + 1}$ for $l = 1, \dots, n - m + 1$, and $v_l = u_{l - n + m}$ for $l = n - m + 2, \dots, n$. These $n$ vectors still satisfy $\sum_{i=1}^n v_iv_i\adjoint = I_k$, and the set $\{v_iv_i\adjoint\}$ still spans $M_k$ since it contains all the original $u_lu_l\adjoint$ (up to rescaling $u_1u_1\adjoint$, which remains non-zero).
  We thus obtain an isometry $V\colon \C^k \to \C^n$ with having the vectors $v_1, \dots, v_n \in \C^k$ as rows:
  \begin{equation*}
    V = \sum_{i=1}^n \ket i \bra{v_i}.
  \end{equation*}
  For any traceless $y \in M_k$,
  \begin{equation*}
    \Phi_V(y)_i = v_i\adjoint y v_i = \Tr (v_iv_i\adjoint \cdot y) = \Tr \left((v_iv_i\adjoint - \frac{\|v_i\|^2}{k}I_k) \cdot y\right),
  \end{equation*}
  where the last equality uses $\Tr(y) = 0$. The centred operators $v_iv_i\adjoint - \|v_i\|^2 I_k/k$ are traceless, and they span all of $\mathfrak{sl}_k(\C)$ since the $v_iv_i\adjoint$ span $M_k$. If $\Phi_V(y) = 0$, then $y$ is orthogonal (in the Hilbert--Schmidt inner product) to every $v_iv_i\adjoint - \|v_i\|^2 I_k/k$, hence to all of $\mathfrak{sl}_k(\C)$. Since $y$ is itself traceless, this forces $y = 0$. Thus $\Phi_V$ is injective, establishing condition~(i).
\end{proof}

The case for $X_{A, \ccdot}$ is curious. Given our previous results, we would expect that $\omega(X_{A, \ccdot}) = \omega(A)$. However, this is not the case. Weaver had already noticed this discrepancy in \cite{weaver2017quantum}: Recall that he works with operator spaces that correspond to quantum graphs with loops at every vertex. One such operator space is $D_n$, the space of all diagonal matrices. He proves the following.

\begin{proposition}[Proposition 2.2 in \cite{weaver2017quantum}]
  If $n \geq k^2 + k - 1$ then $D_n$ has a $k$-clique.
\end{proposition}

Consider the empty graph with loops at every vertex. Just as for loopless graphs, we can consider the quantum graph $X_{A, \ccdot}$, which in this case is equal to $X_{A, \identity}$ (note that this quantum graph has loops since $\mathbf{1} \not\in \ker B$). The corresponding operator space is given precisely by $\img X^R_{A, \ccdot} = D_n$. Weaver's result can thus be interpreted as the empty graph on $n$ vertices having a $k$-clique if $n$ is large enough. The problem persists when restricting to loopless quantum graphs.

\begin{proposition}\label{prop:xacliquelta}
  Let $A$ be the adjacency matrix of the complete classical graph. Then $\omega(X_{A, \ccdot}) < \omega(A)$.
  \begin{proof}
    The projector onto the operator space corresponding to $X_{A, \ccdot}$ is given by
    \begin{equation*}
      \input{diagrams/xkn-proj.tikz}
    \end{equation*}
    The operator space is thus given by
    \begin{equation*}
      S = \lspan\{\ket{i}\bra{j} \mid i \neq j \in [n]\},
    \end{equation*}
    the space of all matrices with zero diagonal. We have $\omega(A) = n$. If there were an $n$-clique, there would be an isometry $V\colon \C^n \to \C^n$ such that $V(S \oplus \C\identity)V\adjoint = M_n$. This is impossible, since $\dim M_n = n^2 > n^2 - n + 1 = \dim S \oplus \C\identity$.
  \end{proof}
\end{proposition}

There are also loopless graphs for which $\omega(X_{A, \ccdot}) > \omega(A)$. In fact, the former can be arbitrarily larger.

\begin{proposition}\label{prop:cliquenumclassbipart}
  Let $A$ be the adjacency matrix of the classical complete bipartite graph. Then $\omega(A) = 2$ and $\omega(X_{A, \ccdot}) \geq n/2$.
  \begin{proof}
    It is straightforward to see that the complete bipartite graph has clique number $2$, since any larger clique must contain an odd cycle. The operator space $S$ corresponding to $\omega(X_{A, \ccdot})$ is spanned by the non-zero coordinates of the matrix
    \begin{equation*}
      A = \begin{pmatrix}
        0 & J\\
        J & 0
      \end{pmatrix}.
    \end{equation*}
    Choosing the isometry $V\colon \C^{n/2} \to \C^n$ as
    \begin{equation*}
      V = \frac{1}{\sqrt 2}\begin{pmatrix}
        I\\
        I
      \end{pmatrix},
    \end{equation*}
    we get $V\adjoint S V = M_k$ and thus in particular $V\adjoint (S \oplus \C\identity) V = M_k$.
  \end{proof}
\end{proposition}

In general, we can slightly lower the trivial upper bound on $\omega(X_{A, \ccdot})$ from $n$ to $n-1$, but not any further.

\begin{proposition}\label{prop:classcomplomega}
  For all $A$, we have $\omega(X_{A, \ccdot}) \leq n - 1$. If $A$ is the adjacency matrix of the complete classical graph, then $\omega(X_{A, \ccdot}) = n - 1$.
  \begin{proof}
    Since the operator space $S$ associated to $X_{A, \ccdot}$ is a coordinate subspace of $M_n$, and $A$ has zero diagonal, we have $\dim (S \oplus \C\identity) \leq n^2 - n + 1$. It follows that for any isometry $V$, $\dim V\adjoint(S \oplus \C\identity)V \leq n^2 - n + 1 < n^2$, so $X_{A, \ccdot}$ cannot have an $n$-clique.

    We now claim that if $S$ is maximal subject to these constraints, that is
    \begin{equation*}
      S = \lspan\{\ket{i}\bra{j} \mid i, j \in [n], i \neq j\},
    \end{equation*}
    then there exists an isometry $V\colon \C^{n - 1} \to \C^n$ such that for all $y \in M_{n-1}$ with $\Tr y = 0$ we have
    \begin{equation}\label{eq:cliquecond}
      P_S(VyV\adjoint) = 0 \quad\Implies\quad y = 0.
    \end{equation}
    \cref{prop:cliquecond} then implies that if $A$ is the complete classical graph then $X_{A, \ccdot}$ has a clique of size $n - 1$.

    Let $y \in M_{n - 1}$ be arbitrary and let $Q \coloneq VyV\adjoint$. Then $y = V\adjoint Q V$ and thus $Q = VV\adjoint QVV\adjoint = RQR$, where $R$ is the projector onto $\img V$. Now $R$ has rank $n - 1$, so it is of the form
    \begin{equation*}
      R = \identity - uu\adjoint
    \end{equation*}
    for some unit-norm $u \in \C^n$. Choose $u = 1/\sqrt{n} \mathbf{1}$. Suppose then that $P_S(V y V\adjoint) = P_S(Q) = 0$. Then $Q \in S^\bot = D_n$. Then $\diag Q = Q\mathbf{1} = RQR\mathbf{1} = 0$. Since $Q$ is diagonal, this implies that $Q = 0$ and thus $y = V\adjoint Q V = 0$ as required.
  \end{proof}
\end{proposition}

Despite \cref{prop:xacliquelta}, there is some hope for lower bounds on $\omega(X_{A, \ccdot})$ in terms of $\omega(A)$. It turns out that the lower bound just barely fails. Concretely, we can prove the following.

\begin{proposition}
  It holds that $\omega(X_{A, \ccdot}) \geq \omega(A) - 1$.
  \begin{proof}
    Let $k = \omega(A)$. Then the classical graph $A$ has a clique of size $k$. Concretely, there exists a set $C \subseteq [n]$, such that $A_{ij} = 1$ for all $i \neq j \in C$. For the operator space $S$ of $X_{A, \ccdot}$, this implies that
    \begin{equation*}
      \lspan\{\ket{i}\bra{j} \mid i, j \in C, i \neq j\} \subseteq \lspan S.
    \end{equation*}
    Enumerate the elements of $C$ as $c_1, \dots, c_k$, and define an isometry $W$ by
    \begin{equation*}
      W = \sum_{i = 1}^k \ket{c_i}\bra{i}.
    \end{equation*}
    Then $W\adjoint S W = \lspan\{\ket{i}\bra{j} \mid i \neq j \in [k]\} \subseteq M_k$. But this is the operator space of $X_{K_k, \ccdot}$! By \cref{prop:classcomplomega}, we know that there exists an isometry $V\colon \C^{k - 1} \to \C^k$ such that $V\adjoint(W\adjoint S W \oplus \C\identity_k)V = M_{k - 1}$. We also have $W\adjoint \identity_n W = \identity_k$, so
    \begin{equation*}
      W\adjoint (S \oplus \C\identity_n) W = W\adjoint S W \oplus \C\identity_k.
    \end{equation*}
    It follows that
    \begin{equation*}
      V\adjoint W\adjoint (S \oplus \C\identity_n) W V = V\adjoint(W\adjoint S W \oplus \C\identity_k) V = M_{k - 1},
    \end{equation*}
    so $WV \colon \C^{k-1} \to \C^n$ witnesses a clique of size $k - 1$.
  \end{proof}
\end{proposition}

One may wonder why $\omega(A)$ fails to be a lower bound by such a close margin. We can shed some light on this by considering the role of loops, and the different operator spaces that may be associated to a classical graph. Recall that Weaver's original definition of a $k$-clique requires the existence of a rank-$k$ projector such that $PSP = PM_nP$, where $S$ is the operator space associated to the quantum graph. As Weaver mentions in \cite{weaver2021quantum}, the operator space $S$ should be seen as the edge relation of the quantum graph (in analogy to the edge relation $E(G) \subseteq V(G) \times V(G)$ of a classical graph $G$), while the ``sandwiching'' operation $PSP$ should be seen as a restriction of a relation to some subset of the underlying set. In terms of classical graphs, the condition $PSP = PM_nP$ may thus be interpreted as the existence of an induced subgraph that is isomorphic to an induced subgraph of the complete graph with loops at all vertices (the largest possible relation).

This seems to suggest that intuitively, we need our quantum graph to have loops in order to find cliques.\footnote{This is not actually true, as the proof of \cref{prop:cliquenumclassbipart} shows. This is because the mental model of viewing $PSP$ as the restriction of a relation to some set is already not rigorous. Nevertheless, it will suffice for the point we are making.} However, the quantum graphs that we are considering do not have loops.
To address this issue, we have modified the definition of cliques to require the existence of a projection $P$ with $P(S \oplus \C\identity)P = PM_nP$. Adding the identity to the operator space of a loopless quantum graph yields a quantum graph with loops at every vertex, which then allows us to use Weaver's definition. This way of adding loops is very natural, because the resulting operator space is the smallest superspace of $S$ that has loops at every vertex.

The problem arises when considering quantum graphs of the form $X_{A, \ccdot}$. We have seen in \cref{xa-qgraph-char} that $X_{A, \ccdot}$ is loopless if and only if $A$ is loopless, and adding the identity to the corresponding operator space turns it into a quantum graph with loops at every vertex as desired. However, there is a second natural way to add these loops. Namely, we can add the loops to the classical graph $A$, letting $A^\circ = A + \identity$, and then consider $X_{A^\circ, \ccdot}$. This does not result in the same operator space. The former approach yields
\begin{equation*}
  S' = S \oplus \C\identity = \lspan (\{\ket{i}\bra{j} \mid A_{ij} \neq 0 \} \cup \{ \identity \})
\end{equation*}
while the latter approach yields
\begin{equation*}
  S^\circ = \lspan \{\ket{i}\bra{j} \mid A_{ij} \neq 0 \text{ or } i = j\}.
\end{equation*}
Concretely, $S^\circ$ contains the diagonal elements $\ket{i}\bra{i}$ for $i \in [n]$, while $S'$ does not. Otherwise $S'$ and $S^\circ$ are the same. We thus expect $S^\circ$ to have larger cliques. This suggests that $S^\circ$ instead of $S'$ might be the more natural choice to upper bound $\omega(A)$ .

Conveniently, our framework allows us to recover $S^\circ$ exactly. Let us consider some arbitrary loopless classical graph $A$ and consider $X_{A + \diag(1-1/n), B}$ for $B = I - J/n$.\footnote{This choice of $B$ has non-zero diagonal, so we need to add the diagonal to $A$ as well. This is merely a notational quirk, since the diagonal is removed again in the resulting linear map.} Then the projector $X^R_{A + \diag(1 - 1/n), B}$ onto the operator space is given by
\begin{equation*}
  \input{diagrams/xabdiag-proj.tikz}\enspace.
\end{equation*}
Adding the identity as in our clique definition gets rid of the last term,
\begin{equation*}
  \input{diagrams/xabdiag-proj-refl.tikz}\enspace.
\end{equation*}
It is not hard to see that this projector precisely has image $S^\circ$: The first term provides $S$, while the second term provides the diagonal elements. And indeed, we find that its clique number is lower bounded by $\omega(A)$.

\begin{proposition}\label{prop:adiag-clique-lower-bound}
  If $A$ is the adjacency matrix of a graph $G$, then $\omega(X_{A + \mathrm{diag}(1-1/n), I-J/n}) \geq \omega(G)$.
  \begin{proof}
    Let $k = \omega(G)$ and let $C = \{c_1, \dots, c_k\} \subseteq [n]$ be a $k$-clique of $G$, so that $A_{c_sc_t} = 1$ for all $s \neq t$. Define the isometry $V\colon \C^k \to \C^n$ by $V\ket{s} = \ket{c_s}$. Recall from the discussion above that $S \oplus \C\identity = S^\circ = \lspan\{\ket{i}\bra{j} \mid A_{ij} \neq 0 \text{ or } i = j\}$. For all $s, t \in [k]$, the matrix unit $\ket{c_s}\bra{c_t}$ lies in $S^\circ$: If $s \neq t$, this follows from $A_{c_sc_t} = 1$, and if $s = t$ from the diagonal condition. Since $V\adjoint \ket{c_s}\bra{c_t} V = \ket{s}\bra{t}$, it follows that $V\adjoint S^\circ V = M_k$, so $V$ witnesses a $k$-clique of $X_{A + \mathrm{diag}(1-1/n), I-J/n}$. \qedhere
  \end{proof}
\end{proposition}

\subsubsection{The Symmetric and Antisymmetric Cases}

For the symmetric and antisymmetric quantum graphs, we can determine $\omega$ exactly.

\begin{proposition}
  The clique number of the Symmetric and the Antisymmetric quantum graphs on $n$ vertices is
  \begin{equation*}
    \omega(\Gsym) = \omega(\Gasym) = \left\lceil \frac n 2 \right\rceil.
  \end{equation*}
\end{proposition}
\begin{proof}
  Let $P_S^a$ and $P_S^s$ be the projection onto the operator spaces of $\Gasym$ and $\Gsym$ respectively.  By \cref{prop:cliquecond}, we want to determine the maximum $k$ such that there exists an isometry $V\colon \C^k \to \C^n$ with $P_S^a(VyV\adjoint) \neq 0$ respectively $P_S^s(VyV\adjoint) \neq 0$ for all traceless non-zero matrices $y \in M_k$. Concretely, the conditions become
  \begin{equation*}
    \input{diagrams/clique-cond-asym.tikz}
  \end{equation*}
  and
  \begin{equation*}
    \input{diagrams/clique-cond-sym.tikz}
  \end{equation*}
  We thus want to understand under which conditions the map
  \begin{equation*}
    \input{diagrams/clique-symasym-map.tikz}
  \end{equation*}
  has traceless matrices in its kernel, where $X \in \{\operatorname{Sym}, \operatorname{Asym}\}$. First note that $V \otimes \overline{V}$ maps $y \mapsto VyV\adjoint$ and thus preserves the trace. We need to find a traceless $x \in \img (V \otimes \overline{V}) \cap X^\bot$.

  First consider $X = \operatorname{Sym}$. We have $x \in \img (V \otimes \overline{V})$ if and only if $x = Vy\overline{V}\transpose$.
  This in turn is equivalent to $\img x \subseteq \img V$ and $\img x\transpose \subseteq \img \overline{V}$. Now if $x$ is required to be antisymmetric then $\img x\transpose = -\img x = \img x$.
  We thus need to determine when there exists an antisymmetric $x \in M_n$ with $\img x \subseteq \img V \cap \img \overline{V}$. Let $S_V \coloneqq \img V \cap \img \overline{V}$. If $\dim S_V \geq 2$, such an $x$ exists. Indeed, choose orthogonal $\psi_1, \psi_2 \in S_V$ and let
  \begin{equation*}
    x = \psi_1\psi_2\transpose - \psi_2\psi_1\transpose.
  \end{equation*}
  Then $x$ is antisymmetric (and thus automatically traceless) and $\img x \subseteq \img S_V$ as desired. On the other hand, if $\dim S_V < 2$ then no such non-zero map exists because there are no rank $1$ asymmetric matrices.

  The story is analogous for $X = \operatorname{Asym}$. We are looking for a traceless symmetric $x \in \img(V \otimes \overline{V})$. Since $\img x\transpose = \img x$, we again require $\img x \subseteq \img V \cap \img \overline{V} = S_V$. For $\dim S_V \geq 2$, we may choose
  \begin{equation*}
    x = \psi_1\psi_2\adjoint + \overline{\psi_2}\psi_1\transpose,
  \end{equation*}
  which is symmetric and satisfies the image condition since $S_V$ is closed under conjugation. Moreover, $x$ is traceless: We have
  \begin{align*}
    \Tr \psi_1\psi_2\adjoint + \overline{\psi_2}\psi_1\transpose &= \Tr \psi_1\psi_2\adjoint + \Tr \overline{\psi_2}\psi_1\transpose \\
                                                                 &= \lrangle{\psi_2, \psi_1} + \lrangle{\overline{\psi_1}, \overline{\psi_2}}\\
                                                                 &= 2\lrangle{\psi_2, \psi_1} = 0.
  \end{align*}
  Conversely, if $\dim S_V < 2$ no such non-zero $x$ exists. If $S_V = \C\psi$, the only choice is
  \begin{equation*}
    x = \psi\psi\transpose
  \end{equation*}
  which has trace $\lrangle{\overline{\psi}, \psi} \neq 0$.

  We have shown that $V\colon \C^k \to \C^n$ witnesses a $k$-clique if and only if $\dim S_V < 2$. We first show that this is unsatisfiable for $k > \lceil n / 2 \rceil$. It holds that
  \begin{equation*}
    \dim S_V = \dim (\img V \cap \img \overline{V}) = \dim \img V + \dim \img \overline{V} - \dim (\img V + \img \overline{V}).
  \end{equation*}
  Since $V$ is an isometry, we have $\rk V = \rk \overline{V} = k$. We thus have $\dim S_V < 2$ if and only if $2k - \dim(\img V \oplus \img \overline{V}) < 2$, or equivalently
  \begin{equation}\label{eq:cliquecondk}
      2k - 2 < \dim (\img V \oplus \img \overline{V}) \leq n.
  \end{equation}
  This is only satisfiable if
  \begin{equation*}
    k \leq \frac{n + 1}{2}. % if you think this is wrong, note that we have \leq instead of <
  \end{equation*}
  For $n$ odd, we have
  \begin{equation*}
    \frac{n + 1}{2} = \left\lceil \frac{n}{2} \right\rceil
  \end{equation*}
  while for $n$ even and $k \in \N$ we have
  \begin{equation*}
    k \leq \frac{n + 1}{2} \qquad\IFF\qquad k \leq \frac{n}{2} = \left\lceil \frac n 2 \right\rceil,
  \end{equation*}
  which yields the desired bound. Conversely, let $k = \lceil n / 2 \rceil$ and choose a random isometry $V \colon \C^k \to \C^n$. Then $\dim(\img V + \img \overline{V}) = n - \dim(\img V \cap \img \overline{V}) = n$ by the lemma below. This completes the proof.
\end{proof}

\begin{lemma}
  Let $1 \le k \le n/2$. If $E$ is a Haar-distributed random $k$-dimensional complex subspace of $\mathbb{C}^n$, then
  \begin{equation*}
    E \cap \overline{E} = \{0\}
  \end{equation*}
  almost surely.
  \begin{proof}
    We generate a random $k$-dimensional subspace $E \subseteq \mathbb{C}^n$ as the range of a complex Ginibre matrix. Let $A \in M_{n \times k}(\mathbb{C})$ be a matrix whose entries are independent complex Gaussian random variables. With probability one, $A$ has rank $k$, and its column space $E = \operatorname{im}(A)$ is a random element of the complex Grassmannian with the unitarily invariant distribution.

    Suppose that $E \cap \overline{E} \neq \{0\}$. Then there exists $v \neq 0$ such that $v \in E \cap \overline{E}$ and thus there exist vectors $x,y \in \mathbb{C}^k$, not both zero, such that $v = A x = \overline{A} y$. Equivalently, $A x - \overline{A} y = 0$. Define the matrix
    \begin{equation*}
      M(A) =
      \begin{bmatrix}
        A & -\overline{A}
      \end{bmatrix}
      \in M_{n \times 2k}(\mathbb{C}).
    \end{equation*}
    Then the above condition becomes
    \begin{equation*}
    M(A)
    \begin{pmatrix}
      x \\
      y
    \end{pmatrix}
    = 0.
    \end{equation*}
    Thus $E \cap \overline{E} \neq \{0\}$ if and only if $\ker M(A) \neq \{0\}$, or equivalently $\operatorname{rank} M(A) < 2k$.
    Since $k \le n/2$, we have $2k \le n$, so $M(A)$ can have full column rank $2k$. The condition $\operatorname{rank} M(A) < 2k$ is equivalent to the vanishing of all $2k \times 2k$ minors of $M(A)$, which are polynomial functions of the real and imaginary parts of the entries of $A$. Since the entries of $A$ are independent continuous random variables, to show that the probability that these polynomial vanish is zero we need to prove that the polynomials are not identically zero.

    To this end, we shall exhibit a matrix $A$ with $\operatorname{rank} M(A) = 2k$. Consider the deterministic matrix
    \begin{equation*}
      A =
      \begin{bmatrix}
        I_k \\
        i I_k \\
        0
      \end{bmatrix}
      \in M_{n\times k}(\mathbb{C}),
    \end{equation*}
    where the last block has size $(n-2k)\times k$. Then
    \begin{equation*}
      M(A)=
      \begin{bmatrix}
        I_k & -I_k \\
        iI_k & iI_k \\
        0 & 0
      \end{bmatrix}.
    \end{equation*}
    If
    \begin{equation*}
      M(A)
      \begin{pmatrix}
        x\\
        y
      \end{pmatrix}
      =0,
    \end{equation*}
    then $x-y=0$ and $ix+iy=0$ which gives $x=y=0$. Therefore $\ker M(A)=\{0\}$ and $\operatorname{rank} M(A)=2k$ for our choice of $A$, finishing the proof.
  \end{proof}
\end{lemma}

\begin{remark}
  Recall that $\Gsym$ and $\Gasym$ are each others complements. The fact that  $\alpha(\Gsym) = \alpha(\Gasym) = 1$, while $\omega(\Gsym) = \omega(\Gasym) = \lceil n/ 2 \rceil$ shows that the well-known classical correspondence between the independent sets of a graph and cliques of its complement does generally not hold for quantum graphs. This counterintuitive result was also discussed in \cite{spitzer_unification_2026}, where the author defines properties of quantum graphs in terms of vertex subsets. Inter alia, they propose an alternative definition of cliques, which recovers the abovementioned correspondence for Weaver's independent sets.
\end{remark}

\subsubsection{The Complete and Empty Graphs}

We conclude by considering the empty and the complete graph. We have seen that the operator space corresponding to the empty graph is $S_{\compl{K}} = 0$, while the operator space $S_K$ corresponding to the complete graph is the set of all traceless matrices. In the former case, we have $S_{\compl{K}} \oplus \C\identity = \C \identity$. It follows that any isometry $V \colon \C \to \C^n$ satisfies $V\adjoint (S_{\compl{K}} \oplus \C\identity) V = M_1$ and we cannot do any better by dimensionality. On the other hand, we have $S_K \oplus \C\identity = M_n$, so the identity is an isometry $\C^n \to \C^n$ witnessing a clique of size $n$. We can thus state the final result of our graph theoretic investigation.

\begin{proposition}
  We have $\omega(K_n) = n$ and $\omega(\compl{K_n}) = 1$.
\end{proposition}

%% file: diagrams/proj-xb-kernel-cond-3.tikz
\begin{tikzpicture}
	\begin{pgfonlayer}{nodelayer}
		\node [style=map] (0) at (0, 0) {$B$};
		\node [style=none] (1) at (0, 0.75) {};
		\node [style=cket] (2) at (0, -1) {$d_y$};
	\end{pgfonlayer}
	\begin{pgfonlayer}{edgelayer}
		\draw [style=x-wire] (2) to (0);
		\draw [style=x-wire] (0) to (1.center);
	\end{pgfonlayer}
\end{tikzpicture}

%% file: content/conclusion.tex
In this work, we have introduced and systematically studied families of quantum graphs on $M_n$ that are invariant under the action of classical matrix groups. By considering the chain of subgroups $U(n) \supseteq O(n) \supseteq \operatorname{Hyp}(n) \supseteq DO(n) \subseteq DU(n)$, we have shown that progressively smaller symmetry groups give rise to progressively richer classes of quantum graphs, mirroring the classical situation. Our main conceptual contribution is the parametrisation of $DU(n)$- and $DO(n)$-invariant quantum graphs by triples of matrices $(A, B, C)$, building on previous work in quantum information theory \cite{singh2021diagonal, nechita2021graphical}.
This parametrisation reveals a clean separation of roles of the three matrices: $A$ encodes a classical graph (\cref{xa-qgraph-char}), $C$ introduces \emph{strange edges} carrying a phase (\cref{prop:gsymgasymstrange}), and $B$ provides a purely quantum contribution with no classical analogue. We call the resulting classical model the strange graph $\mathfrak{G}(A, C)$. Importantly, we prove a \emph{splitting principle} that decomposes graph-theoretic conditions into independent conditions on the $(A, C)$ and $B$ parts.

To the best of our knowledge, this is the first time that large, parametric families of non-trivial quantum graphs have been exhibited for which standard graph parameters such as the number of connected components, the independence number, the clique number, and the chromatic number, can all be computed or tightly bounded analytically. Our work provides a rich source of examples, counterexamples, and a guiding framework for future study of quantum graph theory.

Several questions remain open. The clique number of general ABC graphs $X_{A,B,C}$ remains undetermined (see \cref{prop:clique-bound-XB,prop:classcomplomega} for the partial results obtained in this work), as do tight bounds on the chromatic number of $X_{A, B}$ in terms of $A$ and $B$. The difficulty of this problem lies in the very complex interaction between the strange graph parts described by the matrices $A$ and $C$, and the purely quantum parts described by the projector $B$.

Another direction concerns the connectedness equivalence established in \cref{prop:connectedequivn3}: while connectedness of $X_{A, B, C}$ and of the strange graph $\mathfrak{G}(A, C)$ agree for $n \geq 3$, the correspondence breaks down at the level of connected components, as shown by \cref{prop:conncompcounterex}. It would be interesting to determine the exact number of connected components of $X_{A, \ccdot, C}$ in terms of $\mathfrak{G}(A, C)$, and in particular to understand the role played by isolated strange edges with phase $\pi$. More broadly, the quantum clique number $\omega(X_{A, \ccdot})$ can differ from the classical clique number $\omega(A)$ in both directions, and the precise relationship between these two quantities remains unclear beyond the bound $\omega(X_{A, \ccdot}) \geq \omega(A) - 1$.

A natural next step is to consider \emph{quantum} symmetries instead. It is known that the quantum symmetries of $M_n$ are given by the projective free unitary quantum group $PU_n^+$ \cite[Corollary 4.1]{banica_symmetries_1999}. It would thus be interesting to consider which classes of quantum graphs are acted on by quantum subgroups of $PU_n^+$, and whether they can also be described in discrete terms.